\documentclass[12pt]{article}
\usepackage[left=2cm,right=2cm,
top=2cm,bottom=2cm,bindingoffset=0cm]{geometry}
\usepackage{amssymb,amsmath}
\usepackage[utf8]{inputenc} 
\usepackage[matrix,arrow,curve]{xy}
\usepackage[final]{graphicx} 
\usepackage{mathrsfs}
\usepackage[colorlinks=true, urlcolor=blue, linkcolor=blue, citecolor=blue, pdfborder={0 0 0}]{hyperref}
\usepackage{yhmath}

\hypersetup{frenchlinks=true}

\newtheorem{theorem}{Theorem}[subsection]

\newtheorem{proposition}[theorem]{Proposition}
\newtheorem{remark}[theorem]{Remark}
\newtheorem{corollary}[theorem]{Corollary}
\newtheorem{definition}[theorem]{Definition}
\newtheorem{example}[theorem]{Example}

\newenvironment{proof}{\par $\triangleleft$}{$\triangleright$}

\begin{document}

\begin{center}

\Large \textbf{Metric and topological freedom for operator sequence spaces}\\[0.5cm]
\small {Norbert Nemesh, Sergei Shteiner}\\[0.5cm]

\end{center}
\thispagestyle{empty}

\begin{abstract} 
In this paper we give description of free and cofree objects in the  category of operator sequence spaces. First we show that this category possess the same duality theory as category of normed spaces, then with the aid of these results we give complete description of metrically and topologically free and cofree objects.
\end{abstract}

\section{Preliminaries}

\subsection{Duality theory for normed spaces}
\begin{definition}[\cite{HelQFA}, 0.0.1, 4.4.1]\label{DefDuality} Let $E$, $F$ and $G$ be normed spaces and $\mathcal{D}:E\times F\to G$ be a bounded bilinear operator, then 
\newline
1) $\mathcal{D}$ is called non-degenerate from the left (right) if the operator 
$$
{}^E\mathcal{D}:E\mapsto\mathcal{B}(F,G):x\mapsto(y\mapsto\mathcal{D}(x,y))\qquad
(\mathcal{D}^F:F\mapsto\mathcal{B}(E,G):y\mapsto(x\mapsto\mathcal{D}(x,y)))
$$ 
is injective
\newline
2) $\mathcal{D}$ is called isometric from the left (right) if ${}^E\mathcal{D}$ ($\mathcal{D}^F$) is isometric 
\newline
3) $\mathcal{D}$ is called a vector duality if it is non degenerate from the left and from the right
\newline
4) if $G=\mathbb{C}$ then vector duality $\mathcal{D}$ is called scalar duality 
\end{definition}
Bilinear functionals of the form
$$
\mathcal{D}_{E,E^*}:E\times E^*\to\mathbb{C}:(x,f)\mapsto f(x)
\qquad
\mathcal{D}_{E^*,E}:E^*\times E\to\mathbb{C}:(f,x)\mapsto f(x)
$$
are called the standard scalar dualities. For all $x\in E$ and $f\in E^*$ we have
$$
\Vert x\Vert=\sup\{|\mathcal{D}_{E,E^*}(x,f)|:f\in B_{E^*}\}
\qquad
\Vert f\Vert=\sup\{|\mathcal{D}_{E,E^*}(x,f)|:x\in  B_E\} 
$$
The first equality is a consequence of Hahn-Banach theorem, the second one is the usual definition of operator norm. Note that $\mathcal{D}_{E,E^*}^E$ is the natural embedding $\iota_E$ into the second dual space. For a given $T\in \mathcal{B}(E,F)$, we have $\mathcal{D}_{F,F^*}(T(x),g)=\mathcal{D}_{E,E^*}(x,T^*(g))$
where $x\in E$ and $g\in F^*$. This is nothing more than the usual definition of adjoint operator.

\begin{definition}\label{DefDConv} Let $\mathcal{D}:E\times F\to G$ be a vector duality between normed spaces $E$, $F$ and $G$. We say that a net $(y_\nu)_{\nu\in N}\subset F$ $\mathcal{D}$-conerges to $y\in F$ if for all $x\in E$ a net $(\mathcal{D}(x,y_\nu-y)_{\nu\in N}$ converges to $0$. Topology generated by this type of convergence we will denote by $\sigma_\mathcal{D}(F,E)$.
\end{definition}

Many types of convergence in functional analysis may be formulated in terms of $\mathcal{D}$-convergence, for example usual weak convergence is nothing more than $\mathcal{D}_{X^*,X}$-convergence. 

For a given $p\in[1,+\infty]\cup\{0\}$ by $p'$ we denote conjugate exponent, i.e. $p'=p/(p-1)$ for $p\in(1,+\infty)$ while $1'=\infty$ and $0'=\infty'=1$. Recall the following standard fact.

\begin{proposition}\label{PrSumDuality}
Let $\{E_\lambda:\lambda\in \Lambda\}$ be a family of normed spaces and $p\in[1,+\infty]\cup\{0\}$, then for the scalar duality
$$
\mathcal{D}:\bigoplus{}_p^0\{E_\lambda:\lambda\in \Lambda\}\times \bigoplus{}_{p'}\{E_\lambda^*:\lambda\in \Lambda\}\to \mathbb{C}: (x,f)\mapsto\sum\limits_{\lambda\in \Lambda} f_\lambda(x_\lambda)
$$
the linear operator $\mathcal{D}^{\bigoplus{}_{p'}\{E_\lambda^*:\lambda\in \Lambda\}}$ is isometric. If $p\neq\infty$, then it is an isometric isomorphism.
\end{proposition}

Similar result holds for $\bigoplus{}_p$-sums.

\subsection{Operators between normed spaces}

\begin{definition}\label{DefNorOpType} Let $ T:E\to F$ be a bounded linear operator between normed spaces $E$ and $F$, then $ T$ is called
\newline
1) contractive, if $\Vert T\Vert\leq 1$
\newline
2) \textit{$c$-topologically injective}, if there exist $c > 0$ such that for all $x \in E$ holds $\Vert x\Vert\leq c\Vert  T(x)\Vert$. If mentioning of constant $c$ will be irrelevant we will simply say that $ T$ is topologically injective.
\newline
3) \textit{(strictly) $c$-topologically surjective}, if for all $c'>c$ and  $y\in F$ there exist $x \in E$ such that $ T(x) = y$ and $\Vert x \Vert < c' \Vert y \Vert$ ($\Vert x \Vert \leq c \Vert y \Vert$). 
If mentioning of constant $c$ will be irrelevant we will simply say that $ T$ is (strictly) topologically injective.
\newline
4) isometric, if it is contractive and $1$-topologically injective
\newline
5) (strictly) coisometric, if it is contractive and (strictly) $1$-topologically surjective.
\end{definition}

Clearly, our definition of isometric operator is equivalent to the usual one.

\begin{proposition}\label{PrEquivDescOfIsomCoisomOp}
Let $E$, $F$ be normed spaces and $T:E\to F$ be bounded linear operator. Then,
\newline
1) $T$ (strictly) $c$-topologically surjective $\Longleftrightarrow$ $T(B_E^\circ)\supset c^{-1}B_F^\circ$ ($T(B_E)\supset c^{-1}B_F$) 
\newline
2) $T$ (strictly) coisometric $\Longleftrightarrow$ $T(B_E^\circ)=B_F^\circ$ ($T(B_E)=B_F$)
\end{proposition}
\begin{proof}
1) Assume $T$ is $c$-topologically surjective. Let $y\in c^{-1}B_F^\circ$, then there exist $k'>1$ such that $k'y\in c^{-1}B_F^\circ$. Define $c'=k'c>c$. By assumption there exist $x\in E$ such that $T(x)=y$ and 
$\Vert x\Vert< c'\Vert y\Vert=c\Vert k'y\Vert<1$. Since $y\in c^{-1}B_F^\circ$ is arbitrary, then $T(B_E^\circ)\supset c^{-1}B_F^\circ$. Conversely, assume that $T(B_E^\circ)\supset c^{-1}B_F^\circ$. Let $y\in F$ and $c'>c$, then 
$\tilde{y}=(c')^{-1}\Vert y\Vert^{-1}y\in c^{-1}B_F^\circ$. By assumption there exist $\tilde{x}\in B_E^\circ$ such that $T(\tilde{x})=\tilde{y}$. In this case for $x:=c'\Vert y\Vert\tilde{x}$ we have 
$\Vert x \Vert=c'\Vert y\Vert\Vert\tilde{x}\Vert< c'\Vert y\Vert$ and $T(x)=c'\Vert\ y\Vert T(\tilde{x})=c'\Vert y\Vert\tilde{y}=y$. Since $y\in F$ and $c'>c$ are arbitrary, we conclude that $T$ is $c$-topologically surjective.
\newline
Assume $T$ is strictly $c$-topologically surjective. Let $y\in c^{-1}B_F$, then by assumption there exist $x\in E$ such that $T(x)=y$ and $\Vert x\Vert\leq c\Vert 
y\Vert=1$. Since $y\in c^{-1}B_F^\circ$ is arbitrary, we have $T(B_E)\supset c^{-1}B_F$. 
Conversely assume that $T(B_E)\supset c^{-1}B_F$. Let $y\in F$, then 
$\tilde{y}=c^{-1}\Vert y\Vert^{-1}y\in c^{-1}B_F$. By assumption there exist 
$\tilde{x}\in B_E$ such that $T(\tilde{x})=\tilde{y}$. In this case for 
$x:=c\Vert y\Vert\tilde{x}$ we have $\Vert x \Vert=c'\Vert 
y\Vert\Vert\tilde{x}\Vert\leq c\Vert y\Vert$ and $T(x)=c'\Vert\ y\Vert 
T(\tilde{x})=c'\Vert y\Vert\tilde{y}=y$. Since $y\in F$ is arbitrary, then $T$ is strictly $c$-topologically surjective.
\newline
2) Assume $T$ is coisometric. Then $\Vert T\Vert\leq 1$ and as the consequence $T(B_E^\circ)\subset B_F^\circ$. From paragraph 1) it follows that $T(B_E^\circ)\supset B_F^\circ$. Taking into account the reverse inclusion we can say $T(B_E^\circ)=B_F^\circ$. Conversely, assume that $T(B_E^\circ)=B_F^\circ$. In particular $\Vert T\Vert\leq 1$ and $T(B_E^\circ)\supset B_F^\circ$. From paragraph 2) it follows that $T$ is $1$-topologically surjective. 
Hence, $T$ is coisometric. Similar arguments applies for strictly coisometric operators.
\end{proof}

\begin{proposition}\label{PrDualOps} 
Let $ T:E\to F$ be bounded linear operator between normed spaces and $c>0$, then
\newline
1) if $ T$ is (strictly) $c$-topologically surjective, then $ T^*$ is $c$-topologically injective
\newline
2) if $ T$ $c$-topologically injective, then $ T^*$ is strictly $c$-topologically surjective
\newline
3) if $ T^*$ (strictly) $c$-topologically surjective, then $ T$ is $c$-topologically injective
\newline
4) if $ T^*$ $c$-topologically injective and $E$ is complete, then $ T$ is $c$-topologically surjective
\end{proposition}
\begin{proof}
1) Since $T$ is $c$-topologically surjective we have $c^{-1}B_F^\circ\subset T(B_E^\circ)$, hence for all $g\in F^*$ we have
$$
\Vert  T^*(g)\Vert
=\sup\{|g( T(x))|:x\in B_E^\circ\}
=\sup\{|g(y)|: y\in T(B_E^\circ)\}
\geq\sup\{|g(y)|: y\in c^{-1}B_F^\circ\}
$$
$$
=\sup\{|g(c^{-1}y)|: y\in B_F^\circ\}
=c^{-1}\sup\{|g(y)|: y\in B_F^\circ\}
=c^{-1}\Vert g\Vert
$$
Since $g\in F^*$ is arbitrary $ T^*$ is $c$-topologically injective. Similar argument applies for strictly $c$-topologically surjective operator.
\newline
2) Let $g\in E^*$. Since $ T$ is $c$-topologically injective, then $\tilde{ T}:= T|^{\operatorname{Im}( T)}$ topological linear isomorphism. Denote by $i:\operatorname{Im}( T)\to F$ the natural embedding of $\operatorname{Im}( T)$ into $F$, then $ T=i\tilde{ T}$. Now consider bounded linear functional $f_0:=g\tilde{ T}^{-1}\in F^*$. By Hahn-Banach theorem there exist bounded linear functional $f\in F^*$ such that $\Vert f\Vert=\Vert f_0\Vert$ and $f_0=fi$. In this case
$g=f_0\tilde{ T}=f_0 i\tilde{ T}=f T= T^*(f)$. Since $ T$ is $c$-topologically injective, then for all $x\in F$ we have
$$
|f(x)|=|g(\tilde{ T}^{-1}(x))|
\leq\Vert g\Vert\Vert \tilde{ T}^{-1}(x)\Vert
\leq\Vert g\Vert c\Vert  T(\tilde{ T}^{-1}(x))\Vert
\leq c\Vert g\Vert\Vert x\Vert
$$
Hence $\Vert f\Vert\leq c\Vert g\Vert$. Since $g\in E^*$ is arbitrary, then $ T^*$ is strictly $c$-topologically surjective.
\newline
3) From paragraph $1)$ it follows that $ T^{**}$ is $c$-topologically injective. Note that natural embedding into the second dual is isometric and also that $\iota_F  T = T^{**}\iota_E$. 
Then for all $x\in E$ we get
$$
\Vert T(x)\Vert
=\Vert \iota_F( T(x))\Vert
=\Vert T^{**}(\iota_E(x))\Vert
\geq c^{-1}\Vert \iota_E(x)\Vert
=c^{-1}\Vert x\Vert
$$
Since $x\in E$ is arbitrary then $ T$ is $c$-topologically injective.
\newline
4) Assume that $c^{-1}B_F^\circ\not\subset\operatorname{cl}_F( T(B_E^\circ))$, then there exist $y_0\in c^{-1}B_F\setminus\operatorname{cl}_F( T(B_E^\circ))$. In particular, $\Vert y_0\Vert<c^{-1}$. Consider sets $A=\{y_0\}$ and 
$B=\operatorname{cl}_F( T(B_E^\circ))$. Obviously, $A$ is compact and convex. Since  $B_E^\circ$ is convex, and $ T$ is linear, then $ T(B_E^\circ)$ is also convex. As the consequence $B$ is closed and convex. By theorem 3.4 \cite{RudinFA}  there exist $g\in F^*$ and $\gamma_1,\gamma_2\in\mathbb{R}$ such that for all $y\in B$ holds $\operatorname{Re}(g(y_0))>\gamma_2>\gamma_1>\operatorname{Re}(g(y))$. Without loss of generality we may assume that $\gamma_1>\gamma_2=1$. So all $x\in B_E^\circ$ we get $\operatorname{Re}(g(y_0))>1>\operatorname{Re}(g( T(x)))$. Note that for all $x\in B_E^\circ$ there exist $\alpha\in\mathbb{C}$ 
such that $|\alpha|<1$ and $|g( T(x))|=\operatorname{Re}(g(T(\alpha x)))$. Since $|\alpha|\leq 1$, we see that $\alpha x\in B_E^\circ$ and $| T^*(g)(x)|=| T(g(x))|=\operatorname{Re}(g( T(\alpha x)))<1$. Since $x\in B_E^\circ$ 
is arbitrary, then $\Vert T^*(g)\Vert\leq 1$. Further $\Vert g\Vert>|g(y_0)|/\Vert y_0\Vert>c\operatorname{Re}(g(y_0))>c$, but $ T^*$ is $c$-topologically injective. Hence, 
$\Vert g\Vert\leq c\Vert T^*(g)\Vert\leq c$. Contradiction, so $c^{-1}B_F^\circ\subset \operatorname{cl}_F( 
T(B_E^\circ))$. As $E$ is complete, by proposition 4.4.1 \cite{HelFA} we get $c^{-1}B_F^\circ\subset T(B_E^\circ)$. This implies that $T$ is $c$-topologically surjective.
\end{proof}

\section{Operator sequence spaces}

\subsection {Matrix notation}

\begin{definition}\label{DefMatrNot}
Let $n,k\in\mathbb{N}$, then by $M_{n,k}$ we denote a linear space of complex valued matrices of the size $n\times k$. If $E$ is a linear space, then by $E^k$ we denote linear space of columns of the height $k$ with entries in $E$.
\end{definition}

For a given $\alpha\in M_{n,k}$ and $x\in E^k$ by $\alpha x$ we denote column in $E^n$ such that
$$
(\alpha x)_i=\sum\limits_{j=1}^n \alpha_{ij} x_j
$$
This formula is a natural generalization of matrix multiplication.

By default, the linear space $M_{n,k}$, endowed with operator norm $\Vert\cdot\Vert$, but sometimes we will need so called Hilbert-Schmidt norm. It is defined as follows. Let $\alpha\in M_{n,k}$, then its Hilbert-Schmidt norms is defined as
$$
\Vert\alpha\Vert_{hs}=\operatorname{trace}(|\alpha|^2)^{1/2}
$$
where $|\alpha|=(\alpha^*\alpha)^{1/2}$. Note that $\Vert\alpha\Vert\leq\Vert\alpha\Vert_{hs}$ and $\Vert|\alpha|\Vert_{hs}=\Vert|\alpha^*|\Vert=\Vert\alpha\Vert_{hs}$ (\cite{EROpSp}, 1.2).
By $\operatorname{diag}_n(\lambda_1,\ldots,\lambda_n)$  we will denote diagonal matrix of the size $n\times n$ with $\lambda_1,\ldots,\lambda_n$ on the main diagonal. We also use the notation $\operatorname{diag}_n(\lambda):=\operatorname{diag}_n(\lambda,\ldots,\lambda)$. Given matrices $\alpha_1\in M_{m,n_1},\ldots,\alpha_k\in M_{n,k_m}$ we can glue them together from the right to get the matrix $[\alpha_1,\ldots,\alpha_k]\in M_{n,k_1+\ldots+k_m}$.

\subsection{Examples and definitions} 
For the beginning we need to recall standard definitions from \cite{LamOpFolgen}.

\begin{definition}[\cite{LamOpFolgen}, 1.1.7]\label{DefSQSpace} 
Let $E$ be a linear space, and for each $n\in\mathbb{N}$ we have a norm on $\Vert \cdot \Vert_{\wideparen{n}}:E^n\to\mathbb{R}_+$. 
We say that the pair $X = (E^n, (\Vert \cdot \Vert_{\wideparen{n}})_{n \in \mathbb{N}})$, defines the structure of \textit{operator sequence} space on $E$, if the following conditions are satisfied:

1) for all $m, n \in \mathbb{N}$, $x \in E^{\wideparen{n}}$, $\alpha \in M_{m, n}$. holds
$$
\Vert \alpha x \Vert_{\wideparen{m}} \leq \Vert \alpha \Vert  \Vert x \Vert_{\wideparen{n}}
$$

2) for all $m, n \in \mathbb{N}$, $x \in E^n$, $y \in E^m$ holds
$$
\left\Vert \begin{pmatrix} x \\ y \end{pmatrix} \right\Vert^2_{\wideparen{n + m}} \leq   \Vert x \Vert_{\wideparen{n}}^2 + \Vert y \Vert_{\wideparen{m}}^2
$$

By $X^{\wideparen{n}}$ we will denote the normed space $(E^n,\Vert \cdot \Vert_{\wideparen{n}})$, we will call it $n$-th amplification of $X$.
\end{definition}

\begin{proposition}\label{PrRedundantAxiom} Let $X$ be an operator sequence space, $n\in\mathbb{N}$ and $x\in E^{\wideparen{n}}$, then
\newline
1) for all $m\in\mathbb{N}$ holds $\Vert (x, 0)^{tr}\Vert_{\wideparen{n + m}}=\Vert x\Vert_{\wideparen{n}}$
\newline
2) for any partial isometry $s\in M_{n,n}$ holds $\Vert sx\Vert_{\wideparen{n}}=\Vert x\Vert_{\wideparen{n}}$. In particluar norm doesn't change after permuation of coordinates
\end{proposition}
\begin{proof} 1) Result follows from inequalities
$$
\Vert (x, 0)^{tr}\Vert_{\wideparen{n + m}}\leq \left(\Vert x\Vert_{\wideparen{n}}^2+\Vert 0\Vert_{\wideparen{m}}^2\right)^{1/2}=\Vert x\Vert_{\wideparen{n}}
$$
$$
\Vert x\Vert_{\wideparen{n}}=\Vert[\operatorname{diag}_n(1),0](x,0)^{tr}\Vert_{n}\leq\Vert[\operatorname{diag}_n(1),0]\Vert\Vert(x,0)^{tr}\Vert_{\wideparen{n+m}}=
\Vert(x,0)^{tr}\Vert_{\wideparen{n+m}}
$$
2) Since $s$ is partial isometry, then $s^*s=\operatorname{diag}_n(1)$, so result follows from inequalities
$$
\Vert sx\Vert_{\wideparen{n}}\leq\Vert s\Vert\Vert x\Vert_{\wideparen{n}}=\Vert x\Vert_{\wideparen{n}}=
\Vert s^*sx\Vert_{\wideparen{n}}\leq\Vert s^*\Vert\Vert sx\Vert_{\wideparen{n}}=\Vert sx\Vert_{\wideparen{n}}
$$
\end{proof}

\begin{proposition}\label{PrCHaveUniqueOSS} The Hilbert space $\mathbb{C}$ have unique operator sequence space structure given by identifications $\mathbb{C}^{\wideparen{n}}=l_2^n$.
\end{proposition}
\begin{proof} Let $\mathbb{C}$ endowed with some operator sequence space structure. Fix $\xi\in\mathbb{C}^n$, then consider $\eta=(\Vert \xi\Vert_{l_2^n},0,\ldots,0)^{tr}\in \mathbb{C}^n$. Since $\Vert\eta\Vert_{l_2^n}=\Vert\xi\Vert_{l_2^n}$  there exist unitary matrix $s\in M_{n,n}$ such that $\eta=s\xi$. Therefore $\Vert\eta\Vert_{\wideparen{n}}=\Vert s\xi\Vert_{\wideparen{n}}\leq\Vert s\Vert\Vert\xi\Vert_{\wideparen{n}}=\Vert\xi\Vert_{\wideparen{n}}$. By proposition \ref{PrRedundantAxiom} we get that $\Vert\eta\Vert_{\wideparen{n}}=\Vert \xi\Vert_{l_2^n}$, hence $\Vert\xi\Vert_{\wideparen{n}}\geq\Vert\xi\Vert_{l_2^n}$. On the other hand from second axiom of operator sequence spaces we have 
$\Vert \xi\Vert_{\wideparen{n}}\leq\Vert\xi\Vert_{l_2^n}$, therefore $\Vert \xi\Vert_{\wideparen{n}}=\Vert \xi\Vert_{l_2^n}$. Since $n\in\mathbb{N}$ and $\xi\in \mathbb{C}^{\wideparen{n}}$ are arbitrary we conclude $\mathbb{C}^{\wideparen{n}}=l_2^n$.
\end{proof}

\begin{proposition}\label{PrSQAxiomRed}
Let $(\Vert\cdot\Vert_{\wideparen{n}}:E^n\to\mathbb{R}_+)_{n\in\mathbb{N}}$ be a family of functions satisfying axioms of operator sequence spaces, and assume that equality $\Vert x\Vert_{\wideparen{1}}=0$ implies $x=0$. 
Then $E$ is a operator sequence space.
\end{proposition}
\begin{proof}
Let $x\in E^n$ and $\lambda\in \mathbb{C}\setminus\{0\}$, then 
$$
\Vert\lambda x\Vert_{\wideparen{n}}
=\Vert\operatorname{diag}_n(\lambda)x\Vert_{\wideparen{n}}
\leq\Vert\operatorname{diag}_n(\lambda)\Vert\Vert x\Vert_{\wideparen{n}}
=|\lambda|\Vert x\Vert_{\wideparen{n}}
=|\lambda|\Vert\lambda^{-1}\lambda x\Vert_{\wideparen{n}}
\leq|\lambda||\lambda^{-1}|\Vert\lambda x\Vert_{\wideparen{n}}
=\Vert\lambda x\Vert_{\wideparen{n}}
$$
Consequently $\Vert\lambda x\Vert_{\wideparen{n}}=|\lambda|\Vert x\Vert_{\wideparen{n}}$ for all $\lambda\neq 0$. For $\lambda=0$ equality is obvious.
Let $x',x''\in E^n\setminus\{0\}$, then denote $\mu=(\Vert x'\Vert_{\wideparen{n}}^2+\Vert x''\Vert_{\wideparen{n}}^2)^{1/2}$. In this case
$$
\Vert x'+x''\Vert_{\wideparen{n}}^2
=\left\Vert\begin{pmatrix}\operatorname{diag}_n(\mu) & 0\\ 0 & \operatorname{diag}_n(\mu)\end{pmatrix}\begin{pmatrix}\mu^{-1}x'\\ \mu^{-1}x''\end{pmatrix}\right\Vert_{\wideparen{n}}^2
\leq\left\Vert\begin{pmatrix}\operatorname{diag}_n(\mu) & 0\\ 0 & \operatorname{diag}_n(\mu)\end{pmatrix}\right\Vert^2\left\Vert\begin{pmatrix}\mu^{-1}x'\\\mu^{-1}x''\end{pmatrix}\right\Vert_{\wideparen{n}}^2
$$
$$
\leq\mu^2(\mu^{-2}\Vert x'\Vert_{\wideparen{n}}^2+\mu^{-2}\Vert x''\Vert_{\wideparen{n}}^2)=\Vert x'\Vert_{\wideparen{n}}^2+\Vert x''\Vert_{\wideparen{n}}^2\leq (\Vert x'\Vert_{\wideparen{n}}+\Vert x''\Vert_{\wideparen{n}})^2
$$
Hence, for $x',x''\neq 0$ we have $\Vert x'+x''\Vert_{\wideparen{n}}\leq\Vert x'\Vert_{\wideparen{n}}+\Vert x''\Vert_{\wideparen{n''}}$. For $x'=x''=0$ the equality is obvious.
\end{proof}

\begin{proposition}[\cite{LamOpFolgen}, 1.1.4]\label{PrNormVsSQNorm} Let $X$ be operator sequence space, $n\in\mathbb{N}$. Then for all $x\in X^{\wideparen{n}}$ and $i=1,n$ holds
$$
\Vert x_i\Vert_{\wideparen{1}}\leq\Vert x\Vert_{\wideparen{n}}\leq\sum\limits_{k=1}^n\Vert x_k\Vert_{\wideparen{1}}\leq n\Vert x\Vert_{\wideparen{n}}
$$
\end{proposition}

We say that $X$ is a \textit{operator sequence} space of normed space $(E, \Vert \cdot \Vert_{\wideparen{1}}$). It is easy to see for a given operator sequence space $X$ the normed space $X^{\wideparen{n}}$ have its own natural structure of operator sequence space: it is enough to identify $(X^{\wideparen{n}})^{\wideparen{k}}$ with $X^{\wideparen{nk}}$.

\begin{example}[\cite{LamOpFolgen}, 1.1.8]\label{ExHilSQ} Let $H$ be a Hilbert space, then its maximal operator sequence space structure is given by identifications $\max(H)^{\wideparen{n}}=\bigoplus{}_2\{H:\lambda\in\mathbb{N}_n\}$. Obviously $\max(H)^{\wideparen{n}}$ is a Hilbert space for every $n\in\mathbb{N}$. We will call this structure the standard operator sequence space structure of $H$ and usually denote $\max(H)$ as $H$.
\end{example}

\begin{definition}[\cite{LamOpFolgen}, 1.1.18]\label{ExT2nSQ} 
Let $H$ be a Hilbert space, then its minimal operator sequence space structure is given by identifications $\min(H)^{\wideparen{n}} = \mathcal{B}(l_2^n,H)$. 
\end{definition}

By $t_2^n$ we denote $\min(l_2^n)$.

\begin{definition}\label{DefOpSubAlgSQ} Let $A$ be a subalgebra of $\mathcal{B}(H)$ for some Hilbert space $H$, then we define its standard operator sequence space structure by embedding $A^n\hookrightarrow \mathcal{B}(H,H^{\wideparen{n}})$.
\end{definition}

\begin{proposition}\label{PrCstarAlgSQ} Let $A$ be a $C^*$ algebra, then its standard operator sequence space structure doesn't depend on its representation on Hilbert space and for any $n\in\mathbb{N}$ and $a\in A^{\wideparen{n}}$ we have
$$
\Vert a\Vert_{\wideparen{n}}=\left\Vert\sum\limits_{i=1}^n a_i^*a_i\right\Vert^{1/2}
$$
In particular standard operator sequence space structures of $\mathbb{C}$ regarded as $C^*$ algebra and as Hilbert space are the same.
\end{proposition}
\begin{proof} Let $\pi:A\to\mathcal{B}(H)$ be any isometric ${}^*$-representation of $A$ on the Hilbert space $H$. Fix $n\in\mathbb{N}$ then $a\in A^{\wideparen{n}}$ is identified with operator $T:H\mapsto H^{\wideparen{n}}:\xi\mapsto \oplus{}_2\{\pi(a_i)(\xi):i\in\mathbb{N}_n\}$. Then 
$$
\Vert a\Vert_{\wideparen{n}}^2
=\Vert T\Vert^2
=\sup\{\Vert \oplus_2\{\pi(a_i)(\xi):i\in\mathbb{N}_n\}\Vert^2:\xi\in B_H\}=
$$
$$
=\sup\left\{ \sum\limits_{i=1}^n\langle \pi(a_i)(\xi),\pi(a_i)(\xi)\rangle:\xi\in B_H\right\}
=\sup\left\{ \left\langle \pi\left(\sum\limits_{i=1}^n a_i^*a_i\right)(\xi),\xi\right\rangle:\xi\in B_H\right\}
$$
From proposition 2.2.4 and 2.2.5 \cite{MurphCstarOpTh} we get that $\sum_{i=1}^n a_i^* a_i\geq 0$, so $\pi(\sum_{i=1}^n a_i^* a_i)\geq 0$ and by proposition 6.4.6 \cite{HelFA} we get that
$$
\Vert a\Vert_{\wideparen{n}}^2
=\sup\left\{ \left\langle \pi\left(\sum\limits_{i=1}^n a_i^*a_i\right)(\xi),\xi\right\rangle:\xi\in B_H\right\}
=\left\Vert \pi\left(\sum\limits_{i=1}^n a_i^*a_i\right)\right\Vert
=\left\Vert \sum\limits_{i=1}^n a_i^*a_i\right\Vert
$$
If $A=\mathbb{C}$ we get
$$
\Vert a\Vert_{\wideparen{n}}
=\left| \sum\limits_{i=1}^n \overline{a_i}a_i\right|^{1/2}
=\left(\sum\limits_{i=1}^n |a_i|^2\right)^{1/2}
=\Vert a\Vert_{l_2^n}
$$
so both definitions give the same operator sequence space structure.
\end{proof}

\begin{proposition}\label{PrCommCstarSQ} Let $\Omega$ be a locally compact topological space, then for any $n\in\mathbb{N}$ we have an isometric isomorphism
$$
i_C:C_0(\Omega)^{\wideparen{n}}\to C_0(\Omega,\mathbb{C}^n):f\mapsto (\omega\mapsto(f_i(\omega))_{i\in\mathbb{N}_n})
$$
\end{proposition} 
\begin{proof} Using proposition \ref{PrCstarAlgSQ} for any $f\in C_0(\Omega)^{\wideparen{n}}$ we get
$$
\Vert f\Vert_{\wideparen{n}}
=\left\Vert \sum\limits_{i=1}^n f_i^* f_i\right\Vert^{1/2}
=\sup\left\{\left(\sum\limits_{i=1}^n |f_i(\omega)|^2\right)^{1/2}:\omega\in\Omega\right\}
=\sup\{\Vert i_C(f)(\omega)\Vert:\omega\in\Omega\}
=\Vert i_C(f)\Vert
$$
Thus $i_C$ is an isometry. For a given $g\in C_0(\Omega,\mathbb{C}^n)$ and each $i\in\mathbb{N}_n$ consider continuous function $f_i:\Omega\to\mathbb{C}:\omega\mapsto g(\omega)_i$ and define $f=(f_1,\ldots,f_n)^{tr}\in C_0(\Omega)^{\wideparen{n}}$. Clearly, $i_C(f)=g$, so $i_C$ is surjective. Therefore $i_C$ is a surjective isometry, hence isometric isomorphism.
\end{proof}

\subsection{Operators between operator sequence spaces}

\begin{definition}[\cite{LamOpFolgen}, 1.2.1]\label{DefSBOp}
Let $X$ and $Y$ be operator sequence spaces and $\varphi : X \to Y$ be a linear operator. For a given $n\in\mathbb{N}$ its $n$-th \textit{amplification} is called a linear operator $\varphi^{\wideparen{n}} : X^{\wideparen{n}} \to Y^{\wideparen{n}}$ 
defined by 
$$
\varphi^{\wideparen{n}}(x)=(\varphi(x_i))_{i=1,n}
$$
We say that $\varphi$ \textit{sequentially bounded}, if 
$$
\Vert \varphi \Vert_{sb} := \sup\{\Vert \varphi^{\wideparen{n}}\Vert_{\mathcal{B}(X^{\wideparen{n}},Y^{\wideparen{n}})}:n\in\mathbb{N}\}  < \infty
$$
\end{definition}

\begin{proposition}\label{PrSimplAmplProps}
Let $X$, $Y$, $Z$ be operator sequence spaces, $\varphi:X\to Y$, $\psi:Y\to Z$ be linear operators and $n,m\in\mathbb{N}$. Then 
\newline
1) $\varphi$ injective (surjective)  if and only if  $\varphi^{\wideparen{n}}$ injective (surjective).
\newline
2) $\Vert\varphi\Vert_{\wideparen{n}}\leq\Vert\varphi\Vert_{\wideparen{n+1}}$ and as the consequence $\mathcal{SB}(X,Y)\subset\mathcal{B}(X,Y)$.
\newline
3) $(\psi\varphi)^{\wideparen{n}}=\psi^{\wideparen{n}}\varphi^{\wideparen{n}}$ and as the consequence  $\Vert\psi\varphi\Vert_{sb}\leq\Vert\psi\Vert_{sb}\Vert\varphi\Vert_{sb}$
\newline
4) For all $\alpha\in M_{n,m}$, $x\in X^{\wideparen{m}}$ holds $\varphi^{\wideparen{n}}(\alpha x)=\alpha\varphi^{\wideparen{m}}(x)$
\end{proposition}
\begin{proof}
1) Directly follows from definition
\newline
2) From proposition \ref{PrRedundantAxiom} we get
$$
\Vert\varphi\Vert_{\wideparen{n}}
=\sup\{\Vert\varphi^{\wideparen{n}}(x)\Vert_{\wideparen{n}}:x\in B_{X^{\wideparen{n}}}\}
=\sup\{\Vert\varphi^{\wideparen{n+1}}((x,0)^{tr})\Vert_{\wideparen{n}}:(x,0)^{tr}\in B_{X^{\wideparen{n+1}}}\}
$$
$$
=\sup\{\Vert\varphi^{\wideparen{n+1}}(x)\Vert_{\wideparen{n}}:x\in B_{X^{\wideparen{n+1}}}\}
=\Vert\varphi\Vert_{\wideparen{n+1}}
$$	
\newline
3) For all $x\in X^{\wideparen{n}}$ we have
$$(\psi\varphi)^{\wideparen{n}}(x)
=((\psi\varphi)(x_i))_{i\in\mathbb{N}_n}
=((\psi(\varphi(x_i)))_{i\in\mathbb{N}_n}
=\psi^{\wideparen{n}}((\varphi(x_i))_{i\in\mathbb{N}_n})
=\psi^{\wideparen{n}}(\varphi^{\wideparen{n}}(x))
$$
so $(\psi\varphi)^{\wideparen{n}}=\psi^{\wideparen{n}}\varphi^{\wideparen{n}}$. And what is more, 
$$
\Vert\psi\varphi\Vert_{sb}
=\sup\{\Vert\psi^{\wideparen{n}}\varphi^{\wideparen{n}}\Vert:n\in\mathbb{N}\}
\leq\sup\{\Vert\psi^{\wideparen{n}}\Vert\Vert\varphi^{\wideparen{n}}\Vert:n\in\mathbb{N}\}
\leq\Vert\psi\Vert_{sb}\Vert\varphi\Vert_{sb}
$$
4) For each $i\in\mathbb{N}_n$ holds
$$
\varphi^{\wideparen{n}}(\alpha x)_i
=\varphi((\alpha x)_i)
=\varphi\left(\sum\limits_{j=1}^m \alpha_{ij }x_j\right)
=\sum\limits_{j=1}^m\alpha_{ij} \varphi(x_j)
=\sum\limits_{j=1}^m\alpha_{ij} \varphi^{\wideparen{m}}(x)_j=
(\alpha\varphi^{\wideparen{m}}(x))_i
$$
So $\varphi^{\wideparen{n}}(\alpha x)=\alpha\varphi^{\wideparen{m}}(x)$.
\end{proof}

\begin{definition}\label{DefSBOpType}
Let $\varphi:X\to Y$ be sequentially bounded operator between operator sequence spaces $X$ and $Y$, then $\varphi$ is called:
\newline
1) \textit{sequentially contractive}, if $\Vert \varphi\Vert_{sb}\leq 1$
\newline
2) \textit{sequentially $c$-topologically injective}, if for all $n \in \mathbb{N}$ the linear operator $\varphi^{\wideparen{n}}$ is $c$-topologically injective. If mentioning of constant $c$ will be irrelevant we will simply say that $\varphi$ sequentially topologically injective.
\newline
3) \textit{sequentially (strictly) $c$-topologically surjective}, if for all $n \in \mathbb{N}$ the linear operator $\varphi^{\wideparen{n}}$ is (strictly) $c$-topologically surjective. If mentioning of constant $c$ will be irrelevant we will simply say that $\varphi$ sequentially topologically surjective.
\newline
4) \textit{sequentially isometric}, if for all $n\in\mathbb{N}$ the linear operator $\varphi^{\wideparen{n}}$ is isometric
\newline
5) \textit{sequentially (strictly) coisometric}, if for all $n\in\mathbb{N}$ the linear operator $\varphi^{\wideparen{n}}$ is (strictly) coisometric
\end{definition}

\begin{proposition}\label{PrComposeSQTopInjSur} Let $X$, $Y$, $Z$ be operator sequence spaces and $\varphi_1\in\mathcal{SB}(X,Y)$, $\varphi_2\in\mathcal{SB}(Y,Z)$. Then
\newline
1) if $\varphi_i$ is sequentially $c_i$-topologically injective for $i\in\mathbb{N}_2$, then $\varphi_2\varphi_1$ is sequentially $c_2c_1$-topologically injective.
\newline
2) if $\varphi_i$ is (strictly) sequentially $c_i$-topologically surjective for $i\in\mathbb{N}_2$, then $\varphi_2\varphi_1$ is (strictly) sequentially $c_2c_1$-topologically surjective.
\end{proposition}
\begin{proof}
1) For each $n\in\mathbb{N}$ and $x\in X^{\wideparen{n}}$ we have $\Vert(\varphi_2\varphi_1)^{\wideparen{n}}(x)\Vert_{\wideparen{n}}=\Vert\varphi_2^{\wideparen{n}}(\varphi_1^{\wideparen{n}}(x))\Vert_{\wideparen{n}}
\geq c_2^{-1}\Vert\varphi_1^{\wideparen{n}}(x)\Vert_{\wideparen{n}}\geq c_2^{-1}c_1^{-1}\Vert x\Vert_{\wideparen{n}}$, hence $\varphi_2\varphi_1$ is sequentially $c_2c_1$-topologically injective.

2) Assume $\varphi_i$ is sequentially $c_i$-topologically surjective for $i\in\mathbb{N}_2$. From proposition \ref{PrEquivDescOfIsomCoisomOp} for each $n\in\mathbb{N}$ we have $(\varphi_2\varphi_1)^{\wideparen{n}}(B_{X^{\wideparen{n}}}^\circ)=\varphi_2^{\wideparen{n}}(\varphi_1^{\wideparen{n}}(B_{X^{\wideparen{n}}}^\circ))\supset\varphi_2^{\wideparen{n}}(c_1^{-1}B_{Y^{\wideparen{n}}}^\circ)=c_1^{-1}\varphi_2^{\wideparen{n}}(B_{Y^{\wideparen{n}}}^\circ)=c^{-1}c_2^{-1}B_{Z^{\wideparen{n}}}^\circ$. Again from proposition \ref{PrEquivDescOfIsomCoisomOp} we get that $\varphi_2\varphi_1$ is sequentially $c_2c_1$-topologically surjective.
\end{proof}

Now we can define two main categories in question. These are $SQNor$ and $SQNor_1$. Objects in $SQNor$ are operator sequence spaces, morphisms are sequentially bounded operators. Objects of $SQNor_1$ are operator sequence spaces, morphisms are sequentially contractive operators. 

\begin{proposition}[\cite{LamOpFolgen}, 1.2.14]\label{PrSmithsLemma}
Let $X$, $Y$ be operator sequence spaces and $d=\operatorname{dim}(Y)<\infty$, then for all $T\in\mathcal{B}(X,Y)$ holds
$$
\Vert T\Vert_{sb}=\Vert T^{\wideparen{d}}\Vert
$$
\end{proposition}

\begin{proposition}[\cite{LamOpFolgen}, 1.2.14]\label{PrSQSpaceIsSBFromT2n}
Let $X$ be operator sequence space, $n\in\mathbb{N}$. Then the linear map 
$$
i_{t_2}:X^{\wideparen{n}}\to\mathcal{SB}(t^n_2, X):x\mapsto\left(\xi\mapsto\sum\limits_{i=1}^n\xi_ix_i\right)
$$
is an isometric isomorphism.
\end{proposition}

The space of sequentially bounded operators between operator sequence spaces $X$ and $Y$ will be denoted by $\mathcal{SB}(X, Y)$. Obviously, this is normed space, and what is more we can define operator sequence space structure on $\mathcal{SB}(X, Y)$ via identification
$$
\mathcal{SB}(X, Y)^{\wideparen{n}} = \mathcal{SB}(X, Y^{\wideparen{n}})
$$
In this identification every $\varphi\in\mathcal{SB}(X,Y)^{\wideparen{n}}$ is mapped to the linear operator
$$
A(\varphi):X\to Y^{\wideparen{n}}:x\mapsto(\varphi_i(x))_{i\in\mathbb{N}_n}
$$

\begin{definition}[\cite{LamOpFolgen}, 1.2.11]\label{DefSBbiOp}
Let $\mathcal{R}:X\times Y\to Z$ be bilinear operator between operator sequence spaces $X$, $Y$, $Z$. For a given $n,m\in\mathbb{N}$ its $n\times m$-th amplification is a linear operator
$$
\mathcal{R}^{\wideparen{n\times m}}:X^{\wideparen{n}}\times Y^{\wideparen{m}}\to Z^{\wideparen{nm}}:(x,y)\mapsto(\mathcal{R}(x_i,y_j))_{i\in\mathbb{N}_n,j\in\mathbb{N}_m}
$$
Bilinear operator $\mathcal{R}$ is called sequentially bounded if
$$
\Vert\mathcal{R}\Vert_{sb}:=\sup\{\Vert \mathcal{R}^{\wideparen{n\times m}}\Vert_{\mathcal{B}(X^{\wideparen{n}}\times Y^{\wideparen{m}}, Z^{\wideparen{nm}})}:n,m\in\mathbb{N}\}<\infty
$$
\end{definition}

\begin{definition}\label{DefSBBiOpType}
Let $\mathcal{R}:X\times Y\to Z$ be bounded bilinear operator between normed spaces $X$, $Y$ and $Z$, then
\newline
1) if $Y$ and $Z$ ($X$ and $Z$) are operator sequence spaces, then $\mathcal{R}$ is called sequentially isometric from the left (right) if the operator ${}^X\mathcal{R}:X\to\mathcal{SB}(Y,Z)$ ($\mathcal{R}^Y:Y\to\mathcal{SB}(X,Z)$) is isometric.
\newline
2) if $X$, $Y$, $Z$ are operator sequence spaces, then $\mathcal{R}$ is called sequentially contractive if $\Vert \mathcal{R}\Vert_{sb}\leq 1$.
\end{definition}

For a given operator sequence spaces $X$, $Y$, $Z$ by $\mathcal{SB}(X\times Y, Z)$ we will denote the space of sequentially bounded bilinear operators from $X\times Y$ to $Z$. Obviously, this is a normed space, and what is more we can define operator sequence space structure on $\mathcal{SB}(X\times Y, Z)$ via identification
$$
\mathcal{SB}(X\times Y, Z)^{\wideparen{n}}=\mathcal{SB}(X\times Y,Z^{\wideparen{n}})
$$
where $n\in\mathbb{N}$. In this identification every $\mathcal{R}\in\mathcal{SB}(X\times Y,Z)^{\wideparen{n}}$ is mapped to the bilinear operator 
$$
A(\mathcal{R}):X\times Y\to Z^{\wideparen{n}}:(x,y)\mapsto(\mathcal{R}_i(x,y))_{i\in\mathbb{N}_n}
$$
It is easy to check that for all $x\in X^{\wideparen{n}}$, $y\in Y^{\wideparen{m}}$ and $\alpha\in M_{k,n}$  holds
$$
A((\mathcal{R}^Y)^{\wideparen{m}}(y))^{\wideparen{n}}(x)=A(({}^X\mathcal{R})^{\wideparen{n}}(x))^{\wideparen{m}}(y)=\mathcal{R}^{\wideparen{n\times m}}(x,y)
$$
$$
\mathcal{R}^{\wideparen{n\times k}}(x,\alpha y)
=[\alpha,\ldots,\alpha]\mathcal{R}^{\wideparen{n\times m}}(x,y)
$$

\begin{proposition}\label{PrScalMultSB}
Let $X$ be a operator sequence space, then the bilinear operator $\mathcal{M}:\mathbb{C}\times X\to X:(\alpha, x)\mapsto \alpha x$ is sequentially contractive.
\end{proposition}
\begin{proof}
Let $\alpha\in\mathbb{C}^{\wideparen{n}}$ and $x\in X^{\wideparen{m}}$. Consider matrix $\beta=[\operatorname{diag}_m(\alpha_1),\ldots,\operatorname{diag}_m(\alpha_n)]^{tr}$, then one can easily check that $\Vert\beta\Vert=\Vert\alpha\Vert_{\wideparen{n}}$. Now note that $\Vert\mathcal{M}^{\wideparen{n\times m}}(\alpha, x)\Vert_{\wideparen{n\times m}}
=\Vert\beta x\Vert_{\wideparen{n\times m}}
\leq\Vert\alpha\Vert_{\wideparen{n}}\Vert x\Vert_{\wideparen{m}}$. Since $m,n\in\mathbb{N}$ are arbitrary $\Vert\mathcal{M}\Vert_{sb}\leq 1$.
\end{proof}

\begin{proposition}\label{PrRestrOfSBBilOpIsSB}
Let $X$, $Y$ and $Z$ be operator sequence spaces and $\mathcal{R}:X\times Y\to Z$ be a sequentially bounded bilinear operator, then for a fixed $x\in X^{\wideparen{1}}$ ($y\in Y^{\wideparen{1}}$) the linear operator ${}^X\mathcal{R}(x)$ ($\mathcal{R}^Y(y)$) is sequentially bounded with $\Vert{}^X\mathcal{R}(x)\Vert_{sb}\leq\Vert\mathcal{R}\Vert_{sb}\Vert x\Vert_{\wideparen{1}}$ ($\Vert\mathcal{R}^Y(y)\Vert_{sb}\leq\Vert\mathcal{R}\Vert_{sb}\Vert y\Vert_{\wideparen{1}}$).
\end{proposition}
\begin{proof}
Let $n\in\mathbb{N}$ and $x\in X^{\wideparen{n}}$, then
$$
\Vert(\mathcal{R}^Y(y))^{\wideparen{n}}(x)\Vert_{\wideparen{n}}
=\Vert \mathcal{R}^{\wideparen{n\times 1}}(x,y)\Vert_{\wideparen{n\times 1}}
\leq\Vert \mathcal{R}\Vert_{sb}\Vert x\Vert_{\wideparen{n}}\Vert y\Vert_{\wideparen{1}}
$$
Hence $\Vert\mathcal{R}^Y(y)\Vert_{sb}\leq\Vert\mathcal{R}\Vert_{sb}\Vert y\Vert_{\wideparen{1}}$.
For the remaining case the proof is the same.
\end{proof}

\begin{proposition}\label{PrSQNormViaDuality}
Let $Z$ be operator sequence space, $X$ ($Y$) be operator sequence space and $Y$ ($X$) be a normed space. Assume $\mathcal{R}:X\times Y\to Z$ is sequentially isometric from the right (from the left), then 
there is operator sequence space structure on $Y$ ($X$) given by family of norms
$$
\Vert y\Vert_{\wideparen{k}}^{\mathcal{R}}=\sup\{\Vert\mathcal{R}^{\wideparen{n\times k}}(x,y)\Vert_{\wideparen{n\times k}}:x\in B_{X^{\wideparen{n}}}, n\in\mathbb{N}\}
$$
$$
(\Vert x\Vert_{\wideparen{k}}^{\mathcal{R}}=\sup\{\Vert\mathcal{R}^{\wideparen{k\times n}}(x,y)\Vert_{\wideparen{k\times n}}:y\in B_{Y^{\wideparen{n}}}, n\in\mathbb{N}\})
$$
meanwhile $\Vert \mathcal{R}\Vert_{sb}\leq 1$. If additionally $d=\operatorname{dim}(Z)<\infty$, then
$$
\Vert y\Vert_{\wideparen{k}}^{\mathcal{R}}=\sup\{\Vert\mathcal{R}^{\wideparen{dk\times k}}(x,y)\Vert_{\wideparen{dk\times k}}:x\in B_{X^{\wideparen{dk}}}\}
$$
$$
(\Vert x\Vert_{\wideparen{k}}^{\mathcal{R}}=\sup\{\Vert\mathcal{R}^{\wideparen{k\times dk}}(x,y)\Vert_{\wideparen{k\times dk}}:y\in B_{Y^{\wideparen{dk}}}\})
$$
\end{proposition}
\begin{proof}
We will consider only the case of bilinear operator sequentially isometric from the right. For the remaining case all arguments are the same. Let $y\in Y^{\wideparen{k}}$, where $k\in\mathbb{N}$. We will show that $\Vert y\Vert_{\wideparen{k}}$ is well defined. Indeed
$$
\Vert\mathcal{R}^{\wideparen{n\times k}}(x,y)\Vert_{\wideparen{n\times k}}
=\left\Vert\sum\limits_{j=1}^k\mathcal{R}^{\wideparen{n\times k}}(x,(\delta_{ji}y_i)_{i\in\mathbb{N}_k})\right\Vert_{\wideparen{n\times k}}
\leq\sum\limits_{j=1}^k\Vert\mathcal{R}^{\wideparen{n\times k}}(x,(\delta_{ji}y_i)_{i\in\mathbb{N}_k})\Vert_{\wideparen{n\times k}}
$$
$$
=\sum\limits_{j=1}^k\Vert\mathcal{R}^{\wideparen{n\times 1}}(x,y_j)\Vert_{\wideparen{n\times 1}}
=\sum\limits_{j=1}^k\Vert A(\mathcal{R}^Y(y_j))^{\wideparen{n}}(x)\Vert_{\wideparen{n}}
$$
Now recall that $\mathcal{R}$ is sequentially isometric from the right
$$
\Vert y\Vert_{\wideparen{k}}^{\mathcal{R}}
\leq\sum\limits_{j=1}^k\sup\{\Vert A(\mathcal{R}^Y(y_j))^{\wideparen{n}}(x)\Vert_{\wideparen{n}}:x\in B_{X^{\wideparen{n}}}, n\in\mathbb{N}\}
=\sum\limits_{j=1}^k \Vert\mathcal{R}^Y(y_j)\Vert_{sb}
=\sum\limits_{j=1}^k \Vert y_j\Vert<+\infty
$$
Hence, the function $\Vert\cdot\Vert_{\wideparen{k}}:Y^{\wideparen{k}}\to\mathbb{R}_+$ is well defined. It is remains to check axioms of operator sequence spaces. Let $\alpha\in M_{m,k}$ and $x\in X^{\wideparen{n}}$, then it is easy to see that
$$
\Vert\mathcal{R}^{\wideparen{n\times m}}(x,\alpha y)\Vert_{\wideparen{n\times m}}
=\Vert[\alpha,\ldots,\alpha]\mathcal{R}^{\wideparen{n\times k}}(x,y)\Vert_{\wideparen{n\times k}}
$$
$$
\leq\Vert[\alpha,\ldots,\alpha]\Vert_{M_{m,nk}}\Vert\mathcal{R}^{\wideparen{n\times k}}(x,y)\Vert_{\wideparen{n\times k}}
=\Vert\alpha\Vert\Vert\mathcal{R}^{\wideparen{n\times k}}(x,y)\Vert_{\wideparen{n\times k}}
$$
Hence
$$
\Vert\alpha y\Vert_{\wideparen{m}}^{\mathcal{R}}
=\sup\{\Vert\mathcal{R}^{\wideparen{n\times m}}(x,\alpha y)\Vert_{\wideparen{n\times m}}:x\in B_{X^{\wideparen{n}}},n\in\mathbb{N}\}
\leq\sup\{\Vert\alpha\Vert\Vert\mathcal{R}^{\wideparen{n\times k}}(x,y)\Vert_{\wideparen{n\times k}}:x\in B_{X^{\wideparen{n}}},n\in\mathbb{N}\}
$$
$$
\leq\Vert\alpha\Vert\sup\{\Vert\mathcal{R}^{\wideparen{n\times k}}(x,y)\Vert_{\wideparen{n\times k}}:x\in B_{X^{\wideparen{n}}},n\in\mathbb{N}\}
=\Vert\alpha\Vert\Vert y\Vert_{\wideparen{k}}^{\mathcal{R}}
$$
Let $0<l<k$ and $y=(y', y'')^{tr}$, where $y'\in Y^{\wideparen{l}}$, $y''\in Y^{\wideparen{k-l}}$, then
$$
\Vert \mathcal{R}^{\wideparen{n\times k}}(x,y)\Vert_{\wideparen{n\times k}}^2
=\left\Vert\begin{pmatrix}\mathcal{R}^{\wideparen{n\times l}}(x,y')\\ \mathcal{R}^{\wideparen{n\times (k-l)}}(x,y'')\end{pmatrix}\right\Vert_{n\times k}^2
\leq\Vert\mathcal{R}^{\wideparen{n\times l}}(x,y')\Vert_{\wideparen{n\times l}}^2
+\Vert\mathcal{R}^{\wideparen{n\times (k-l)}}(x,y'')\Vert_{\wideparen{n\times (k-l)}}^2
$$
Consequently,
$$
\Vert y\Vert_{\wideparen{k}}^{\mathcal{R}}{}^2
\leq\sup\{\Vert\mathcal{R}^{\wideparen{n\times l}}(x,y')\Vert_{\wideparen{n\times l}}^2:x\in B_X^{\wideparen{n}},n\in\mathbb{N}\}
+\sup\{\Vert\mathcal{R}^{\wideparen{n\times (k-l)}}(x,y'')\Vert_{\wideparen{n\times (k-l)}}^2:x\in B_{X^{\wideparen{n}}},n\in\mathbb{N}\}
$$
$$
=\Vert y'\Vert_{\wideparen{l}}^{\mathcal{R}}{}^2
+\Vert y''\Vert_{\wideparen{k-l}}^{\mathcal{R}}{}^2
$$
Finally, for all $y\in Y^{\wideparen{1}}$ we have
$$
\Vert y\Vert_{\wideparen{1}}^{\mathcal{R}}
=\sup\{\Vert\mathcal{R}^{\wideparen{n\times 1}}(x,y)\Vert_{\wideparen{n\times 1}}:x\in B_{X^{\wideparen{n}}},n\in\mathbb{N}\}
=\sup\{\Vert A(\mathcal{R}^Y(y))^{\wideparen{n}}(x)\Vert_{\wideparen{n}}:x\in B_{X^{\wideparen{n}}},n\in\mathbb{N}\}
$$
$$
=\Vert \mathcal{R}^Y(y)\Vert_{sb}=\Vert y\Vert
$$
Now from proposition \ref{PrSQAxiomRed} it follows that, family of functions $(\Vert\cdot\Vert_{\wideparen{k}})_{k\in\mathbb{N}}$ defines a operator sequence space structure on the normed space $Y$. From definition of norm on $Y^{\wideparen{k}}$ it follows that $\Vert\mathcal{R}^{\wideparen{n\times k}}\Vert\leq 1$ for all $n\in\mathbb{N}$. Since $n,k\in\mathbb{N}$ are arbitrary, then $\Vert\mathcal{R}\Vert_{sb}\leq 1$.

If $Z$ is finite dimensional, then from proposition \ref{PrSmithsLemma} we get
$$
\Vert y\Vert_{\wideparen{k}}^{\mathcal{R}}
=\sup\{\Vert\mathcal{R}^{\wideparen{n\times k}}(x,y)\Vert_{\wideparen{n\times k}}:x\in B_{X^{\wideparen{n}}},n\in\mathbb{N}\}
=\sup\{\Vert A((\mathcal{R}^Y){}^{\wideparen{k}}(y))^{\wideparen{n}}(x)\Vert_{\wideparen{n\times k}}:x\in B_{X^{\wideparen{n}}},n\in\mathbb{N}\}
$$
$$
=\Vert A((\mathcal{R}^Y){}^{\wideparen{k}}(y))\Vert_{sb}
=\Vert A((\mathcal{R}^Y){}^{\wideparen{k}}(y))^{\wideparen{dk}}\Vert
=\sup\{\Vert A((\mathcal{R}^Y){}^{\wideparen{k}}(y))^{\wideparen{dk}}(x)\Vert_{\wideparen{k\times dk}}:x\in B_{X^{\wideparen{d}}}\}
$$
$$
=\sup\{\Vert\mathcal{R}^{\wideparen{dk\times k}}(x,y)\Vert_{\wideparen{dk\times k}}:x\in B_{X^{\wideparen{dk}}}\}
$$
\end{proof}

\begin{proposition}\label{PrFreezIsomSQIsom}
Let $Z$ be operator sequence space, $X$ ($Y$) be operator sequence space and $Y$ ($X$) be a normed space. Assume $\mathcal{R}:X\times Y\to Z$ is sequentially isometric from the right (from the left). Endow $Y(X)$ with the structure of operator sequence space as it was done in \ref{PrSQNormViaDuality}, then the  
linear operator $\mathcal{R}^Y$ (${}^X\mathcal{R}$) is sequentially isometric.
\end{proposition}
\begin{proof}
We will consider only the case of bilinear operator sequentially isometric from the right. For the remaining case all arguments are the same. Let $k\in\mathbb{N}$. For all $y\in Y^{\wideparen{k}}$ we have
$$
\Vert y\Vert_{\wideparen{k}}
=\sup\{\Vert\mathcal{R}^{\wideparen{n\times k}}(x,y)\Vert_{\wideparen{n\times k}}:x\in B_{X^{\wideparen{n}}}, n\in\mathbb{N}\}
=\sup\{\Vert A((\mathcal{R}^Y)^{\wideparen{k}}(y))^{\wideparen{n}}(x)\Vert_{\wideparen{n\times k}}:x\in B_{X^{\wideparen{n}}}, n\in\mathbb{N}\}
$$
$$
=\Vert A((\mathcal{R}^Y)^{\wideparen{k}}(y))\Vert_{sb}
=\Vert (\mathcal{R}^Y)^{\wideparen{k}}(y)\Vert_{\wideparen{k}}
$$
So, $\mathcal{R}^Y$ is sequentially isometric
\end{proof}

If conditions of previous proposition are satisfied we say that bilinear operator $\mathcal{R}$ induces operator sequence space structure on $Y$ ($X$).

\begin{proposition}\label{PrSQOpSqQuanIsEquivToStandard}
Let $X$, $Y$ be operator sequence spaces, then the standard operator sequence space structure of $\mathcal{SB}(X,Y)$ coincides with operator sequence space structure of induced by bilinear operator
$$
\mathcal{E}:X\times\mathcal{SB}(X,Y)\to Y:(x,\varphi)\mapsto\varphi(x)
$$
\end{proposition}
\begin{proof}
In this particular case statement that $\mathcal{E}$ is sequentially isometric from the right is tautological. Hence $\mathcal{E}$ induces operator sequence space structure on $\mathcal{SB}(X,Y)$. Let $k\in\mathbb{N}$ and $\varphi\in\mathcal{SB}(X,Y)^{\wideparen{k}}$. 
Obviously $\mathcal{\mathcal{E}}^{\mathcal{SB}(X,Y)}=1_{\mathcal{SB}(X,Y)}$, so
$$
\Vert\varphi\Vert_{\wideparen{k}}^{\mathcal{E}}
=\sup\{\Vert\mathcal{E}^{\wideparen{n\times k}}(x,\varphi)\Vert_{\wideparen{n\times k}}:x\in B_{X^{\wideparen{n}}}, n\in\mathbb{N}\}
=\sup\{\Vert A((\mathcal{E}^{\mathcal{SB}(X,Y)})^{\wideparen{k}}(\varphi))^{\wideparen{n}}(x)\Vert_{\wideparen{n\times k}}:x\in B_{X^{\wideparen{n}}}, n\in\mathbb{N}\}
$$
$$
=\sup\{\Vert A((1_{\mathcal{SB}(X,Y)})^{\wideparen{k}}(\varphi))^{\wideparen{n}}\Vert: n\in\mathbb{N}\}
=\sup\{\Vert A(\varphi)^{\wideparen{n}}\Vert: n\in\mathbb{N}\}
=\Vert A(\varphi)\Vert_{sb}=\Vert\varphi\Vert_{\wideparen{k}}
$$
\end{proof}

\subsection{Completion of operator sequence spaces}

\begin{definition}\label{DefSQBanSpace}
Sequential operator space $X$ is called \textit{Banach operator sequence space}, if $X^{\wideparen{1}}$ is a Banach space.
\end{definition}

\begin{proposition}\label{PrSQSpaceComplSoAmplISCompl}
Let $X$ be operator sequence space, $n\in\mathbb{N}$. Then $X^{\wideparen{1}}$ is a Banach space  if and only if  $X^{\wideparen{n}}$ is a Banach space.
\end{proposition}
\begin{proof}. Assume $X^{\wideparen{1}}$ is complete. Let $(x^{(k)})_{k\in\mathbb{N}}$ be a Cauchy sequence in $X^{\wideparen{n}}$. Fix $\varepsilon>0$, then there exist $N\in\mathbb{N}$ such that $k,m> N$ 
implies $\Vert x^{(k)}-x^{(m)}\Vert_{X^{\wideparen{n}}}<\varepsilon$. From proposition \ref{PrNormVsSQNorm} it follows that $\Vert x_i^{(k)}-x_i^{(m)}\Vert_{\wideparen{n}}<\varepsilon$ for $i\in\mathbb{N}_n$. Hence the sequences 
$(x_i^{(k)})_{k\in\mathbb{N}}$ for $i\in\mathbb{N}_n$ are Cauchy sequences. Since $X^{\wideparen{1}}$ is complete, then there exist limits $x_i=\lim\limits_{k\to\infty}x_i^{(k)}$. Consider column $x=(x_i)_{i\in\mathbb{N}_n}\in  X^{\wideparen{n}}$. 
Again from proposition \ref{PrNormVsSQNorm} we have
$$
\lim_{k\to\infty}\Vert x^{(k)}-x\Vert_{\wideparen{n}}\leq\sum\limits_{i=1}^n\lim\limits_{k\to\infty}\Vert x_i^{(k)}-x_i\Vert_{\wideparen{1}}=0
$$
Thus, any Cauchy sequence $(x^{(k)})_{k\in\mathbb{N}}\subset X^{\wideparen{n}}$ is convergent, hence $X^{\wideparen{n}}$ is a Banach space. Conversely, assume $X^{\wideparen{n}}$ is a Banach space. Let $(x^{(k)})_{k\in\mathbb{N}}$ 
be a Cauchy sequence in $X^{\wideparen{1}}$. Fix $\varepsilon>0$, then there exist $N\in\mathbb{N}$, such that $k,m> N$ implies $\Vert x^{(k)}-x^{(m)}\Vert_{\wideparen{1}}<\varepsilon$.  Consider sequence 
$(\tilde{x}^{(k)})_{k\in\mathbb{N}}$ in $X^{\wideparen{n}}$ such that $\tilde{x}_i^{(k)}=x^{(k)}\delta_{1,i}$ for $i\in\mathbb{N}_n$. Then from proposition \ref{PrRedundantAxiom} we see that 
$\Vert \tilde{x}^{(k)}-\tilde{x}^{(m)}\Vert_{\wideparen{n}}=\Vert x^{(k)}-x^{(m)}\Vert_{\wideparen{1}}<\varepsilon$. Since $X^{\wideparen{n}}$ is complete, there exist the limit $\tilde{x}\in X^{\wideparen{n}}$. From proposition 
\ref{PrNormVsSQNorm} it follows that
$$
\lim\limits_{k\to\infty}\Vert x^{(k)}-\tilde{x}_1\Vert_{\wideparen{1}}
=\lim\limits_{k\to\infty}\Vert (\tilde{x}^{(k)}-\tilde{x})_1\Vert_{\wideparen{1}}
\leq\lim\limits_{k\to\infty}\Vert \tilde{x}^{(k)}-\tilde{x}\Vert_{\wideparen{n}}=0
$$
Thus any Cauchy sequence $(x^{(k)})_{k\in\mathbb{N}}\subset X^{\wideparen{1}}$ is convergent, hence $X^{\wideparen{1}}$ is a Banach space. 
\end{proof}

\begin{theorem}\label{ThSQCompl}
Let $X$ be operator sequence space, $\overline{X}$ be completion of $X^{\wideparen{1}}$, and $j_X:X\to \overline{X}$ be isometric inclusion with dense image. Then there is operator sequence space structure on $\overline{X}$, such that $j_X$ is sequentially isometric.
\end{theorem}
\begin{proof} Let $n\in\mathbb{N}$, $\overline{x}\in \overline{X}^{n}$. Then for each $i\in\mathbb{N}_n$ there exist a sequence $(x_i^{(k)})_{k\in\mathbb{N}}$ such that $\overline{x}_i=\lim\limits_{k\to\infty}j_X(x_i^{(k)})$. 
In particular, sequences $(x_i^{(k)})_{k\in\mathbb{N}}$ are Cauchy sequences in $X^{\wideparen{1}}$. For each $k\in\mathbb{N}$ consider $x^{(k)}=(x_i^{(k)})_{i\in\mathbb{N}_n}\in X^{\wideparen{n}}$. By definition we put
$$
\Vert\overline{x}\Vert_{\wideparen{n}}=\lim\limits_{k\to\infty}\Vert x^{(k)}\Vert_{\wideparen{n}}
$$
We will show, that this is well defined norm on $X^{\wideparen{n}}$ and what is more this family of norms defines operator sequence space structure on $\overline{X}$. Fix $\varepsilon>0$, since $(x_i^{(k)})_{k\in\mathbb{N}}$ are Cauchy sequences, then there exist $N_i\in\mathbb{N}$ for $i\in\mathbb{N}_n$ such that $k,m>N_i$ implies 
$\Vert x_i^{(k)}-x_i^{(m)}\Vert_{\wideparen{1}}<\varepsilon$. Consider $N=\max\limits_{i\in\mathbb{N}_n}N_i$, then from proposition \ref{PrNormVsSQNorm} for $k,m>N$ we get
$$
\left|\Vert x^{(k)}\Vert_{\wideparen{n}}-\Vert x^{(m)}\Vert_{\wideparen{n}}\right|\leq\Vert x^{(k)}-x^{(m)}\Vert\leq\sum\limits_{i=1}^n\Vert x_i^{(k)}-x_i^{(m)}\Vert_{\wideparen{1}}<n\varepsilon
$$
Thus the sequence $(\Vert x^{(k)}\Vert)_{k\in\mathbb{N}}$ is a Cauchy sequence and its limit in definition of $\Vert \overline{x}\Vert_{\wideparen{n}}$ does exists. Now we will show this limit does not depend on the choice of the sequence.  
Let $(x''^{(k)})_{k\in\mathbb{N}}$, $(x'^{(k)})_{k\in\mathbb{N}}$ be two such sequences in $X^{\wideparen{n}}$, such that $\overline{x}_i=\lim\limits_{k\to\infty} j_X(x_i'^{(k)})=\lim\limits_{k\to\infty} j_X(x_i'^{(k)})$ for all $i\in\mathbb{N}_n$. Then from proposition \ref{PrNormVsSQNorm}, we have 

$$
\left|\lim\limits_{k\to\infty}\Vert x''^{(k)}\Vert_{\wideparen{n}}-\lim\limits_{k\to\infty}\Vert x''^{(k)}\Vert_{\wideparen{n}}\right|\leq
\lim\limits_{k\to\infty}\Vert x''^{(k)}-x'^{(k)}\Vert_{\wideparen{n}}\leq
\sum\limits_{i=1}^n\lim\limits_{k\to\infty}\Vert x_i''^{(k)}-x_i'^{(k)}\Vert_{\wideparen{1}}\\
=\sum\limits_{i=1}^n0=0
$$
Hence this limits are equal and $\Vert \overline{x}\Vert_{\wideparen{n}}$ is well defined. Let $\overline{x}'\in X^{\wideparen{n}}$, $\overline{x}''\in X^{\wideparen{m}}$ and $\alpha\in M_{l,n}$, then
$$
\Vert\alpha\overline{x}'\Vert_{\wideparen{l}}
=\lim\limits_{k\to\infty}\Vert\alpha x'^{(k)}\Vert_{\wideparen{l}}
\leq\Vert\alpha\Vert\lim\limits_{k\to\infty}\Vert x'^{(k)}\Vert_{\wideparen{n}}
=\Vert\alpha\Vert\Vert\overline{x}'\Vert_{\wideparen{n}}
$$
$$
\left\Vert\begin{pmatrix} \overline{x}'\\ \overline{x}''\end{pmatrix}\right\Vert_{\wideparen{n+m}}^2
=\lim\limits_{k\to\infty}\left\Vert\begin{pmatrix} x'^{(k)}\\ x''^{(k)}\end{pmatrix}\right\Vert_{\wideparen{n+m}}^2
\leq\lim\limits_{k\to\infty}(\Vert x'^{(k)}\Vert_{\wideparen{n}}^2+\Vert x''^{(k)}\Vert_{\wideparen{m}}^2)
=\Vert\overline{x}'\Vert_{\wideparen{n}}^2+\Vert\overline{x}''\Vert_{\wideparen{m}}^2
$$
From proposition \ref{PrSQAxiomRed} we see that functions in question defines operator sequence space structure on $\overline{X}$. For all $x\in X^{\wideparen{n}}$ consider stationary   sequence $(j_X^{\wideparen{n}}(x))_{k\in\mathbb{N}}$, then
$$
\Vert j_X^{\wideparen{n}}(x)\Vert_{\wideparen{n}}
=\lim\limits_{k\to\infty}\Vert x^{(k)}\Vert_{\wideparen{n}}
=\Vert x\Vert_{\wideparen{n}}
$$
So $j_X$ is sequentially isometric.
\end{proof}

\begin{proposition}\label{PrExtLinOpByCont} Let $X$ and $Y$ be operator sequence spaces and $\varphi\in\mathcal{SB}(X,Y)$. Then there exist unique $\overline{\varphi}\in\mathcal{SB}(\overline{X},\overline{Y})$ extending $\varphi$ and what is more $\Vert \overline{\varphi}\Vert_{sb}=\Vert \varphi\Vert_{sb}$
\end{proposition}
\begin{proof}
It is well known that there exist unique extension $\overline{\varphi}\in\mathcal{B}(\overline{X},\overline{Y})$. For a given $x\in X^{\wideparen{n}}$ choose any sequence $(x^{(k)})_{k\in\mathbb{N}}\subset X^{\wideparen{n}}$ such that $\overline{x}=\lim\limits_{k\to\infty} j_X(x^{(k)})$. Then
$$
\Vert\overline{\varphi}^{\wideparen{n}}(\overline{x})\Vert_{\wideparen{n}}
=\lim\limits_{k\to\infty}\Vert \varphi^{\wideparen{n}}(x^{(k)})\Vert_{\wideparen{n}}
\leq\Vert \varphi\Vert_{sb}\lim\limits_{k\to\infty}\Vert x^{(k)}\Vert_{\wideparen{n}}
=\Vert \varphi\Vert_{sb}\Vert \overline{x}\Vert_{\wideparen{n}}
$$
Similarly, for any $x\in X^{\wideparen{n}}$ we have
$$
\Vert \varphi^{\wideparen{n}}(x)\Vert_{\wideparen{n}}
=\Vert\overline{\varphi}^{\wideparen{n}}(j_X^{\wideparen{n}}(x))\Vert_{\wideparen{n}}
=\Vert\overline{\varphi}\Vert_{sb}\Vert j_X^{\wideparen{n}}(x))\Vert_{\wideparen{n}}
=\Vert\overline{\varphi}\Vert_{sb}\Vert x\Vert_{\wideparen{n}}
$$
Since $n\in\mathbb{N}$ is arbitrary, then $\Vert \overline{\varphi}\Vert_{sb}=\Vert \varphi\Vert_{sb}$ and in particular $\overline{\varphi}\in\mathcal{SB}(\overline{X},\overline{Y})$
\end{proof}

\begin{proposition}\label{PrExtBilOpByCont} Let $X$, $Y$ and $Z$ be operator sequence spaces and $\mathcal{R}\in\mathcal{SB}(X\times Y,Z)$. Then there exist unique $\overline{\mathcal{R}}\in\mathcal{SB}(\overline{X}\times\overline{Y},\overline{Z})$ extending $\mathcal{R}$ and what is more $\Vert\overline{\mathcal{R}}\Vert_{sb}=\Vert\mathcal{R}\Vert_{sb}$.
\end{proposition}
\begin{proof} From proposition 1.9 \cite{DefFloTensNorOpId} we know that there exist unique bounded bilinear extension $\overline{\mathcal{R}}\in\mathcal{B}(\overline{X}\times\overline{Y},\overline{Z})$. For a given $\overline{x}\in \overline{X}^{\wideparen{n}}$, $\overline{y}\in \overline{Y}^{\wideparen{m}}$ choose any sequences $(x^{(k)})_{k\in\mathbb{N}}\subset X^{\wideparen{n}}$ and $(y^{(k)})_{k\in\mathbb{N}}\subset Y^{\wideparen{m}}$ such that $\overline{x}=\lim\limits_{k\to\infty} j_X^{\wideparen{n}}(x^{(k)})$ and $\overline{y}=\lim\limits_{k\to\infty} j_Y^{\wideparen{m}}(y^{(k)})$. Then $\overline{\mathcal{R}}^{\wideparen{n\times m}}(\overline{x},\overline{y})=\lim\limits_{k\to\infty}\mathcal{R}^{\wideparen{n\times m}}(x^{(k)}, y^{(k)})$ and
$$
\Vert\overline{\mathcal{R}}^{\wideparen{n\times m}}(\overline{x},\overline{y})\Vert_{\wideparen{n\times m}}
=\lim\limits_{k\to\infty}\Vert \mathcal{R}^{\wideparen{n\times m}}(x^{(k)}, y^{(k)})\Vert_{\wideparen{n\times m}}
\leq\Vert\mathcal{R}\Vert_{sb}\lim\limits_{k\to\infty}\Vert x^{(k)}\Vert_{\wideparen{n}}\Vert y^{(k)}\Vert_{\wideparen{m}}
=\Vert\mathcal{R}\Vert_{sb}\Vert\overline{x}\Vert_{\wideparen{n}}\Vert \overline{y}\Vert_{\wideparen{m}}
$$
Similarly for any $x\in X^{\wideparen{n}}$, $y\in Y^{\wideparen{m}}$ we have
$$
\Vert\mathcal{R}^{\wideparen{n\times m}}(x,y)\Vert_{\wideparen{n\times m}}
=\Vert\overline{\mathcal{R}}^{\wideparen{n\times m}}(j_X^{\wideparen{n}}(x),j_Y^{\wideparen{m}}(y))\Vert_{\wideparen{n\times m}}
\leq\Vert\overline{\mathcal{R}}\Vert_{sb}\Vert j_X^{\wideparen{n}}(x)\Vert_{\wideparen{n}}\Vert j_Y^{\wideparen{m}}(y)\Vert_{\wideparen{m}}
=\Vert\overline{\mathcal{R}}\Vert_{sb}\Vert x\Vert_{\wideparen{n}}\Vert y\Vert_{\wideparen{m}}
$$
Since $n,m\in\mathbb{N}$ are arbitrary, then $\Vert\overline{\mathcal{R}}\Vert_{sb}=\Vert\mathcal{R}\Vert_{sb}$ and in particular $\overline{\mathcal{R}}\in\mathcal{SB}(\overline{X}\times\overline{Y},\overline{Z})$
\end{proof}

Now we can enlarge the list of our main categories with $SQBan$ and $SQBan_1$. Their definitions are similar to definitions of $SQNor$ and $SQNor_1$.

\subsection{Duality theory for operator sequence spaces}

\begin{definition}[\cite{LamOpFolgen}, 1.3.8]\label{DeffSQDual} 
Let $X$ be operator sequence space, then by definition its sequential dual space is the space $X^\triangle := \mathcal{SB}(X, \mathbb{C})$. Note that here we consider $\mathbb{C}$ with standard operator sequence space structure from example \ref{ExHilSQ}. 
\end{definition}

\begin{proposition}[\cite{LamOpFolgen}, 1.3.9]\label{PrEveryLinFuncIsSQBounded}
Let $X$ be operator sequence space, and $f\in X^\triangle$. Then for all $n\in\mathbb{N}$ holds $\Vert f^{\wideparen{n}}\Vert=\Vert f\Vert$, and as the consequence $\Vert f\Vert_{sb}=\Vert f\Vert$.
\end{proposition}

\begin{proposition}[\cite{LamOpFolgen}, 1.3.9]\label{PrSQNormsViaDuality}
Let $X$ be operator sequence space, then $\mathcal{D}_{X,X^*}$ is sequentially isometric from the left and from the right. What is more for all $n\in\mathbb{N}$, $x\in X^{\wideparen{n}}$ and $f\in (X^\triangle)^{\wideparen{n}}$ we have
$$
\Vert x\Vert_{\wideparen{n}}
=\Vert x\Vert_{\wideparen{n}}^{\mathcal{D}_{X^*,X}}
\qquad\qquad
\Vert f\Vert_{\wideparen{n}}
=\Vert f\Vert_{\wideparen{n}}^{\mathcal{D}_{X,X^*}}
$$
As the consequence we get that natural embedding into the second dual
$$
\iota_X:X\to X^{\triangle\triangle}
$$
is sequentially isometric.
\end{proposition}
\begin{proof}
Since standard scalar duality is isometric from the left and from the right then using proposition \ref{PrEveryLinFuncIsSQBounded} we conclude that it is also sequentially isometric from the left and from the right. From proposition 1.3.12 \cite{LamOpFolgen} we know that
$$
\Vert x\Vert_{\wideparen{n}}
=\sup\{\Vert A(f)^{\wideparen{n}}(x)\Vert_{\wideparen{n\times n}}: f\in B_{(X^\triangle)^{\wideparen{n}}}\}
\qquad
\Vert f\Vert_{\wideparen{n}}
=\sup\{\Vert A(f)^{\wideparen{n}}(x)\Vert_{\wideparen{n\times n}}:x\in B_{X^{\wideparen{n}}}\}
$$
Now the desired equalities follow from identity $\mathcal{D}_{X,X^*}^{\wideparen{n\times n}}(x,f)=A((\mathcal{D}_{X,X^*}^{X^*})^{\wideparen{n}}(f))^{\wideparen{n}}(x)
=A(f)^{\wideparen{n}}(x)$.
Thus we see that original operator sequence space structures of $X$ and $X^\triangle$ coincide with the ones induced by bilinear operators $\mathcal{D}_{X^*,X}$ and $\mathcal{D}_{X,X^*}$. Hence, using that standard scalar duality is sequentially isometric from the right, we can apply proposition  \ref{PrFreezIsomSQIsom} to get that operator $\mathcal{D}_{X^*,X}^X$ is a sequential isometry. It is remains to note that $\iota_X=\mathcal{D}_{X^*,X}^X$. 
\end{proof}

\begin{remark}\label{RemSqReflexiv} We will say that $X$ is sequentially reflexive if $\iota_X$ is sequential isometric isomorphism. By proposition \ref{PrSQNormsViaDuality} operator $\iota_X$ is always sequentially isometric, so $\iota_X$ is a sequential isometric isomorphism if and only if it is surjective, which is equivalent to the usual reflexivity.
\end{remark}

\begin{proposition}\label{PrFreezDualityGetSQIsom}
Let $X$ ($Y$) be operator sequence space and $Y$ ($X$) be a normed space. Assume we are given a scalar duality $\mathcal{D}:X\times Y\to\mathbb{C}$ such that $\mathcal{D}^Y$ (${}^X\mathcal{D}$) 
are isometric isomorphisms, then if we consider $Y$ ($X$) with induced operator sequence space structure, then $\mathcal{D}^Y$ (${}^X\mathcal{D}$) would become sequentially isometric isomorphism.
\end{proposition}
\begin{proof}
We will consider the case when $\mathcal{D}^Y$ is an isometric isomorphism, for the remaining case all arguments are the same. Let $n\in\mathbb{N}$. By proposition \ref{PrEveryLinFuncIsSQBounded} bilinear operator $\mathcal{D}$ is sequentially isometric from the right. Then by proposition \ref{PrFreezIsomSQIsom} the linear operator $\mathcal{D}^Y$ is sequentially isometric, but it is also bijective, because $\mathcal{D}^Y$ is bijective. Therefore $\mathcal{D}^Y$ is sequentially isometric isomorphism. 
\end{proof}

\subsection{Duality theory for operators between operator sequence spaces}

\begin{proposition}[\cite{LamOpFolgen}, 1.3.14]\label{PrDualSBOp}
Let $X$, $Y$ be operator sequence spaces and $\varphi\in \mathcal{SB}(X,Y)$. Then $\varphi^\triangle \in\mathcal{SB}(Y^\triangle ,X^\triangle )$ and for all $n\in\mathbb{N}$ holds 
$\Vert(\varphi^\triangle )^{\wideparen{n}}\Vert=\Vert\varphi^{\wideparen{n}}\Vert$. As the consequence, $\Vert\varphi^\triangle \Vert_{sb}=\Vert\varphi\Vert_{sb}$.
\end{proposition}

\begin{corollary}\label{CorDualFunc}
From proposition \ref{PrDualSBOp} it follows that we have four well defined versions of functor ${}^\triangle$. They are of the form ${}^\triangle:\mathcal{K}\to\mathcal{K}$, where $\mathcal{K}\in\{SQNor,SQNor_1,SQBan,SQBan_1\}$. 
\end{corollary}

Further we will prove several technical lemmas necessary for description of duality for sequentially bounded operators.

\begin{definition}[\cite{LamOpFolgen}, 1.3.15]\label{DefT2n}
Let $X$ be operator sequence space and $n\in\mathbb{N}$, then by $t_2^n(X)$, we denote the normed space $X^n$ with the norm
$$
\Vert x\Vert_{t_2^n(X)}:=\inf\left\{\Vert\tilde{\alpha}\Vert_{hs}\Vert \tilde{x}\Vert_{\wideparen{k}}:x=\tilde{\alpha} \tilde{x}\right\}
$$
where $\tilde{\alpha}\in M_{n,k}$, $x\in X^k$ and $k\in\mathbb{N}$. If $Y$ is a operator sequence space and $\varphi\in\mathcal{SB}(X,Y)$, then by $t_2^n(\varphi)$ we will denote the linear operator
$$
t_2^n(\varphi): t_2^n(X)\to t_2^n(Y): x\mapsto \varphi^{\wideparen{n}}(x)
$$
\end{definition}

\begin{proposition}\label{PrT2nNormProperty}
Let $X$ be a operator sequence space $n\in\mathbb{N}$, then
$$
\Vert x\Vert_{t_2^n(X)}=\inf\left\{\Vert\alpha'\Vert_{hs}\Vert x'\Vert_{\wideparen{k}}:x=\alpha'x'\right\}
$$
where $\alpha'\in M_{n,n}$ is an invertible matrix and $x'\in X^{n}$.
\end{proposition}
\begin{proof}
Define right hand side of the equality to be proved by $\Vert x\Vert_{t_2^n(X)}'$. Fix $\varepsilon>0$, then there exist $\tilde{\alpha}\in M_{n,k}$ and $\tilde{x}\in X^{k}$, $k\in\mathbb{N}$ such that 
$x=\tilde{\alpha}\tilde{x}$ and $\Vert\tilde{\alpha}\Vert_{hs}\Vert\tilde{x}\Vert_{\wideparen{k}}<\Vert x\Vert_{t_2^n(X)}+\varepsilon$. Consider polar decomposition 
$\tilde{\alpha}=|\tilde{\alpha}^*| \rho$ of matrix $\tilde{\alpha}$. Let $p$ be orthogonal projection on $\operatorname{Im}(|\tilde{\alpha}^*|)^\perp$. Then for all $\delta\in\mathbb{R}$ the matrix 
$\alpha_\delta'=|\tilde{\alpha}^*|+\delta p$ is invertible because $\operatorname{Ker}(\alpha_\delta')=\{0\}$. Since $\alpha'_0=|\tilde{\alpha}|$ and the function $\Vert\alpha_\delta'\Vert_{hs}$ is continuous for $\delta\in\mathbb{R}$, then there exist such $\delta_0$ that 
$\Vert\alpha_{\delta_0}'\Vert_{hs}<\Vert|\tilde{\alpha}^*|\Vert_{hs}+\varepsilon\Vert \tilde{x}\Vert_{\wideparen{k}}^{-1}=\Vert\tilde{\alpha}\Vert_{hs}+\varepsilon\Vert \tilde{x}\Vert_{\wideparen{k}}^{-1}$. 
Denote $\alpha'=\alpha_{\delta_0}'\in M_{n,n}$ and $x'=\rho\tilde{x}\in Y^n$, then 
$$
\alpha'x'
=(|\tilde{\alpha}^*|+\delta_0 p)\rho \tilde{x}
=|\tilde{\alpha}^*|\rho \tilde{x}+\delta_0 p\rho \tilde{x}
=\tilde{\alpha}\tilde{x}
$$
By construction of polar decomposition $\Vert \rho\Vert\leq 1$ hence using definition of $\Vert x\Vert_{t_2^n(X)}'$ we get
$$
\Vert x\Vert_{t_2^n(X)}'\leq
\Vert\alpha'\Vert_{hs}\Vert x'\Vert_{\wideparen{n}}
\leq (\Vert\tilde{\alpha}\Vert_{hs}+\varepsilon\Vert \tilde{x}\Vert_{\wideparen{k}})\Vert \rho\Vert\Vert\tilde{x}\Vert_{\wideparen{n}}
\leq\Vert\tilde{\alpha}\Vert_{hs}\Vert\tilde{x}\Vert_{\wideparen{k}}+\varepsilon
\leq \Vert x\Vert_{t_2^n(X)}+2\varepsilon
$$
Since $\varepsilon>0$ is arbitrary, then $\Vert x\Vert_{t_2^n(X)}'\leq\Vert x\Vert_{t_2^n(X)}$. The reverse inequality is obvious, so $\Vert x\Vert_{t_2^n(X)}=\Vert x\Vert_{t_2^n(X)}'$.
\end{proof}

\begin{proposition}\label{PrT2nOfOpIsWellDef}
Let $X$, $Y$ be operator sequence spaces, $\varphi\in\mathcal{SB}(X,Y)$ and $n,k\in\mathbb{N}$. Then 
\newline
1) for all $\alpha\in M_{n,k}$ and $x\in t_2^k(X)$ holds $t_2^n(\varphi)(\alpha x)=\alpha t_2^k(\varphi)(x)$
\newline
2) $t_2^n(\varphi)\in\mathcal{B}(t_2^n(X),t_2^n(Y))$, and $\Vert t_2^n(\varphi)\Vert\leq\Vert\varphi^{\wideparen{n}}\Vert$
\newline
3) if $\varphi^{\wideparen{n}}$ (strictly) $c$-topologically surjective, then $t_2^n(\varphi)$ is also (strictly) $c$-topologically surjective
\newline
4)  if $\varphi^{\wideparen{n}}$ $c$-topologically injective, then $t_2^n(\varphi)$ is also $c$-topologically injective
\end{proposition}
\begin{proof}
1) Since $t_2^n(\varphi)=\varphi^{\wideparen{n}}$ as linear maps, then the result follows from paragraph 4 of proposition \ref{PrSimplAmplProps}. 
\newline
2) Let $x\in t_2^n(X)$ and $x=\alpha'x'$, where $\alpha\in M_{n,n}$ is an invertible matrix and $x'\in X^{n}$, then $t_2^n(\varphi)(x)=\alpha't_2^n(\varphi)(x')=\alpha'\varphi^{\wideparen{n}}(x')$. Hence from the definition of the norm on $t_2^n(Y)$ it follows
$$
\Vert t_2^n(\varphi)(x)\Vert_{t_2^n(Y)}
\leq\Vert\alpha'\Vert_{hs}\Vert\varphi^{\wideparen{n}}(x')\Vert_{\wideparen{n}}
\leq\Vert\alpha'\Vert_{hs}\Vert\varphi^{\wideparen{n}}\Vert\Vert x'\vert_{\wideparen{n}}
$$
Now take infimum over all representations of $x$ described above, then by proposition \ref{PrT2nNormProperty} we have
$$
\Vert t_2^n(\varphi)(x)\Vert_{t_2^n(Y)}\leq\Vert\varphi^{\wideparen{n}}\Vert\Vert x\Vert_{t_2^n(X)}
$$
Therefore $\Vert t_2^n(\varphi)\Vert\leq\Vert\varphi^{\wideparen{n}}\Vert$ and $t_2^n(\varphi)\in\mathcal{B}(t_2^n(X),t_2^n(Y))$.
\newline
3) Assume $\varphi^{\wideparen{n}}$ is $c$-topologically surjective. Let $y\in t_2^n(Y)$ and $y=\alpha' y'$, where $\alpha'\in M_{n,n}$ is an invertible matrix, $y'\in Y^n$. Let $c<c''<c'$. Since $\varphi^{\wideparen{n}}$ is $c$-topologically surjective, then there exist $x'\in X^n$ such that $\varphi^{\wideparen{n}}(x')=y'$ and $\Vert x'\Vert_{\wideparen{n}}< c''\Vert y'\Vert_{\wideparen{n}}$. Consider 
$x:=\alpha'x'$, then $t_2^n(\varphi)(x)=\alpha't_2^n(\varphi)(x')=\alpha'\varphi^{\wideparen{n}}(x')=\alpha' y'=y$. From definition of the norm on $t_2^n(X)$ we have
$$
\Vert x\Vert_{t_2^n(X)}
\leq\Vert\alpha'\Vert_{hs}\Vert x'\Vert_{\wideparen{n}}
\leq\Vert\alpha'\Vert_{hs} c''\Vert y'\Vert_{\wideparen{n}}
$$
Now take infimum over all representation of $y$ described above, then proposition \ref{PrT2nNormProperty} gives $\Vert x\Vert_{t_2^n(X)}\leq c''\Vert y\Vert_{t_2^n(Y)}<c'\Vert y\Vert_{t_2^n(Y)}$
Thus, for all $y\in t_2^n(Y)$ and $c'>c$ there exist $x\in t_2^n(X)$ such that $t_2^n(\varphi)(x)=y$ and $\Vert x\Vert_{t_2^n(X)}< c'\Vert y\Vert_{t_2^n(Y)}$. Therefore $t_2^n(\varphi)$ is 
$c$-topologically surjective.
\newline
Assume $\varphi^{\wideparen{n}}$ is strictly $c$-topologically surjective. Let $y\in t_2^n(Y)$ and $y=\alpha' y'$, where $\alpha'\in M_{n,n}$ is an invertible matrix, $y'\in Y^n$. Since $\varphi^{\wideparen{n}}$ is  $c$-topologically surjective, then there exist $x'\in X^n$ such that $\varphi^{\wideparen{n}}(x')=y'$ and $\Vert x'\Vert_{\wideparen{n}}\leq c\Vert y'\Vert_{\wideparen{n}}$. Consider $x:=\alpha'x'$, then 
$t_2^n(\varphi)(x)=\alpha't_2^n(\varphi)(x')=\alpha'\varphi^{\wideparen{n}}(x')=\alpha' y'=y$. From the definition of the norm on $t_2^n(X)$ we have
$$
\Vert x\Vert_{t_2^n(X)}
\leq\Vert\alpha'\Vert_{hs}\Vert x'\Vert_{\wideparen{n}}
\leq\Vert\alpha'\Vert_{hs} c\Vert y'\Vert_{\wideparen{n}}
$$
Now take infimum over all representations of $y$ described above, then proposition \ref{PrT2nNormProperty} gives $\Vert x\Vert_{t_2^n(X)}\leq c\Vert y\Vert_{t_2^n(Y)}$
Thus, for all $y\in t_2^n(Y)$ there exist $x\in t_2^n(X)$ such that $t_2^n(\varphi)(x)=y$ and $\Vert x\Vert_{t_2^n(X)}\leq c\Vert y\Vert_{t_2^n(Y)}$. Therefore $t_2^n(\varphi)$ strictly $c$-topologically surjective.
\newline
4) Assume $x\in t_2^n(X)$, then denote $y:=t_2^n(\varphi)(x)$. Consider representation $y=\alpha' y'$, where $\alpha'\in M_{n,n}$ is an invertible matrix and $y'\in Y^n$. Then 
$y'=(\alpha')^{-1}y=(\alpha')^{-1}t_2^n(\varphi)(x)=t_2^n(\varphi)((\alpha')^{-1}x)\in\operatorname{Im}(t_n^2(\varphi))
$. Since $\varphi^{\wideparen{n}}$ is $c$-topologically injective, then it is injective, so for $y'\in \operatorname{Im}(t_2^n(\varphi))$ there exist $x'\in X^n$ such that 
$y'=t_2^n(\varphi)(x')=\varphi^{\wideparen{n}}(x')$. Since $\varphi^{\wideparen{n}}$ is $c$-topologically injective, then $\Vert x'\Vert_{\wideparen{n}}\leq c\Vert y'\Vert$. From the definition of the norm on $t_2^n(X)$ we have
$$
\Vert x\Vert_{t_2^n(X)}\leq\Vert\alpha'\Vert_{hs}\Vert x'\Vert_{\wideparen{n}}\leq c\Vert\alpha'\Vert_{hs}\Vert y'\Vert_{\wideparen{n}}
$$
Now take infimum over all representations of $y$ described above, then proposition \ref{PrT2nNormProperty} gives $\Vert x\Vert_{t_2^n(X)}\leq c\Vert y\Vert_{t_2^n(Y)}=c\Vert t_2^n(\varphi)(x)\Vert_{t_2^n(Y)}$. 
Thus for all $x\in t_2^n(X)$ holds $\Vert t_2^n(\varphi)(x)\Vert_{t_2^n(Y)}\geq c^{-1}\Vert x\Vert_{t_2^n(X)}$. Therefore $t_2^n(\varphi)$ is $c$-topologically injective.
\end{proof}

\begin{proposition}[\cite{LamOpFolgen}, 1.3.16]\label{PrT2nTraingDuality}
Let $X$ be operator sequence space and $n\in\mathbb{N}$. Then we have isometric isomorphisms
$$
\alpha_X^n:t_2^n(X^\triangle)\to (X^{\wideparen{n}})^*: f\mapsto\left(x\mapsto\sum\limits_{i=1}^n f_i(x_i)\right)
\qquad
\beta_X^n:(X^\triangle)^{\wideparen{n}}\to t_2^n(X)^*:f\mapsto\left(x\mapsto\sum\limits_{i=1}^n f_i(x_i)\right)
$$
\end{proposition}

\begin{proposition}\label{PrTwoTypesDualOpEquiv}
Let $X$, $Y$ be operator sequence spaces, $\varphi\in \mathcal{SB}(X,Y)$ and $n\in\mathbb{N}$, then 
\newline
1) $(\varphi^\triangle)^{\wideparen{n}}$ is $c$-topologically (surjective) injective $\Longleftrightarrow$ $t_2^n(\varphi)^*$ is $c$-topologically (surjective) injective
\newline
2) $t_2^n(\varphi^\triangle)$ is $c$-topologically (surjective) injective $\Longleftrightarrow$ $(\varphi^{\wideparen{n}})^*$is $c$-topologically (surjective) injective
\newline
3) $\Vert (\varphi^\triangle)^{\wideparen{n}}\Vert=\Vert t_2^n(\varphi)^*\Vert$ and $\Vert t_2^n(\varphi^\triangle)\Vert=\Vert (\varphi^{\wideparen{n}})^*\Vert$ and $\Vert t_2^n(\varphi)\Vert=\Vert\varphi^{\wideparen{n}}\Vert$
\end{proposition}
\begin{proof}
Let $g\in (Y^\triangle)^{\wideparen{n}}$ and $x\in t_2^n(X)$, then
$$
(\alpha_X^n(\varphi^\triangle)^{\wideparen{n}})(g)(x)
=\alpha_X^n((\varphi^\triangle)^{\wideparen{n}}(g))(x)
=\sum\limits_{k=1}^n (\varphi^\triangle)^{\wideparen{n}}(g)_k(x_k)
=\sum\limits_{k=1}^n (\varphi^\triangle)(g_k)(x_k)
=\sum\limits_{k=1}^n g_k(\varphi(x_k))
$$
$$
(t_2^n(\varphi)^* \alpha_Y^n)(g)(x)
=t_2^n(\varphi)^*(\alpha_Y^n(g))(x)
=\alpha_Y^n(g)(t_2^n(\varphi)(x))
=\sum\limits_{k=1}^n g_k(t_2^n(\varphi)(x)_k)
=\sum\limits_{k=1}^n g_k(\varphi(x_k))
$$
Since $g$ and $x$ are arbitrary, then $\alpha_X^n(\varphi^\triangle)^{\wideparen{n}}=t_2^n(\varphi)^* \alpha_Y^n$. As $\alpha_Y^n$ and $\alpha_X^n$ are isometric isomorphisms we get that 1) holds and $\Vert (\varphi^\triangle)^{\wideparen{n}}\Vert=\Vert t_2^n(\varphi)^*\Vert$.
Let $g\in t_2^n(Y^\triangle)$ and $x\in X^{\wideparen{n}}$, then
$$
(\beta_X^n t_2^n(\varphi^\triangle))(g)(x)
=\beta_X^n(t_2^n(\varphi^\triangle)(g))(x)
=\sum\limits_{k=1}^n t_2^n(\varphi^\triangle)(g)_k(x_k)
=\sum\limits_{k=1}^n (\varphi^\triangle)(g_k)(x_k)
=\sum\limits_{k=1}^n g_k(\varphi(x_k))
$$
$$
((\varphi^{\wideparen{n}})^*\beta_Y^n)(g)(x)
=(\varphi^{\wideparen{n}})^*(\beta_Y^n(g))(x)
=\beta_Y^n(g)(\varphi^{\wideparen{n}})(x))
=\sum\limits_{k=1}^n g_k(\varphi^{\wideparen{n}})(x)_k)
=\sum\limits_{k=1}^n g_k(\varphi(x_k))
$$
Since $g$ and $x$ are arbitrary, then $\beta_X^n t_2^n(\varphi^\triangle)=(\varphi^{\wideparen{n}})^*\beta_Y^n$. As $\beta_Y^n$ and $\beta_X^n$ are isometric isomorphisms we get that 2) holds and $\Vert t_2^n(\varphi^\triangle)\Vert=\Vert (\varphi^{\wideparen{n}})^*\Vert$.

Finally, from propositions \ref{PrT2nOfOpIsWellDef}, \ref{PrDualSBOp} we have inequalities $\Vert t_2^n(\varphi)\Vert\leq\Vert\varphi^{\wideparen{n}}\Vert=\Vert(\varphi^\triangle)^{\wideparen{n}}\Vert=\Vert t_2^n(\varphi)^*\Vert=\Vert t_2^n(\varphi)\Vert$, so $\Vert t_2^n(\varphi)\Vert=\Vert\varphi^{\wideparen{n}}\Vert$.
\end{proof}

\begin{theorem}\label{ThDualSQOps}
Let $X$, $Y$ be operator sequence spaces and $\varphi\in\mathcal{SB}(X,Y)$, then
\newline
1) $\varphi$ (strictly) sequentially $c$-topologically surjective $\Longrightarrow$
$ \varphi^\triangle$ sequentially $c$-topologically injective
\newline
2) $ \varphi$ sequentially $c$-topologically injective $\Longrightarrow$
$ \varphi^\triangle$ strictly sequentially $c$-topologically surjective
\newline
3) $\varphi^\triangle$ (strictly) sequentially $c$-topologically surjective $\Longrightarrow$
$ \varphi$ sequentially $c$-topologically injective
\newline
4) $ \varphi^\triangle$ sequentially $c$-topologically injective and $X$ is complete $\Longrightarrow$
$ \varphi$ sequentially $c$-topologically surjective
\newline
5) $\varphi$ sequentially coisometric $\Longrightarrow$ 
$\varphi^\triangle$ sequentially isometric, if $X$ is complete, then the reverse implication is also true.
\newline
6) $ \varphi$ sequentially isometric $\Longleftrightarrow$
$\varphi^\triangle$ sequentially strictly coisometric
\end{theorem}
\begin{proof}
For each $n\in\mathbb{N}$ we have the following chain of implications
\newline
\begin{tabular}{llllll}
$\varphi^{\wideparen{n}}$ & $c$-topologically injective & $\implies$ & $t_2^n(\varphi)$                    & $c$-topologically injective       &\ref{PrT2nOfOpIsWellDef}\\
                        &                              & $\implies$ & $t_2^n(\varphi)^*$                    & strictly $c$-topologically surjective      &\ref{PrDualOps}\\
                        &                              & $\implies$ & $(\varphi^\triangle)^{\wideparen{n}}$ & strictly $c$-topologically surjective &\ref{PrTwoTypesDualOpEquiv}\\
                        &                              & $\implies$ & $t_2^n(\varphi^\triangle)$            & strictly $c$-topologically surjective &\ref{PrT2nOfOpIsWellDef}\\
                        &                              & $\implies$ & $(\varphi^{\wideparen{n}})^*$         & strictly $c$-topologically surjective &\ref{PrTwoTypesDualOpEquiv}\\
                        &                              & $\implies$ & $\varphi^{\wideparen{n}}$             & $c$-topologically injective       &\ref{PrDualOps}\\
\end{tabular}
\newline
So we get 2) and 3). Again for each $n\in\mathbb{N}$ we have the following chain of implications
\newline
\begin{tabular}{llclll}
$\varphi^{\wideparen{n}}$ & (strictly) $c$-topologically surjective & $\implies$ & $t_2^n(\varphi)$                    & $c$-topologically surjective     &\ref{PrT2nOfOpIsWellDef}\\
                        &                               & $\implies$ & $t_2^n(\varphi)^*$                  & $c$-topologically injective      &\ref{PrDualOps}\\
                        &                               & $\implies$ & $(\varphi^\triangle)^{\wideparen{n}}$ & $c$-topologically injective &\ref{PrTwoTypesDualOpEquiv}\\
                        &                               & $\implies$ & $t_2^n(\varphi^\triangle)$          & $c$-topologically injective &\ref{PrT2nOfOpIsWellDef}\\
                        &                               & $\implies$ & $(\varphi^{\wideparen{n}})^*$         & $c$-topologically injective &\ref{PrTwoTypesDualOpEquiv}\\
                        &                               & $\overset{\mbox{if $X$ is complete}}{\implies}$ & $\varphi^{\wideparen{n}}$             & $c$-topologically surjective     &\ref{PrDualOps}\\
\end{tabular}
\newline
So we get 1) and 4). Paragraphs 5)---6) are direct consequences of 1)---4) with $c=1$ if one takes into account that $\varphi$ is sequentially contractive  if and only if  $\varphi^\triangle$ is sequentially contractive (see proposition \ref{PrDualSBOp}).
\end{proof}

\subsection{Weak topologies for operator sequence spaces}

\begin{definition}\label{DefSQDconv} Let $\mathcal{D}:X\times Y\to Z$ be a vector duality between operator sequence spaces $X$, $Y$ and $Z$. We say that a net $(y_\nu)_{\nu\in N}\subset Y^{\wideparen{n}}$ sequentially $\mathcal{D}$-converges to $y\in Y^{\wideparen{n}}$ if it $\mathcal{D}^{\wideparen{m\times n}}$-converges for each $m\in\mathbb{N}$. Topology generated by this type of convergence we will denote by $\sigma_{\mathcal{D}}^{\widehat{n}}(Y,X)$.
\end{definition}

\begin{proposition}\label{PrDConvEquivCoordwsConv} Let $\mathcal{D}:X\times Y\to Z$ be a vector duality between operator sequence spaces $X$, $Y$ and $Z$, then the following are equivlent
\newline
1) net $(y_\nu)_{\nu\in N}\subset Y^{\wideparen{n}}$ sequentially $\mathcal{D}$-converges to $y\in Y^{\wideparen{n}}$
\newline
2) for each $i\in\mathbb{N}_n$ the net $( (y_\nu)_i)_{\nu\in N}\subset Y^{\wideparen{1}}$ $\mathcal{D}$-converges to $y_i\in Y^{\wideparen{1}}$.
\end{proposition}
\begin{proof}
1)$\implies$ 2) Note that for all $i\in\mathbb{N}_n$ and $x\in X^{\wideparen{1}}$ we have $\mathcal{D}(x,(y_\nu)_i-y_i)=(\mathcal{D}^{\wideparen{1\times n}}(x,y_\nu-y))_i$. Using proposition \ref{PrNormVsSQNorm}, we get
$$
\lim\limits_{\nu}\Vert \mathcal{D}(x,(y_\nu)_i-y_i)\Vert_{\wideparen{1}}
\leq\lim\limits_{\nu}\Vert \mathcal{D}^{\wideparen{1\times n}}(x,y_\nu-y)\Vert_{\wideparen{1\times n}}=0
$$
so $((y_\nu)_i)_{\nu\in N}$ $\mathcal{D}$-converges to $y_i$.

2)$\implies$ 1) Again from proposition \ref{PrNormVsSQNorm} for all $m\in\mathbb{N}$ and $x\in X^{\wideparen{m}}$ we get
$$
\lim\limits_{\nu}\Vert\mathcal{D}^{\wideparen{m\times n}}(x,y_\nu-y)\Vert_{\wideparen{m\times n}}
\leq\lim\limits_{\nu}\sum\limits_{j=1}^{m}\sum\limits_{i=1}^n\Vert\mathcal{D}(x_j,(y_\nu)_i-y_i)\Vert_{\wideparen{1}}
=\sum\limits_{j=1}^{m}\sum\limits_{i=1}^n\lim\limits_{\nu}\Vert\mathcal{D}(x_j,(y_\nu)_i-y_i)\Vert_{\wideparen{1}}=0
$$
so $(y_\nu)_{\nu\in N}$ sequentially $\mathcal{D}$-converges to $y$.
\end{proof}

\begin{proposition}\label{PrDContEquivCoordwsCont}
Let $\mathcal{D}_1:X_1\times Y_1\to Z_1$ and $\mathcal{D}_2:X_2\times Y_2\to Z_2$ be vector dualities between operator sequence spaces and $\varphi:Y_1\to Y_2$ be a linear operator, then $\varphi$ is $\sigma_{\mathcal{D}_1}(Y_1^{\wideparen{1}},X_1^{\wideparen{1}})$-$\sigma_{\mathcal{D}_2}(Y_2^{\wideparen{1}},X_2^{\wideparen{1}})$ continuous if and only if $\varphi^{\wideparen{n}}$ is $\sigma_{\mathcal{D}_1}^{\wideparen{n}}(Y_1,X_1)$-$\sigma_{\mathcal{D}_2}^{\wideparen{n}}(Y_2,X_2)$ continuous.
\end{proposition}
\begin{proof}
1)$\implies$ 2) Assume the net $(y_\nu)_{\nu\in N}\subset Y^{\wideparen{n}}$ sequentially $\mathcal{D}_1$-converges to $y\in Y^{\wideparen{n}}$. Then by proposition \ref{PrDConvEquivCoordwsConv} for each $i\in\mathbb{N}_n$ the net $((y_\nu)_i)_{\nu\in N}$ $\mathcal{D}_1$-converges to $y_i$. From assumption on $\varphi$ we get that the net $(\varphi((y_\nu)_i))_{\nu\in N}$ $\mathcal{D}_2$-converges to $\varphi(y_i)$ for each $i\in\mathbb{N}_n$. Again by the same proposition this means that the net $(\varphi^{\wideparen{n}}(y_\nu))_{\nu\in N}$ sequentially $\mathcal{D}_2$-converges to $\varphi^{\wideparen{n}}(y)$. Since the net $(y_\nu)_{\nu\in N}$ is arbitrary, then $\varphi^{\widehat{n}}$ is $\sigma_{\mathcal{D}_1}^{\wideparen{n}}(Y_1,X_1)$-$\sigma_{\mathcal{D}_2}^{\wideparen{n}}(Y_2,X_2)$ continuous.

2)$\implies$ 1) Assume the net $(y_\nu)_{\nu\in N}\subset Y^{\wideparen{1}}$ $\mathcal{D}_1$-converges to $y\in Y^{\wideparen{1}}$. Define $\widetilde{y}_\nu\in Y^{\wideparen{n}}$ such that $(\widetilde{y}_\nu)_1=y_\nu$ and $(\widetilde{y}_\nu)_i=0$ for $i\in\mathbb{N}_n\setminus\{1\}$. By proposition \ref{PrDConvEquivCoordwsConv} the net $(\widetilde{y}_\nu)$ sequentially $\mathcal{D}_1$-converges to $y\in Y^{\wideparen{n}}$ such that $y_1=y$ and $y_i=0$ for $i\in\mathbb{N}_n\setminus\{1\}$. From assumption on $\varphi^{\wideparen{n}}$ the net $(\varphi^{\wideparen{n}}(\widetilde{y}_\nu))_{\nu\in N}$ sequentially  $\mathcal{D}_2$-converges to $\varphi^{\wideparen{n}}(\widetilde{y})$. By proposition \ref{PrDConvEquivCoordwsConv} we get that the net $(\varphi^{\wideparen{1}}((\widetilde{y}_\nu)_1))_{\nu\in N}=(\varphi(y_\nu)_{\nu\in N}$ $\mathcal{D}_2$-converges to $\varphi^{\wideparen{1}}((\widetilde{y})_1)=\varphi(y)$. Since the net $(y_\nu)_{\nu\in N}$ is arbitrary, then $\varphi$ is $\sigma_{\mathcal{D}_1}(Y_1,X_1)$-$\sigma_{\mathcal{D}_2}(Y_2,X_2)$ continuous.
\end{proof}

\begin{definition}\label{DefWeakConvForSQSp} Let $X$ be an operator sequence space, then we define weak topology on $X^{\wideparen{n}}$ as $\sigma_{\mathcal{D}_{X^*,X}}^{\wideparen{n}}(X,X^*)$ topology and weak${}^*$ topology on $(X^\triangle)^{\wideparen{n}}$ as $\sigma_{\mathcal{D}_{X,X^*}}^{\wideparen{n}}(X^*,X)$ topology.
\end{definition}

In particular proposition \ref{PrDConvEquivCoordwsConv} tells us that weak and weak${}^*$ convergence are equivalent to weak and weak${}^*$ coordinatewise convergence respectively. From proposition \ref{PrDContEquivCoordwsCont} we get that continuity of the linear operator with respect to d if and only if erent weak topologies is equivalent to the continuity of the same type of amlified operator.

\begin{proposition}[\cite{LamOpFolgen}, 1.3.19]\label{PrDoubleDualIsom} Let $X$ be an operator sequence space, then there exist isometric isomorphism
$$
\widetilde{\iota_X}^n:(X^{\triangle\triangle})^{\wideparen{n}}\to(X^{\wideparen{n}})^{**}:\psi\mapsto\left(f\mapsto\sum\limits_{i=1}^n \psi_i((\alpha_X^n)^{-1}(f)_i)\right)
$$
which is also a weak${}^*$-weak${}^*$ homeomorphism.
\end{proposition}
\begin{proof} From proposition \ref{PrT2nTraingDuality} it follows that the desired isometric isomorphism is $\widetilde{\iota_X}^n:=((\alpha_X^n)^*)^{-1}\beta_{X^\triangle}^n$. Its action is given by the formula $\widetilde{\iota_X}^n(\psi)(f)=\sum_{i=1}^n \psi_i((\alpha_X^n)^{-1}(f)_i)$ where $\psi\in (X^{\triangle\triangle})$ and $f\in (X^{\wideparen{n}})^*$. Assume a net $(\psi_\nu)_{\nu\in N}\subset (X^{\triangle\triangle})^{\wideparen{n}}$ weak${}^*$ converges to $\psi\in (X^{\triangle\triangle})^{\wideparen{n}}$. By proposition \ref{PrDConvEquivCoordwsConv} this is equivalent to weak${}^*$ convergence of $((\psi_\nu)_i)_{\nu\in N}\subset X^{\triangle\triangle}$ to $\psi_i\in X^{\triangle\triangle}$ for each $i\in\mathbb{N}_n$. The latter is equivalent to convergence of the net $(\psi_\nu)_i(g))_{\nu\in N}$ to $(\psi)_i(g)$ for all $g\in X^\triangle$ and $i\in\mathbb{N}_n$. One can easily see such converrgence is possible  if and only if  the net $(\sum_{i=1}^n(\psi_\nu)_i(g_i))_{\nu\in N}$ converges to $\sum_{i=1}^n(\psi_\nu)_i(g_i)$ for all $g=(g_i)_{i\in\mathbb{N}_n}\in t_2^n(X^\triangle)$. This is equivalent to convergence of the net $(\widetilde{\iota_X}(\psi_\nu)(f))_{\nu\in N}$ to $\widetilde{\iota_X}(\psi)(f)$ for all $f\in (X^{\wideparen{n}})^*$. This means that the net $(\widetilde{\iota_X}^n(\psi_\nu))_{\nu\in N}\subset (X^{\wideparen{n}})^{**}$ weak${}^*$ converges to $\widetilde{\iota_X}^n(\psi)$. Since $\widetilde{\iota_X}^n$ is a bijections and all steps in the proof where equivalences then $\widetilde{\iota_X}^n$ is a weak${}^*$-weak${}^*$ homeomorphism.
\end{proof}

Now we are able to proof an operator sequence space analogue of Goldstine theorem.

\begin{proposition}\label{PrGoldsteinTh} Let $X$ be an operator sequence space, then $\iota_X^{\wideparen{n}}(B_{X^{\wideparen{n}}})$ is weak${}^*$ dense in $B_{(X^{\triangle\triangle})^{\wideparen{n}}}$. As the consequence $\iota_X^{\wideparen{n}}(X^{\wideparen{n}})$ is weak${}^*$ dense in $(X^{\triangle\triangle})^{\wideparen{n}}$.
\end{proposition} 
\begin{proof} For all $x\in X^{\wideparen{n}}$ and $f\in (X^*)^{\wideparen{n}}$ we have
$$
\widetilde{\iota_X}^n(\iota_X^{\wideparen{n}}(x))(f)
=\sum\limits_{i=1}^n\iota_X^{\wideparen{n}}(x)_i((\alpha_X^n)^{-1}(f)_i)
=\sum\limits_{i=1}^n\iota_X(x_i)((\alpha_X^n)^{-1}(f)_i)
=\sum\limits_{i=1}^n((\alpha_X^n)^{-1}(f)_i)(x_i)
$$
$$
=\alpha_X^n((\alpha_X^n)^{-1}(f))(x)=f(x)=\iota_{X^{\wideparen{n}}}(x)(f)
$$
so $\widetilde{\iota_X}^n\iota_X^{\wideparen{n}}=\iota_{X^{\wideparen{n}}}$ and since $\widetilde{\iota_X}^n$ is an isomorphism $\iota_X^{\wideparen{n}}=(\widetilde{\iota_X}^n)^{-1}\iota_{X^{\wideparen{n}}}$. 
By theorem 3.96 \cite{FabZizBanSpTh} we have that $\iota_{X^{\wideparen{n}}}(B_{X^{\wideparen{n}}})$ is weak${}^*$ dense in $B_{(X^{\wideparen{n}})^{**}}$. Since $\widetilde{\iota_X}$ is an isometric weak${}^*$-weak${}^*$ homeomorphism, then $\iota_X^{\wideparen{n}}(B_{X^{\wideparen{n}}})=(\widetilde{\iota_X})^{-1}\iota_{X^{\wideparen{n}}}(B_{X^{\wideparen{n}}})$ is weak${}^*$ dense in $(\widetilde{\iota_X})^{-1}(B_{(X^{\wideparen{n}})^{**}})=B_{(X^{\triangle\triangle})^{\wideparen{n}}}$.
\end{proof}

\begin{proposition}\label{PrWStarContExtSBOp} Let $X$ and $Y$ be two operator sequence spaces and $\varphi\in \mathcal{SB}(X,Y^\triangle)$. Then there exist unique weak${}^*$ continuous $\widetilde{\varphi}\in\mathcal{SB}(X^{\triangle\triangle},Y^\triangle)$ extending $\varphi$ and what is more $\Vert\widetilde{\varphi}\Vert_{sb}=\Vert\varphi\Vert$ 
\end{proposition}
\begin{proof} Denote $\widetilde{\varphi}=(\varphi^\triangle\iota_Y)^\triangle=\iota_{Y}^\triangle\varphi^{\triangle\triangle}$. It is weak${}^*$ continuous as a dual of bounded operator. One can easily check that $\varphi^{\triangle\triangle}\iota_X=\iota_{Y^\triangle}\varphi$ and $\iota_Y^\triangle\iota_{Y^\triangle}=1_{Y^\triangle}$, so $\widetilde{\varphi}\iota_X=\iota_{Y}^\triangle\varphi^{\triangle\triangle}\iota_X=\iota_{Y}^\triangle\iota_{Y^\triangle}\varphi=\varphi$ and we get that $\widetilde{\varphi}$ is weak${}^*$-continuous extension of $\varphi$. By proposition \ref{PrGoldsteinTh} we have that $\iota_X(X)$ is weak${}^*$ dense in $X^{\triangle\triangle}$. Hence $\widetilde{\varphi}$ is the unique extension of  $\varphi$. From propositions \ref{PrSimplAmplProps}, \ref{PrSQNormsViaDuality} and \ref{PrDualSBOp} we have
$$
\Vert\varphi\Vert_{sb}
=\Vert\widetilde{\varphi}\iota_X\Vert_{sb}
\leq\Vert\widetilde{\varphi}\Vert_{sb}\Vert\iota_X\Vert_{sb}
=\Vert\varphi\Vert_{sb}
$$
$$
\Vert\widetilde{\varphi}\Vert_{sb}
=\Vert\iota_{Y}^\triangle\varphi^{\triangle\triangle}\Vert_{sb}
\leq\Vert\iota_{Y}^\triangle\Vert_{sb}\Vert\varphi^{\triangle\triangle}\Vert_{sb}
=\Vert\iota_{Y}\Vert_{sb}\Vert\varphi\Vert_{sb}
=\Vert\varphi\Vert_{sb}
$$
So, $\Vert\widetilde{\varphi}\Vert_{sb}=\Vert\widetilde{\varphi}\Vert_{sb}$.
\end{proof}

\subsection{Subspaces and quotients of operator sequence spaces}

\begin{definition}[\cite{LamOpFolgen}, 1.1.26]\label{DefSQSubSpace}
Let $X$ be operator sequence space, $X_0$ subspace of $X$, then there is natural operator sequence space structure on $X_0$ defined by $X_0^{\wideparen{n}}=(X_0^n,\Vert\cdot\Vert_{\wideparen{n}})$.
\end{definition}

In this case the natural inclusion $i_{X_0,X}:X_0\to X$ obviously is sequentially isometric.

\begin{definition}[\cite{LamOpFolgen}, 1.1.27]\label{DefSQFactorSpace}
Let $X$ be operator sequence space, and $X_0$ subspace of $X$, then there is natural operator sequence space structure on $X / X_0$ defined by identifications 
$(X / X_0)^{\wideparen{n}} = X^{\wideparen{n}} / X_0^{\wideparen{n}}$, where $n\in\mathbb{N}$. 
\end{definition}

\begin{proposition}\label{PrFactorSQOp} 
Let $\varphi:X\to Y$ be a sequentially bounded operator between operator sequence spaces $E$ and $F$. Let $X_0$ and $Y_0$ be closed subspaces of $X$ and $Y$ respectively, such that $\varphi(X_0)\subset Y_0$, then there exist well defined sequentially bounded linear operator $\widehat{\varphi}:X/X_0\to Y/Y_0:x+X_0\mapsto T(x)+Y_0$ such that $\Vert\widehat{\varphi}^{\wideparen{n}}\Vert\leq\Vert \varphi^{\wideparen{n}}\Vert$ for all $n\in\mathbb{N}$ so $\Vert\widehat{\varphi}\Vert_{sb}\leq\Vert \varphi\Vert_{sb}$. Moreover,
\newline
1) if $X_0\subset \operatorname{Ker}(\varphi)\subset X_0$, then  $\Vert\widehat{\varphi}\Vert_{sb}=\Vert \varphi\Vert_{sb}$
\newline
2) if $\operatorname{Ker}(\varphi)= X_0$ and $\varphi$ is sequentially $c$-topologically surjective, then $\widehat{\varphi}$ is sequentially $c$-topologicaly injective isomorphism
\newline
3) if $\operatorname{Ker}(\varphi)= X_0$ and $\varphi$ is sequentially coisometric, then $\widehat{\varphi}$ is a sequential isometric isomorphism
\end{proposition}
\begin{proof}
Since for each $n\in\mathbb{N}$ we have $\varphi^{\widehat{n}}(X_0^{\wideparen{n}})\subset Y_0^{\wideparen{n}}$, then from proposition 1.5.2 \cite{HelFA} we get that $\Vert\widehat{\varphi}^{\wideparen{n}}\Vert\leq\Vert \varphi^{\wideparen{n}}\Vert$, so $\Vert\wideparen{\varphi}\Vert_{sb}\leq\Vert\varphi\Vert_{sb}$. 1) Clearly, $X_0^{\wideparen{n}}\subset\operatorname{Ker}(\varphi^{\wideparen{n}})$, so from proposition 1.5.3 \cite{HelFA} we get that $\Vert\widehat{\varphi}^{\wideparen{n}}\Vert=\Vert \varphi^{\wideparen{n}}\Vert$, so $\Vert\wideparen{\varphi}\Vert_{sb}=\Vert\varphi\Vert_{sb}$. 2) Similarly, $X_0^{\wideparen{n}}=\operatorname{Ker}(\varphi^{\wideparen{n}})$, so again from lemma A.2.1 \cite{EROpSp} we get that $\varphi^{\wideparen{n}}$ is $c$-topologically injective isomorphism. Hence $\varphi$ is sequentially $c$ topologically injective ismorphism. 3) By paragraph 1) we have $\Vert\widehat{\varphi}^{\wideparen{n}}\Vert\leq\Vert \varphi^{\wideparen{n}}\Vert\leq 1$. By paragraph 3) we get that $\widehat{\varphi}$ is a $1$-topologically injective isommorphism. Therefore $\varphi^{\wideparen{n}}$ is an isometric isomorphism for each $n\in\mathbb{N}$. Hence $\varphi$ is a sequentially isometric isomorphism.
\end{proof}

Applying this propostion to $\varphi=i_{X_0,X}$ we see that the natural quotient mapping $\pi_{X_0,X}$ is sequentially coisometric.

\begin{proposition}[\cite{LamOpFolgen}, 1.4.13]\label{PrDualForQuotsAndSubsp} Let $X$ be an operator sequence space and $X_0$ its closed subspace, then there exist sequentially isometric ismorphisms
$$
(X/X_0)^\triangle= X_0^\perp\qquad\qquad X^\triangle/X_0^\perp=X_0^\triangle
$$
\end{proposition}
\begin{proof} From theorem \ref{ThDualSQOps} operator $\pi_{X_0,X}^\triangle$ is sequentially isometric. Note that $\operatorname{Im}(\pi_{X_0,X}^\triangle)=\{f\circ\pi_{X_0,X}:f\in (X/X_0)^\triangle\}=\{g\in X^\triangle: g(X_0)=\{0\}\}=X_0^\perp$. Hence corestriction of $\pi_{X_0,X}^\triangle|^{X_0^\perp}:(X/X_0)^\triangle\to X_0^\perp$ is a sequentially isometric isomorphism. Again from theorem \ref{ThDualSQOps} operator $i_{X_0,X}^\triangle$ is sequentially coisometric. Note that $\operatorname{Ker}(i_{X_0,X}^\triangle)=\{f\in X^\triangle:f\circ i=0\}=\{f\in X^\triangle: f(X_0)=\{0\}\}=X_0^\perp$. Hence by proposition \ref{PrFactorSQOp} operator $\widehat{i_{X_0,X}^\triangle}:X^\triangle/X_0^\perp\to X_0^\triangle$ is a sequentially isometric isomorphism.
\end{proof}

\begin{proposition}\label{PrDualForWStarClQuotsAndSubsp} Let $X$ be a Banach operator sequence space and $W$ be weak${}^*$ closed subspace of $X^*$, then there exist sequentially isometric ismorphisms
$$
(X/W_\perp)^\triangle= W \qquad\qquad X^\triangle/W=W_\perp^\triangle
$$
which are weak${}^*$-weak${}^*$ homeomorphisms.
\end{proposition}
\begin{proof} Since $W$ is weak${}^*$ closed, then by theorem 4.7 \cite{RudinFA} we have  $(W_\perp)^\perp=W$. Now applying proposition \ref{PrDualForQuotsAndSubsp} to $X$ and $W_\perp$ we get the desired sequential isometric isomorphisms, they are $\pi_{W_\perp,X}^\triangle|^W$ and $\widehat{i_{W_\perp,X}^\triangle}$. As dual operators $\pi_{W_\perp,X}^\triangle$ and $i_{W_\perp,X}^\triangle$ are weak${}^*$-weak${}^*$ continuous. Clearly, $\pi_{W_\perp,X}^\triangle|^W$ is weak${}^*$-weak${}^*$ continuous as corestriction of such operator to the weak${}^*$ closed subspace $W$. By lemma A.2.4 \cite{BleOpAlgAndMods} operator $\widehat{i_{W_\perp,X}^\triangle}$ is also weak${}^*$-weak${}^*$ continuous. Thus  $\pi_{W_\perp,X}^\triangle|^W$ and $\widehat{i_{W_\perp,X}^\triangle}$ are weak${}^*$-weak${}^*$ continuous isometries, so by lemma A.2.5 \cite{BleOpAlgAndMods} they are weak${}^*$-weak${}^*$ homeomorphisms.
\end{proof}

\subsection{Direct sums of operator sequence spaces}

\begin{definition}[\cite{LamOpFolgen}, 1.1.28]\label{DefSQProd}
Let $\{X_\lambda: \lambda \in \Lambda\}$ be a family of operator sequence spaces. By definition their $\bigoplus_\infty$-sum is a operator sequence space structure on 
$\bigoplus_\infty\{X_\lambda^{\wideparen{1}}:\lambda\in \Lambda\}$, defined by identification
$$
\left(\bigoplus{}_\infty\{X_\lambda:\lambda \in \Lambda\}\right)^{\wideparen{n}}
=\bigoplus{}_\infty\{X_\lambda^{\wideparen{n}}:\lambda\in \Lambda\}
$$
Also, for a given $x\in \left(\bigoplus{}_\infty\{X_\lambda:\lambda \in \Lambda\}\right)^{\wideparen{n}}$ by $x_\lambda$ we denote element of $X_\lambda^{\wideparen{n}}$ such that $(x_\lambda)_i=(x_i)_\lambda$ for all $i\in\mathbb{N}_n$.
\end{definition}

\begin{proposition}\label{PrVectDualProdComp} Let $\{X_\lambda:\lambda\in\Lambda\}$ and $\{Z_\lambda:\lambda\in\Lambda\}$ be two families of operator sequence spaces and $Y$ be a operator sequence space. Let $\mathcal{D}_\lambda: Y\times Z_\lambda\to X_\lambda$ where $\lambda\in\Lambda$ is a family of vector dualities, then define vector duality
$$
\mathcal{D}:Y\times\bigoplus{}_\infty\{Z_\lambda:\lambda\in\Lambda\}\to\bigoplus{}_\infty\{X_\lambda:\lambda\in\Lambda\}:(y,z)\mapsto\oplus_\infty\{\mathcal{D}_\lambda(y,z_\lambda):\lambda\in\Lambda\}
$$
Assume $\mathcal{D}_\lambda^{Z_\lambda}$ is sequentially isometric for each $\lambda\in\Lambda$, then so does $\mathcal{D}^{\bigoplus{}_\infty\{Z_\lambda:\lambda\in\Lambda\}}$. If additionally $\mathcal{D}_\lambda^{Z_\lambda}$ is surjective for each $\lambda\in\Lambda$, then $\mathcal{D}^{\bigoplus{}_\infty\{Z_\lambda:\lambda\in\Lambda\}}$ is a sequential isometric isomorphism.
\end{proposition}
\begin{proof} Denote $Z=\bigoplus{}_\infty\{Z_\lambda:\lambda\in\Lambda\}$. Let $n\in\mathbb{N}$ and $z\in Z^{\wideparen{n}}$. Since $\mathcal{D}_\lambda^{Z_\lambda}$ is sequentially isometric, then
$$
\Vert z_\lambda\Vert_{\wideparen{n}}
=\Vert (\mathcal{D}_\lambda^{Z_\lambda})^{\wideparen{n}}(z_\lambda)\Vert_{\wideparen{n}}
=\sup\{\Vert \mathcal{D}_\lambda^{\wideparen{k\times n}}(y,z_\lambda)\Vert_{\wideparen{k\times n}}:k\in\mathbb{N},y\in B_{Y^{\wideparen{k}}}\}
$$
Now note that,
$$
\begin{aligned}
\Vert(\mathcal{D}^Z)^{\wideparen{n}}(z)\Vert_{\wideparen{n}}
&=\Vert A((\mathcal{D}^Z)^{\wideparen{n}}(z))\Vert_{sb}\\
&=\sup\{\Vert A((\mathcal{D}^Z)^{\wideparen{n}}(z))^{\wideparen{k}}(y)\Vert_{\wideparen{k\times n}}:k\in\mathbb{N},y\in B_{Y^{\wideparen{k}}}\}\\
&=\sup\{\Vert \mathcal{D}^{\wideparen{k\times n}}(y,z)\Vert_{\wideparen{k\times n}}:k\in\mathbb{N},y\in B_{Y^{\wideparen{k}}}\}\\
&=\sup\{\Vert \oplus_\infty\{\mathcal{D}_\lambda^{\wideparen{k\times n}}(y,z_\lambda):\lambda\in\Lambda\}\Vert_{\wideparen{k\times n}}:k\in\mathbb{N},y\in B_{Y^{\wideparen{k}}}\}\\
&=\sup\{\Vert \mathcal{D}_\lambda^{\wideparen{k\times n}}(y,z_\lambda)\Vert_{\wideparen{k\times n}}:k\in\mathbb{N},y\in B_{Y^{\wideparen{k}}},\lambda\in\Lambda\}\\
&=\sup\{\Vert z_\lambda\Vert_{\wideparen{n}}:\lambda\in\Lambda\}\\
&=\Vert z\Vert_{\wideparen{n}}
\end{aligned}
$$
Hence $\mathcal{D}^Z$ is a sequential isometry. Now consider second assumption. Define natural projections $p_\lambda:\bigoplus{}_\infty\{X_\lambda:\lambda\in\Lambda\}\to X_\lambda:x\mapsto x_\lambda$. Take any $\varphi\in\mathcal{SB}(Y,X)$, and define $\varphi_\lambda=p_\lambda\varphi$. For each $\lambda\in\Lambda$ we know that $\mathcal{D}_\lambda^{Z_\lambda}$ is surjective, so there is $z_\lambda\in Z_\lambda$ such taht $\mathcal{D}_\lambda^{Z_\lambda}(z_\lambda)=\varphi_\lambda$. Since $\mathcal{D}_\lambda^{Z_\lambda}$ is isometric, then $\Vert z_\lambda\Vert=\Vert p_\lambda\varphi\Vert\leq\Vert \varphi\Vert$ so $\sup\{\Vert z_\lambda\Vert:\lambda\in\Lambda\}<\infty$. Then we have well defined $z\in\bigoplus{}_\infty\{Z_\lambda:\lambda\in\Lambda\}$. Note that for all $y\in Y$ we have
$$
\mathcal{D}^{Z}(z)(y)
=\oplus_\infty\{\mathcal{D}_\lambda^{Z_\lambda}(z_\lambda)(y):\lambda\in\Lambda\}
=\oplus_\infty\{\varphi_\lambda(y):\lambda\in\Lambda\}
=\oplus_\infty\{p_\lambda\varphi(y):\lambda\in\Lambda\}
=\varphi(y)
$$
hence $\mathcal{D}^Z(z)=\varphi$. Since $\varphi$ is arbitrary, then $\mathcal{D}^Z$ is surjective, but it is also injective as any isometry. Hence $\mathcal{D}^Z$ and all its amplifications are bijective, but they are all isometric, therefore $\mathcal{D}^Z$ is a sequential isometric isomorphism.
\end{proof}

\begin{proposition}\label{PrSQProdUnivProp} Let $\{X_\lambda:\lambda\in \Lambda\}$ be a family of operator sequence spaces, then
\newline
1) there is a sequential isometric isomorphism
$$
\mathcal{SB}\left(Y,\bigoplus{}_\infty\{X_\lambda:\lambda\in\Lambda\}\right)
=\bigoplus{}_\infty\{\mathcal{SB}(Y,X_\lambda):\lambda\in\Lambda\}
$$
\newline
2) the operator sequence space $\bigoplus{}_\infty\{X_\lambda:\lambda\in\Lambda\}$ with natural projections $p_\lambda:\bigoplus{}_\infty\{X_\lambda:\lambda\in\Lambda\}\to X_\lambda$ is a categorical product in $SQNor_1$.
\end{proposition}
\begin{proof} 1) By proposition \ref{PrSQOpSqQuanIsEquivToStandard} vector dualities $\mathcal{E}_\lambda:Y\times\mathcal{SB}(Y,X_\lambda)\to X_\lambda:(y,\varphi)\mapsto \varphi(y)$ satisfy both assumptions of proposition \ref{PrVectDualProdComp}, hence $\mathcal{E}^{\bigoplus{}_\infty\{\mathcal{SB}(Y,X_\lambda):\lambda\in\Lambda\}}$ is a desired isometric isomorphism.
\newline
2) For all $n\in\mathbb{N}$ and $x\in \left(\bigoplus{}_\infty\{X_\lambda:\lambda\in\Lambda\}\right)^{\wideparen{n}}$ we have
$$
\Vert p_\lambda^{\wideparen{n}}(x)\Vert_{\wideparen{n}}
=\Vert (x_{i,\lambda})_{i\in\mathbb{N}_n}\Vert_{\wideparen{n}}
\leq\sup\{\Vert (x_{i,\lambda})_{i\in\mathbb{N}_n}\Vert_{\wideparen{n}}:\lambda\in \Lambda\}
=\Vert x\Vert_{\wideparen{n}}
$$
so $p_\lambda$ is sequentially bounded, and even sequentially contractive. Now consider any family of sequentially contractie operators $\{\varphi_\lambda\in\mathcal{SB}(Y,X_\lambda):\lambda\in\Lambda\}$. By previous paragraph for  $\varphi=\mathcal{E}^{\bigoplus{}_\infty\{\mathcal{SB}(Y,X_\lambda):\lambda\in\Lambda\}}(\oplus_\infty\{\varphi_\lambda:\lambda\in\Lambda\})$ we have $\Vert\varphi\Vert_{sb}=\sup\{\Vert\varphi_\lambda\Vert_{sb}:\lambda\in\Lambda\}\leq 1$. Moreover, for all $y\in Y$ we have
$$
p_\lambda\varphi(y)
=p_\lambda\mathcal{E}^{\bigoplus{}_\infty\{\mathcal{SB}(Y,X_\lambda):\lambda\in\Lambda\}}(\oplus_\infty\{\varphi_\lambda:\lambda\in\Lambda\})(y)
=p_\lambda(\oplus_\infty\{\varphi_\lambda(y):\lambda\in\Lambda\})=\varphi_\lambda(y)
$$
i.e. $p_\lambda\varphi=\varphi_\lambda$. Since $Y$ and the family $\{\varphi_\lambda:\lambda\in\Lambda\}$ are arbitrary, then $\bigoplus{}_\infty\{X_\lambda:\lambda\in\Lambda\}$ is indeed a product in $SQNor_1$.
\end{proof}

\begin{definition}\label{DefSQCoProd}
Let $\{X_\lambda: \lambda \in \Lambda\}$ be a family of operator sequence spaces. By definition their $\bigoplus_1^0$-sum is a operator sequence space structure on  
$\bigoplus_1^0\{X_\lambda^{\wideparen{1}}:\lambda\in \Lambda\}$, defined by embedding
$$
\bigoplus{}_1^0\{X_\lambda:\lambda \in \Lambda\}\hookrightarrow
\left(\bigoplus{}_\infty\{X_\lambda^\triangle:\lambda\in \Lambda\}\right)^\triangle
$$
\end{definition}

\begin{proposition}\label{PrVectDualCoProdComp} Let $\{X_\lambda:\lambda\in\Lambda\}$ and $\{Z_\lambda:\lambda\in\Lambda\}$ be two families of operator sequence spaces and $Y$ be a operator sequence space. Let $\mathcal{D}_\lambda: X_\lambda\times Z_\lambda\to Y$ where $\lambda\in\Lambda$ is a family of vector dualities, then define vector duality
$$
\mathcal{D}:\bigoplus{}_1^0\{X_\lambda:\lambda\in\Lambda\}\times\bigoplus{}_\infty\{Z_\lambda:\lambda\in\Lambda\}\to Y:(x,z)\mapsto\sum\limits_{\lambda\in\Lambda}\mathcal{D}_\lambda(x_\lambda,z_\lambda)
$$
Assume $\mathcal{D}_\lambda^{Z_\lambda}$ is sequentially isometric for each $\lambda\in\Lambda$, then so does $\mathcal{D}^{\bigoplus{}_\infty\{Z_\lambda:\lambda\in\Lambda\}}$. If additionally $\mathcal{D}_\lambda^{Z_\lambda}$ is surjective for each $\lambda\in\Lambda$, then $\mathcal{D}^{\bigoplus{}_\infty\{Z_\lambda:\lambda\in\Lambda\}}$ is a sequential isometric isomorphism.
\end{proposition}
\begin{proof} Denote $Z=\bigoplus{}_\infty\{Z_\lambda:\lambda\in\Lambda\}$ and $X=\bigoplus{}_1^0\{X_\lambda:\lambda\in\Lambda\}$. Let $n\in\mathbb{N}$ and $z\in Z^{\wideparen{n}}$. Since $\mathcal{D}_\lambda^{Z_\lambda}$ is sequentially isometric, then
$$
\Vert z_\lambda\Vert_{\wideparen{n}}
=\Vert (\mathcal{D}_\lambda^{Z_\lambda})^{\wideparen{n}}(z_\lambda)\Vert_{\wideparen{n}}
=\sup\{\Vert \mathcal{D}_\lambda^{\wideparen{k\times n}}(x_\lambda,z_\lambda)\Vert_{\wideparen{k\times n}}:k\in\mathbb{N},x_\lambda\in B_{X_\lambda^{\wideparen{k}}}\}
$$
Now note that,
$$
\begin{aligned}
\Vert(\mathcal{D}^Z)^{\wideparen{n}}(z)\Vert_{\wideparen{n}}
&=\Vert A((\mathcal{D}^Z)^{\wideparen{n}}(z))\Vert_{sb}\\
&=\sup\{\Vert A((\mathcal{D}^Z)^{\wideparen{n}}(z))^{\wideparen{k}}(x)\Vert_{\wideparen{k\times n}}:k\in\mathbb{N},x\in B_{X^{\wideparen{k}}}\}\\
&=\sup\{\Vert \mathcal{D}_{Y,Y^*}^{\wideparen{kn\times m}}(A((\mathcal{D}^Z)^{\wideparen{n}}(z))^{\wideparen{k}}(x),f)\Vert_{\wideparen{kn\times m}}:k\in\mathbb{N},x\in B_{X^{\wideparen{k}}},m\in\mathbb{N},f\in B_{(Y^\triangle)^{\wideparen{m}}}\}\\
\end{aligned}
$$
One can check that $\mathcal{D}_{Y,Y^*}(\mathcal{D}^Z(z)(x),f)=\mathcal{D}_{\bigoplus{}_1^0\{X_\lambda:\lambda\in\Lambda\},\bigoplus_\infty\{X_\lambda:\lambda\in\Lambda\}}(x,\oplus_\infty\{((\mathcal{D}_\lambda^{Z_\lambda})(z_\lambda))^*(f):\lambda\in\Lambda\})$, so applying proposition \ref{PrSQNormsViaDuality} we get
$$
\begin{aligned}
\Vert(\mathcal{D}^Z)^{\wideparen{n}}(z)\Vert_{\wideparen{n}}
&=\sup\{\Vert \mathcal{D}_{Y,Y^*}^{\wideparen{kn\times m}}(A((\mathcal{D}^Z)^{\wideparen{n}}(z))^{\wideparen{k}}(x),f)\Vert_{\wideparen{kn\times m}}:k\in\mathbb{N},x\in B_{X^{\wideparen{k}}},m\in\mathbb{N},f\in B_{(Y^\triangle)^{\wideparen{m}}}\}\\
&=\sup\{\Vert \mathcal{D}_{\bigoplus{}_1^0\{X_\lambda:\lambda\in\Lambda\},\bigoplus_\infty\{X_\lambda^*:\lambda\in\Lambda\}}^{\wideparen{k\times nm}}(x,\oplus_\infty\{A((({}^\triangle\cdot\mathcal{D}_\lambda^{Z_\lambda})^{\wideparen{n}}(z_\lambda))^{\wideparen{m}}(f):\lambda\in\Lambda\})\Vert_{\wideparen{k\times nm}}: \\
&\qquad\qquad k\in\mathbb{N},x\in B_{X^{\wideparen{k}}},m\in\mathbb{N},f\in B_{(Y^\triangle)^{\wideparen{m}}}\}\\
&=\sup\{\Vert \oplus_\infty\{A((({}^\triangle\cdot\mathcal{D}_\lambda^{Z_\lambda})^{\wideparen{n}}(z_\lambda))^{\wideparen{m}}(f):\lambda\in\Lambda\}\Vert_{\wideparen{m\times n}}: m\in\mathbb{N},f\in B_{(Y^\triangle)^{\wideparen{m}}}\}\\
&=\sup\{\Vert A((({}^\triangle\cdot\mathcal{D}_\lambda^{Z_\lambda})^{\wideparen{n}}(z_\lambda))^{\wideparen{m}}(f)\Vert_{\wideparen{m\times n}}: m\in\mathbb{N},f\in B_{(Y^\triangle)^{\wideparen{m}}},\lambda\in\Lambda\}\\
\end{aligned}
$$
Apply proposition \ref{PrSQNormsViaDuality} once again
$$
\begin{aligned}
\Vert(\mathcal{D}^Z)^{\wideparen{n}}(z)\Vert_{\wideparen{n}}
&=\sup\{\Vert A((({}^\triangle\cdot\mathcal{D}_\lambda^{Z_\lambda})^{\wideparen{n}}(z_\lambda))^{\wideparen{m}}(f)\Vert_{\wideparen{m\times n}}: m\in\mathbb{N},f\in B_{(Y^\triangle)^{\wideparen{m}}},\lambda\in\Lambda\}\\
&=\sup\{\Vert\mathcal{D}_{X_\lambda,X_\lambda^*}^{\wideparen{ln\times m}}(A(((\mathcal{D}_\lambda^{Z_\lambda})^{\wideparen{n}}(z_\lambda))^{\wideparen{l}}(x_\lambda),f)\Vert_{\wideparen{ln\times m}}: l\in\mathbb{N},x_\lambda\in B_{X_\lambda^{\wideparen{l}}},m\in\mathbb{N},\\
&\qquad\qquad f\in B_{(Y^\triangle)^{\wideparen{m}}},\lambda\in\Lambda\}\\
&=\sup\{\Vert A(((\mathcal{D}_\lambda^{Z_\lambda})^{\wideparen{n}}(z_\lambda))^{\wideparen{l}}(x_\lambda)\Vert_{\wideparen{l\times n}}: l\in\mathbb{N},x_\lambda\in B_{X_\lambda^{\wideparen{l}}},\lambda\in\Lambda\}\\
&=\sup\{\Vert (\mathcal{D}_\lambda^{Z_\lambda})^{\wideparen{n}}(z_\lambda)\Vert_{\wideparen{n}}: \lambda\in\Lambda\}\\
&=\sup\{\Vert z_\lambda \Vert_{\wideparen{n}}: \lambda\in\Lambda\}\\
&=\Vert z \Vert_{\wideparen{n}}
\end{aligned}
$$
Hence $\mathcal{D}^Z$ is a sequential isometry. Now consider second assumption. Define natural injections $i_\lambda:X_\lambda\to\bigoplus{}_1^0\{X_\lambda:\lambda\in\Lambda\}:x_\lambda\mapsto (\ldots,0,x_\lambda,0,\ldots)$. Take any $\varphi\in\mathcal{SB}(\bigoplus{}_1^0\{X_\lambda:\lambda\in\Lambda\})$, and define $\varphi_\lambda=\varphi i_\lambda$. For each $\lambda\in\Lambda$ we know that $\mathcal{D}_\lambda^{Z_\lambda}$ is surjective, so there is $z_\lambda\in Z_\lambda$ such that $\mathcal{D}_\lambda^{Z_\lambda}(z_\lambda)=\varphi_\lambda$. Since $\mathcal{D}_\lambda^{Z_\lambda}$ is isometric, then $\Vert z_\lambda\Vert=\Vert p_\lambda\varphi\Vert\leq\Vert \varphi\Vert$ so $\sup\{\Vert z_\lambda\Vert:\lambda\in\Lambda\}<\infty$. Then we have well defined $z\in\bigoplus{}_\infty\{Z_\lambda:\lambda\in\Lambda\}$. Note that for all $x\in \bigoplus{}_1^0\{X_\lambda:\lambda\in\Lambda\}$ we have
$$
\mathcal{D}^{Z}(z)(x)
=\sum_{\lambda\in\Lambda}\mathcal{D}_\lambda^{Z_\lambda}(z_\lambda)(x_\lambda)
=\sum_{\lambda\in\Lambda}\varphi_\lambda(x_\lambda)
=\sum_{\lambda\in\Lambda}\varphi i_\lambda(x_\lambda)
=\varphi\left(\sum_{\lambda\in\Lambda} i_\lambda(x_\lambda)\right)
=\varphi(x)
$$
hence $\mathcal{D}^Z(z)=\varphi$. Since $\varphi$ is arbitrary, then $\mathcal{D}^Z$ is surjective, but it is also injective as any isometry. Hence $\mathcal{D}^Z$ and all its amplifications are bijective, but they are all isometric, therefore $\mathcal{D}^Z$ is a sequential isometric isomorphism.
\end{proof}

\begin{proposition}\label{PrSQCoProdUnivProp} Let $\{X_\lambda:\lambda\in \Lambda\}$ be a family of operator sequence spaces, then
\newline
1) there is a sequential isometric isomorphism
$$
\mathcal{SB}\left(\bigoplus{}_1^0\{X_\lambda:\lambda\in\Lambda\},Y\right)
=\bigoplus{}_\infty\{\mathcal{SB}(X_\lambda,Y):\lambda\in\Lambda\}
$$
\newline
2) the operator sequence space $\bigoplus{}_1^0\{X_\lambda:\lambda\in\Lambda\}$ with natural injections $i_\lambda:X_\lambda\to\bigoplus{}_\infty\{X_\lambda:\lambda\in\Lambda\}$ is a categorical coproduct in $SQNor_1$.
\end{proposition}
\begin{proof} 1) By proposition \ref{PrSQOpSqQuanIsEquivToStandard} vector dualities $\mathcal{E}_\lambda:X_\lambda\times\mathcal{SB}(X_\lambda,Y)\to Y:(x_\lambda,\varphi)\mapsto \varphi(x_\lambda)$ satisfy both assumptions of proposition \ref{PrVectDualProdComp}, hence $\mathcal{E}^{\bigoplus{}_\infty\{\mathcal{SB}(X_\lambda,Y):\lambda\in\Lambda\}}$ is a desired isometric isomorphism.
\newline
2) For all $n\in\mathbb{N}$ and $x\in \left(\bigoplus{}_1^0\{X_\lambda:\lambda\in\Lambda\}\right)^{\wideparen{n}}$ holds
$$
\begin{aligned}
\Vert i_\lambda^{\wideparen{n}}(x)\Vert_{\wideparen{n}}
&=\sup\{\Vert\mathcal{D}_{\bigoplus{}_1^0\{X_\lambda:\lambda\in \Lambda\},\bigoplus{}_\infty\{X_\lambda^*:\lambda\in \Lambda\}}^{\wideparen{n\times n}}(i_\lambda^{\wideparen{n}}(x),f)\Vert_{\wideparen{n\times n}}: f\in B_{(\bigoplus{}_\infty\{X_\lambda^\triangle:\lambda\in \Lambda\})^{\wideparen{n}}}\}\\
&=\sup\{\Vert\mathcal{D}_{X_\lambda,X_\lambda^*}^{\wideparen{n\times n}}(\tilde{p}_\lambda^{\wideparen{n}}(f),x)\Vert_{\wideparen{n\times n}}: f\in B_{(\bigoplus{}_\infty\{X_\lambda^\triangle:\lambda\in \Lambda\})^{\wideparen{n}}}\}\\
&=\sup\{\Vert\mathcal{D}_{X_\lambda,X_\lambda^*}^{\wideparen{n\times n}}(f,x)\Vert_{\wideparen{n\times n}}: f\in B_{(X_\lambda^\triangle)^{\wideparen{n}}}\}\\
&=\Vert x\Vert_{\wideparen{n}}
\end{aligned}
$$ 
so $i_\lambda$ is sequentially bounded, and even sequentially isometric. Now consider any family of sequentially contractie operators $\{\varphi_\lambda\in\mathcal{SB}(X_\lambda,Y):\lambda\in\Lambda\}$. By previous paragraph for  $\varphi=\mathcal{E}^{\bigoplus{}_\infty\{\mathcal{SB}(X_\lambda,Y):\lambda\in\Lambda\}}(\oplus_\infty\{\varphi_\lambda:\lambda\in\Lambda\})$ we have $\Vert\varphi\Vert_{sb}=\sup\{\Vert\varphi_\lambda\Vert_{sb}:\lambda\in\Lambda\}\leq 1$. Moreover, for all $y\in Y$ we have
$$
\varphi i_\lambda(x_\lambda)
=\mathcal{E}^{\bigoplus{}_\infty\{\mathcal{SB}(Y,X_\lambda):\lambda\in\Lambda\}}(\oplus_\infty\{\varphi_\lambda:\lambda\in\Lambda\})(i_\lambda(x_\lambda))
=\sum\limits_{\lambda'\in\Lambda}\varphi_{\lambda'}(i_\lambda(x_\lambda))
=\varphi_\lambda(x_\lambda)
$$
i.e. $\varphi i_\lambda=\varphi_\lambda$. Since $Y$ and the family $\{\varphi_\lambda:\lambda\in\Lambda\}$ are arbitrary, then $\bigoplus{}_1^0\{X_\lambda:\lambda\in\Lambda\}$ is indeed a coproduct in $SQNor_1$.
\end{proof}

\begin{proposition}\label{PrDualOfCoprodIsProd}
Let $\{X_\lambda:\lambda\in \Lambda\}$ be a family of operator sequence spaces, then there exist sequentially isometric isomorphism
$$
\left(\bigoplus{}_1^0\{X_\lambda:\lambda\in \Lambda\}\right)^\triangle
=\bigoplus{}_\infty\{X_\lambda^\triangle:\lambda\in \Lambda\}
$$
\end{proposition}
\begin{proof}
The result follows from proposition \ref{PrSQCoProdUnivProp} with $Y=\mathbb{C}$.
\end{proof}

\begin{definition}\label{DefSQc0Sum}
Let $\{X_\lambda: \lambda \in \Lambda\}$ be a family of operator sequence space. By definition their $\bigoplus_0^0$-sum  is an operator sequence space structure on $\bigoplus_0^0\{X_\lambda^{\wideparen{1}}:\lambda\in \Lambda\}$, considered as subspace of operator sequence space $\bigoplus_\infty\{X_\lambda:\lambda\in \Lambda\}$.
\end{definition}

\begin{proposition}\label{PrDensSubsetOfSumOfDoubleDuals} Let $\{X_\lambda:\lambda\in\Lambda\}$ be a family of operator sequence spaces, then the set $\{\bigoplus{}_\infty\{\iota_{X_\lambda}^{\wideparen{n}}(x_\lambda):\lambda\in \Lambda\}:x\in B_{(\bigoplus{}_0^0\{X_\lambda:\lambda\in \Lambda\})^{\wideparen{n}}}\}$ is weak${}^*$ dense in $B_{(\bigoplus{}_\infty\{X_\lambda^{\triangle\triangle}:\lambda\in \Lambda\})^{\wideparen{n}}}$
\end{proposition}
\begin{proof}
Let $\psi\in (\bigoplus{}_\infty\{X_\lambda^{**}:\lambda\in \Lambda\})^{\wideparen{m}}$ with $\Vert\psi\Vert_{\wideparen{m}}\leq 1$. In particular $\Vert\psi_{i,\lambda}\Vert\leq 1$ for all $i\in\mathbb{N}_m$ and $\lambda\in\Lambda$. For any $\lambda\in\Lambda$ by theorem 3.96 \cite{FabZizBanSpTh} we have that $\iota(B_{X_\lambda})$ is weak${}^*$ dense in $X_\lambda^{**}$ so for each $i\in\mathbb{N}_m$ we have a net $(x_{\nu,i,\lambda}'':\nu\in N_{i,\lambda})\subset B_{X_\lambda}$ that is weak${}^*$ converges to $\psi_{i,\lambda}$. For each $i\in\mathbb{N}_m$ consider poset $N_i=\prod_{\lambda\in\Lambda}N_{i,\lambda}$ with standard product order, natural projections $\pi_{i,\lambda}:N_i\to N_{i,\lambda}$ and define a subnet $x_{\nu,i,\lambda}'=x_{\pi_{i,\lambda}(\nu),i,\lambda}''$ for all $\nu\in N_i$. So we get a net $(x_{\nu,i,\lambda}':\nu\in N_i)$ that is weak${}^*$ converges to $\psi_{i,\lambda}$. The latter is equivalent to the weak${}^*$ convergence of the net $(\bigoplus_\infty\{\iota_{X_\lambda}(x_{\nu,i,\lambda}'):\lambda\in\Lambda\}:\nu\in N_i)\subset B_{\bigoplus{}_\infty\{X_\lambda^{**}:\lambda\in \Lambda\}}$ to $\psi_i$. Again, consider poset $N=\prod_{i=1}^m N_i$ with standard product order, natural projections $\pi_i:N\to N_i$ and define a subnet $x_{\nu,i,\lambda}=x_{\pi_i(\nu),i,\lambda}'$ for all $\nu\in N$. Then we get a net $(\bigoplus_\infty\{\iota_{X_\lambda}(x_{\nu,i,\lambda}):\lambda\in\Lambda\}:\nu\in N)$ that weak${}^*$ converges to $\psi_i$. By propositon \ref{PrDConvEquivCoordwsConv} we get that the net $(\bigoplus_\infty\{\iota_{X_\lambda}(x_{\nu,\lambda}):\lambda\in\Lambda\}:\nu\in N)$ weak${}^*$ converges to $\psi$ and thanks to  the defiition of the norm in $\bigoplus_\infty$-sum this net is in the unit ball of $(\bigoplus{}_\infty\{X_\lambda^{\triangle\triangle}:\lambda\in \Lambda\})^{\wideparen{m}}$. The last is equivalent to the desired density result. 
\end{proof}

\begin{proposition}\label{PrDualOfc0SumIsCoProd}
Let $\{X_\lambda:\lambda\in \Lambda\}$ be a family of operator sequence spaces, then there exist sequentially isometric isomorphism
$$
\left(\bigoplus{}_0^0\{X_\lambda:\lambda\in \Lambda\}\right)^\triangle
=\bigoplus{}_1\{X_\lambda^\triangle:\lambda\in \Lambda\}
$$
\end{proposition}
\begin{proof}
For each $n\in\mathbb{N}$ and $f\in \left(\bigoplus{}_0^0\{X_\lambda^\triangle:\lambda\in \Lambda\}\right)^{\wideparen{n}}$ we have 
$$
\begin{aligned}
\Vert(\mathcal{D}&_{\bigoplus{}_0^0\{X_\lambda:\lambda\in \Lambda\},\bigoplus{}_1\{X_\lambda^*:\lambda\in \Lambda\}}^{\bigoplus{}_1\{X_\lambda^*:\lambda\in \Lambda\}})^{\wideparen{n}}(f)\Vert_{\wideparen{n}}=\\
&=\sup\{\Vert\mathcal{D}_{\bigoplus{}_0^0\{X_\lambda:\lambda\in \Lambda\},\bigoplus{}_1\{X_\lambda^*:\lambda\in \Lambda\}}^{m\times n}(x,f)\Vert_{\wideparen{m\times n}}:m\in\mathbb{N}, x\in B_{\bigoplus{}_0^0\{X_\lambda:\lambda\in \Lambda\}}\}\\
&=\sup\{\Vert\mathcal{D}_{\bigoplus{}_1\{X_\lambda^*:\lambda\in \Lambda\},\bigoplus{}_\infty\{X_\lambda^{**}:\lambda\in \Lambda\}}^{m\times n}(f,\oplus_\infty\{\iota_{X_\lambda}(x_\lambda):\lambda\in\Lambda\})\Vert_{\wideparen{m\times n}}:m\in\mathbb{N}, x\in B_{\bigoplus{}_0^0\{X_\lambda:\lambda\in \Lambda\}}\}\\
\end{aligned}
$$
Since tautologically $\mathcal{D}$ is weak${}^*$ continuous in the second variable, then from proposition \ref{PrDensSubsetOfSumOfDoubleDuals} we get
$$
\begin{aligned}
\Vert(&\mathcal{D}_{\bigoplus{}_0^0\{X_\lambda:\lambda\in \Lambda\},\bigoplus{}_1\{X_\lambda^*:\lambda\in \Lambda\}}^{\bigoplus{}_1\{X_\lambda^*:\lambda\in \Lambda\}})^{\wideparen{n}}(f)\Vert_{\wideparen{n}}=\\
&=\sup\{\Vert\mathcal{D}_{\bigoplus{}_1\{X_\lambda^*:\lambda\in \Lambda\},\bigoplus{}_\infty\{X_\lambda^{**}:\lambda\in \Lambda\}}^{m\times n}(f,\psi)\Vert_{\wideparen{m\times n}}:m\in\mathbb{N}, x\in B_{\bigoplus{}_\infty\{X_\lambda^{\triangle\triangle}:\lambda\in \Lambda\}}\}\\
&=\Vert f\Vert_{\wideparen{n}}
\end{aligned}
$$
Therefore $\mathcal{D}_{\bigoplus{}_0^0\{X_\lambda:\lambda\in \Lambda\},\bigoplus{}_1\{X_\lambda^*:\lambda\in \Lambda\}}^{\bigoplus{}_1\{X_\lambda^*:\lambda\in \Lambda\}}$ is a sequential isometry, but by proposition \ref{PrSumDuality} it is also bijective, hence this is the desired sequential isometric isomorphism.
\end{proof}

Similar results holds for Banach operator sequence spaces (just replace $\bigoplus{}_1^0$-sums and $\bigoplus{}_0^0$-sums with $\bigoplus{}_1$-sums and $\bigoplus{}_0$-sums).

Next proposition extensively uses terminology and results of \cite{BrownItoUniquePredual}.

\begin{proposition}\label{PrUniquePredualForCoproduct}
Let $\{X_\lambda:\lambda\in\Lambda\}$ be a family of reflexive operator sequence spaces, then $\bigoplus{}_\infty\{X_\lambda:\lambda\in\Lambda\}$ have unique (up to sequential isometry) Banach operator sequence space predual $\bigoplus{}_1\{X_\lambda^\triangle:\lambda\in\Lambda\}$
\end{proposition} 
\begin{proof} For each $\lambda\in\Lambda$ the space $X_\lambda$ is reflexive, so it belongs to the class $(L_0)$, so by theorem 1 \cite{BrownItoUniquePredual} the space $\bigoplus_0\{X_\lambda:\lambda\in\Lambda\}$ is in the class $(L_0)$. By remark after proposition 4 \cite{BrownItoUniquePredual} and propositions \ref{PrDualOfCoprodIsProd}, \ref{PrDualOfc0SumIsCoProd} we get that $\bigoplus_0\{X_\lambda:\lambda\in\Lambda\}^{**}=\bigoplus_\infty\{X_\lambda^{**}:\lambda\in\Lambda\}=\bigoplus_\infty\{X_\lambda:\lambda\in\Lambda\}$ have as Banach space unique up to isometric isomorphism predual Banach space $(\bigoplus_0\{X_\lambda:\lambda\in\Lambda\})^{*}=\bigoplus_1\{X_\lambda^{*}:\lambda\in\Lambda\}$. Since being operator seqence space predual is a stronger property than being Banach space predual, then the only candidate for operator sequence space predual of $\bigoplus{}_\infty\{X_\lambda:\lambda\in\Lambda\}$ is $\bigoplus{}_1\{X_\lambda^\triangle:\lambda\in\Lambda\}$. By remark \ref{RemSqReflexiv} the space $X_\lambda$ is sequentially reflexive for each $\lambda\in\Lambda$ and by proposition \ref{PrDualOfCoprodIsProd} we get
$$
\left(\bigoplus{}_1\{X_\lambda^\triangle:\lambda\in\Lambda\}\right)^\triangle
=\bigoplus{}_\infty\{X_\lambda^{\triangle\triangle}:\lambda\in\Lambda\}
=\bigoplus{}_\infty\{X_\lambda:\lambda\in\Lambda\}
$$
\end{proof}

\subsection{Minimal and maximal structure of operator sequence space}

\begin{definition}[\cite{LamOpFolgen}, 2.1.1]\label{DefSQMin} Minimal structure of operator sequence space $\min(E)$ for a normed space $E$ is given by identifications $\min(E)^{\wideparen{n}} = \mathcal{B}(l_2^n, E)$, so for each $x \in E^n$ we have
$$
\Vert x\Vert_{\wideparen{n}}=\sup\left\{\left\Vert\sum\limits_{i=1}^n \xi_i x_i\right\Vert:\xi\in B_{l_2^n}\right\}
$$
\end{definition}

\begin{proposition}[\cite{LamOpFolgen}, 2.1.4]\label{PrCharMinSQ}
Let $X$ be an operator sequence space, then the following are equivalent
\newline
1) $X=\min(X^{\wideparen{1}})$
\newline
2) for every operator sequence space $Y$ each bounded linear operator $\varphi:Y\to X$ is sequentially bounded and $\Vert\varphi\Vert_{sb}=\Vert\varphi\Vert$
\newline
3) for every operator sequence space $Y$ there is isometric isomorphism $\mathcal{SB}(Y,X)^{\wideparen{1}}=\mathcal{B}(Y^{\wideparen{1}},X^{\wideparen{1}})$
\end{proposition}

\begin{proposition}[\cite{LamOpFolgen}, 1.1.11, 2.1.5]\label{PrMinFucntor}
The map
$$
\begin{aligned}
\min : Nor_1 \to SQNor_1 : X&\mapsto \min(X)\\
\varphi&\mapsto\varphi
\end{aligned}
$$
is a covariant functor from category of normed spaces into the category of operator sequence spaces
\end{proposition}

Clearly, the following definition is a generalization of example \ref{ExT2nSQ}.

\begin{definition}[\cite{LamOpFolgen}, 2.1.7]\label{DefSQMax} Maximal structure of operator sequence space $\max(E)$ for a given normed space $E$ is given by family of norms
$$
\Vert x\Vert_{\wideparen{n}}=\inf\left\{\Vert\alpha\Vert_{M_{n,k}}\left(\sum\limits_{i=1}^k\Vert \tilde   x_i\Vert^2\right)^{1/2}:x=\alpha\tilde x\right\}
$$
where $x\in E^{\wideparen{n}}$, $\alpha\in M_{n,k}$, $\tilde{x}\in E^k$.
\end{definition}

\begin{proposition}[\cite{LamOpFolgen}, 2.1.9]\label{PrCharMaxSQ}
Let $X$ be an operator sequence space, then the following are equivalent
\newline
1) $X=\max(X^{\wideparen{1}})$
\newline
2) for every operator sequence space $Y$ each bounded linear operator $\varphi:X\to Y$ is sequentially bounded and $\Vert\varphi\Vert_{sb}=\Vert\varphi\Vert$
\newline
3) for every operator sequence space $Y$ there is isometric isomorphism $\mathcal{SB}(X,Y)^{\wideparen{1}}=\mathcal{B}(X^{\wideparen{1}},Y^{\wideparen{1}})$
\end{proposition}

\begin{proposition}[\cite{LamOpFolgen}, 1.1.11, 2.1.10]\label{PrMaxFucntor}
The map
$$
\begin{aligned}
\max : Nor_1 \to SQNor_1 : X&\mapsto \max(X)\\
\varphi&\mapsto\varphi
\end{aligned}
$$
is a covariant functor from the category of normed spaces into the category of operator sequence spaces.
\end{proposition}

\begin{proposition}\label{PrMinPreserveEmbedings} Let $\varphi:E\to F$ be bounded linear operator between normed spaces $E$ and $F$, then 
\newline
1) if $\varphi$ is $c$-topologically injective, then $\min(\varphi)$ is sequentially $c$-topologically injective
\newline
2) if $\varphi$ is isometric, then $\min(\varphi)$ is sequentially isometric
\end{proposition}
\begin{proof} 1) For each $n\in\mathbb{N}$ and $x\in \min(E)^{\wideparen{n}}$ we have
$$
\Vert \min(\varphi)^{\wideparen{n}}(x)\Vert_{\wideparen{n}}
=\sup\left\{\left\Vert\sum\limits_{i=1}^n\xi_i \varphi^{\wideparen{n}}(x)_i\right\Vert:\xi\in B_{l_2^n}\right\}
=\sup\left\{\left\Vert\sum\limits_{i=1}^n\xi_i \varphi(x_i)\right\Vert:\xi\in B_{l_2^n}\right\}
$$
$$
=\sup\left\{\left\Vert\varphi\left(\sum\limits_{i=1}^n\xi_i x_i\right)\right\Vert:\xi\in B_{l_2^n}\right\}
\geq c^{-1}\sup\left\{\left\Vert\sum\limits_{i=1}^n\xi_i x_i\right\Vert:\xi\in B_{l_2^n}\right\}
=c^{-1}\Vert x\Vert_{\wideparen{n}}
$$
Hence $\min(\varphi)$ is sequentially $c$-topologically injective.
\newline
2) By previous paragraph $\min(\varphi)$ is $1$-topologically injective. On the other hand, by proposition \ref{PrCharMinSQ} we have $\Vert\min(\varphi)\Vert_{sb}=\Vert\varphi\Vert=1$. Therefore $\min(\varphi)$ is sequentially isometric.
\end{proof}

\begin{proposition}\label{PrMinCommuteWithProd} Let $\{X_\lambda:\lambda\in\Lambda\}$ be a family of minimal operator sequence spaces, then $\bigoplus{}_\infty\{X_\lambda:\lambda\in\Lambda\}$ is also minimal.
\end{proposition} 
\begin{proof}
Let $Y$ be arbitrary operator sequence space, then from propositions \ref{PrSQProdUnivProp}, \ref{PrCharMinSQ} and \ref{PrCharMaxSQ} we have isometric identifications
$$
\mathcal{SB}\left(Y,\bigoplus{}_\infty\{X_\lambda:\lambda\in\Lambda\}\right)^{\wideparen{1}}
=\bigoplus{}_\infty\{\mathcal{SB}(Y,X_\lambda)^{\wideparen{1}}:\lambda\in\Lambda\}
=\bigoplus{}_\infty\{\mathcal{B}(Y^{\wideparen{1}},X_\lambda^{\wideparen{1}}):\lambda\in\Lambda\}
$$
$$
=\bigoplus{}_\infty\{\mathcal{SB}(\max(Y^{\wideparen{1}}),X_\lambda)^{\wideparen{1}}:\lambda\in\Lambda\}
=\mathcal{SB}\left(\max(Y^{\wideparen{1}}),\bigoplus{}_\infty\{X_\lambda:\lambda\in\Lambda\}\right)^{\wideparen{1}}
$$
$$
=\mathcal{B}\left(Y^{\wideparen{1}},\left(\bigoplus{}_\infty\{X_\lambda:\lambda\in\Lambda\}\right)^{\wideparen{1}}\right)
$$
Since $Y$ is arbitrary, from proposition \ref{PrCharMinSQ} we conclude that $\bigoplus{}_\infty\{X_\lambda:\lambda\in\Lambda\}$ have minimal operator sequence space structure.
\end{proof}

\begin{proposition}\label{PrCommCstarAlgIsMin} Let $A$ be a commutative $C^*$ algebra and $X$ be an operator sequence space, then every bounded linear operator $\varphi:X\to A$ is sequentially bounded with $\Vert\varphi\Vert_{sb}=\Vert\varphi\Vert$. As the consequence the standard operator sequence space structure of $A$ is minimal.
\end{proposition}
\begin{proof} As $A$ is a commutatitive $C^*$ algebra, by Gelfand-Naimark theorem 2.1.10 \cite{MurphCstarOpTh} we may assume that $A=C_0(\Omega)$. Using proposition \ref{PrCstarAlgSQ} for any $n\in\mathbb{N}$ and $x\in X^{\wideparen{n}}$ we have 
$$
\Vert\varphi^{\wideparen{n}}(x)\Vert_{\wideparen{n}}
=\Vert i_C(\varphi^{\wideparen{n}}(x))\Vert
=\sup\{\Vert i_C(\varphi^{\wideparen{n}}(x))(\omega)\Vert:\omega\in\Omega\}
=\sup\{\langle i_C(\varphi^{\wideparen{n}}(x))(\omega),\xi\rangle:\omega\in\Omega,\xi\in B_{\mathbb{C}^n}\}
$$
$$
=\sup\left\{\left|\sum_{i=1}^n \varphi(x_i)(\omega)\overline{\xi_i}\right|:\omega\in\Omega,\xi\in B_{\mathbb{C}^n}\right\}
=\sup\left\{\left| \varphi\left(\sum_{i=1}^n \overline{\xi_i} x_i\right)(\omega)\right|:\omega\in\Omega,\xi\in B_{\mathbb{C}^n}\right\}
$$
$$
=\sup\left\{\left\Vert \varphi\left(\sum_{i=1}^n \overline{\xi_i} x_i\right)\right\Vert:\xi\in B_{\mathbb{C}^n}\right\}
\leq\Vert\varphi\Vert\sup\left\{\left\Vert \sum_{i=1}^n \overline{\xi_i} x_i\right\Vert_{\wideparen{n}}:\xi\in B_{\mathbb{C}^n}\right\}
$$
$$
\leq\Vert\varphi\Vert\Vert x\Vert_{\wideparen{n}}\sup\{\Vert\operatorname{diag}_n(\overline{\xi_1},\ldots,\overline{\xi_n})\Vert:\xi\in B_{\mathbb{C}^n}\}
=\Vert\varphi\Vert\Vert x\Vert_{\wideparen{n}}\sup\left\{\max_{i\in\mathbb{N}_n}|\overline{\xi_i}|:\xi\in B_{\mathbb{C}^n}\right\}
\leq\Vert\varphi\Vert\Vert x\Vert_{\wideparen{n}}
$$
Therefore $\Vert\varphi\Vert_{sb}\leq\Vert\varphi\Vert$. Since we always have $\Vert\varphi\Vert\leq\Vert\varphi\Vert_{sb}$, then we get the desired equality. As operator sequence space $X$ is arbitrary, from proposition \ref{PrCharMinSQ} we see that $A$ have minimal operator sequence space structure.
\end{proof}

\begin{proposition}\label{PrMinIsSubspOfCommCstarAlg} Let $X$ be an operator sequence space, then $X$ is minimal if and only if there exist sequential isometry from $X$ into $C(\Omega)$ for some compact topological space $\Omega$.
\end{proposition}
\begin{proof} 
Assume $X$ have minimal structure. Consider natural isometry $i:X\to C(B_{X^*})$ (see A1 \cite{DefFloTensNorOpId}). By proposition \ref{PrMinPreserveEmbedings} we know that $\min(i):\min(X^{\wideparen{1}})\to\min(C(B_{X^*})^{\wideparen{1}})$ is sequentially isometric. By proposition \ref{PrCommCstarAlgIsMin} we have $\min(C(B_{X^*})^{\wideparen{1}})=C(B_{X^*})$ and by assumption $\min(X^{\wideparen{1}})=X$, so we get the desired sequential isometry $\min(i):X\to C(B_{X^*})$.

Conversely, assume we are given sequential isometry $i:X\to C(\Omega)$. Since $i^{\wideparen{1}}:X^{\wideparen{1}}\to C(\Omega)^{\wideparen{1}}$ is an isometry, by proposition \ref{PrMinPreserveEmbedings} we have sequential isometry $\min(i):\min(X^{\wideparen{1}})\to\min(C(\Omega)^{\wideparen{1}})$. By proposition \ref{PrCommCstarAlgIsMin} we have $\min(C(\Omega)^{\wideparen{1}})=C(\Omega)$, so we have one more sequential isometry $\min(i):\min(X^{\wideparen{1}})\to C(\Omega)$. Since $i=\min(i)$ as linear maps we conclude that $X=\min(X^{\wideparen{1}})$ 
\end{proof}

\begin{proposition}\label{PrMaxPreserveQuotients} Let $\varphi:E\to F$ be bounded linear operator between normed spaces $E$ and $F$, then
\newline
1) if $\varphi$ is $c$-topologically surjective, then $\max(\varphi)$ is sequentially $c$-topologically surjective
\newline
2) is $\varphi$ is coisometric, then $\max(\varphi)$ is sequentially coisometric
\end{proposition}
\begin{proof} 1) By lemma A.2.1 \cite{EROpSp} we know that $\widehat{\varphi}:E/\operatorname{Ker}(\varphi)\to F$ is $c^{-1}$-topologically injective isomorphism of normed spaces. Then it have right inverse bounded operator $\psi:F\to E/\operatorname{Ker}(\varphi)$ with $\Vert\psi\Vert\leq c$. By proposition \ref{PrCharMaxSQ} we have sequentially bounded operrator $\psi':\max(F)\to\max(E)/\operatorname{Ker}(\varphi):x\mapsto \psi(x)$ with $\Vert\psi'\Vert_{sb}=\Vert\psi\Vert\leq c$. From proposition \ref{PrFactorSQOp} we have factorization $\max(\varphi)=\widehat{\max(\varphi)}\pi_{\operatorname{Ker}(\varphi),E}$, where $\widehat{\max(\varphi)}:E/\operatorname{Ker}(\varphi)\to F$ is sequentially bounded operator. Clearly $\widehat{\max(\varphi)}=\widehat{\varphi}$ and $\psi=\psi'$ as linear maps, hence $\widehat{\max(\varphi)}$ and $\psi'$ are sequentially bounded linear operators which are inverse to each other. Now, for any $n\in\mathbb{N}$  and $y\in\max(F)^{\wideparen{n}}$ consider $x=(\psi')^{\wideparen{n}}(y)$, then $(\widehat{\max(\varphi)})^{\wideparen{n}}(x)=y$ and $\Vert x\Vert_{\wideparen{n}}=\Vert(\psi')^{\wideparen{n}}(y)\Vert_{\wideparen{n}}\leq\Vert(\psi')^{\wideparen{n}}\Vert\Vert y\Vert_{\wideparen{n}}\leq\Vert \psi'\Vert_{sb}\Vert y\Vert_{\wideparen{n}}\leq c\Vert y\Vert_{\wideparen{n}}$. Since $n\in\mathbb{N}$ and $y\in \max(F)^{\wideparen{n}}$ are arbitrary, then $\widehat{\max(\varphi)}$ is sequentially $c$-topologically surjective. Since $\pi_{\operatorname{Ker}(\varphi),E}$ is sequentially $1$-topologically surjective, then by proposition \ref{PrComposeSQTopInjSur} $\max(\varphi)=\widehat{\max(\varphi)}\pi_{\operatorname{Ker}(\varphi),E}$ is $c$-topologically surjective.

2) By previous paragraph $\max(\varphi)$ is $1$-topologically surjective. On the other hand, by proposition \ref{PrCharMaxSQ} we have $\Vert\max(\varphi)\Vert_{sb}=\Vert\varphi\Vert=1$. Therefore $\max(\varphi)$ is sequentially coisometric.
\end{proof}

\begin{proposition}\label{PrMaxCommuteWithCoprod} Let $\{X_\lambda:\lambda\in\Lambda\}$ be a family of maximal operator sequence spaces, then $\bigoplus{}_1\{X_\lambda:\lambda\in\Lambda\}$ is also maximal.
\end{proposition} 
\begin{proof}
Let $Y$ be arbitrary operator sequence space, then from propositions \ref{PrSQCoProdUnivProp}, \ref{PrCharMinSQ} and \ref{PrCharMaxSQ} we have isometric identifications
$$
\mathcal{SB}\left(\bigoplus{}_1^0\{X_\lambda:\lambda\in\Lambda\},Y\right)^{\wideparen{1}}
=\bigoplus{}_\infty\{\mathcal{SB}(X_\lambda,Y)^{\wideparen{1}}:\lambda\in\Lambda\}
=\bigoplus{}_\infty\{\mathcal{B}(X_\lambda^{\wideparen{1}},Y^{\wideparen{1}}):\lambda\in\Lambda\}
$$
$$
=\bigoplus{}_\infty\{\mathcal{SB}(X_\lambda,\min(Y^{\wideparen{1}}))^{\wideparen{1}}:\lambda\in\Lambda\}
=\mathcal{SB}\left(\bigoplus{}_1^0\{X_\lambda:\lambda\in\Lambda\},\min(Y^{\wideparen{1}})\right)^{\wideparen{1}}
$$
$$
=\mathcal{B}\left(\left(\bigoplus{}_1^0\{X_\lambda:\lambda\in\Lambda\}\right)^{\wideparen{1}},Y^{\wideparen{1}}\right)
$$
Since $Y$ is arbitrary, from proposition \ref{PrCharMaxSQ} we conclude that $\bigoplus{}_1^0\{X_\lambda:\lambda\in\Lambda\}$ have maximal operator sequence space structure.
\end{proof}

\begin{proposition}\label{Prl1IsMax} Let $\Lambda$ be an arbitrary, set, then $l_1^0(\Lambda):=\bigoplus_1\{\mathbb{C}:\lambda\in\Lambda\}$ have maximal operator sequence structure.
\end{proposition}
\begin{proof} By proposition \ref{PrCHaveUniqueOSS} operator sequence space structure of $\mathbb{C}$ is unique and in particular maximal. Now result follows from proposition \ref{PrMaxCommuteWithCoprod}.
\end{proof}

\begin{proposition}\label{PrMaxIsQuotientOfl1} Let $X$ be an operator sequence space, then $X$ is maximal if and only if there exist sequential coisometry from $l_1(\Lambda)$ onto $X$ for some set $\Lambda$.
\end{proposition}
\begin{proof} 
Assume $X$ have maximal structure. Consider natural coisometry $\pi:l_1(B_X)\to X$ (see A1 \cite{DefFloTensNorOpId}). By proposition \ref{PrMaxPreserveQuotients} we know that $\max(\pi):\max(l_1(B_X)^{\wideparen{1}})\to\max(X^{\wideparen{1}})$ is sequentially coisometric. By proposition \ref{Prl1IsMax} we have $\max(l_1(B_X)^{\wideparen{1}})=l_1(B_X)$ and by assumption $\max(X^{\wideparen{1}})=X$, so we get the desired sequential coisometry $\max(\pi):l_1(B_X)\to X$.

Conversely, assume we are given sequential coisometry $\pi:l_1(\Lambda)\to X$, then by proposition \ref{PrFactorSQOp} we have that $X$ and $l_1(\Lambda)/\operatorname{Ker}(\pi)$ are sequentially isometrically isomorphic via $\widehat{\pi}$. Since $\pi^{\wideparen{1}}:l_1(\Lambda)^{\wideparen{1}}\to X^{\wideparen{1}}$ is coisometric too, by proposition \ref{PrMaxPreserveQuotients} we have sequential coisometry $\max(\pi):\max(l_1(\Lambda)^{\wideparen{1}})\to \max(X^{\wideparen{1}})$.From proposition \ref{Prl1IsMax} it is known that $\max(l_1(\Lambda)^{\wideparen{1}})=l_1(\Lambda)$, so we have one more sequential coisometry $\max(\pi):l_1(\Lambda)\to\max(X^{\wideparen{1}})$. Again by proposition \ref{PrFactorSQOp} we see that $\max(X^{\wideparen{1}})$ and $l_1(\Lambda)/\operatorname{Ker}(\pi)$ are sequentially isometrically isomorphic via $\widehat{\max(\pi)}$. Therefore $X=l_1(\Lambda)/\operatorname{Ker}(\pi)=\max(X^{\wideparen{1}})$.
\end{proof}

\begin{proposition}[\cite{LamOpFolgen}, 2.1.11]\label{PrDualityAndMinMax}
Let $E$ be a normed space, then identity operator gives sequential isometric isomorphisms
$$
\max(E^*)=\min(E)^\triangle,
\qquad
\min(E^*)=\max(E)^\triangle
$$
\end{proposition}

Similar results holds in categories $SQNor$, $SQBan$ and $SQBan_1$.

\subsection{Tensor products of operator sequence spaces}

It is natural to expect some kind of tensor product linearizing sequentially bounded bilinear operators.
\begin{definition}[\cite{LamOpFolgen}, 3.1.1]\label{DefSQMaxTenProd}
Let $X$ and $Y$ be operator sequence spaces, then their maximal tensor product is a operator sequence space $X\otimes_{\mathrm{Max}}Y$ with the family of norms 
$(\Vert\cdot\Vert_{(X\otimes_{\mathrm{Max}}Y)^{\wideparen{n}}})_{n\in\mathbb{N}}$ given by equalities
$$
\Vert u\Vert_{(X\otimes_{\mathrm{Max}}Y)^{\wideparen{n}}}
=\inf\left\{\Vert[\alpha_1,\ldots,\alpha_k]\Vert_{M_{n,kl^2}}\left(\sum\limits_{i=1}^k\Vert x_i\Vert_{X^{\wideparen{l}}}^2\Vert y_i\Vert_{Y^{\wideparen{l}}}^2 \right)^{1/2}:u=\sum\limits_{i=1}^k\alpha_i(x_i\otimes y_i)\right\}
$$
where $u\in (X\otimes_{\mathrm{Max}}Y)^{\wideparen{n}}$, $\alpha_1,\ldots,\alpha_k\in M_{n,l^2}$ and $x\in X^{\wideparen{l}}$, $y\in Y^{\wideparen{l}}$. Using standard completion procedure for operator sequence spaces, we define completed version of this tensor product, which we will denote $X\otimes^{\mathrm{Max}} Y$.
\end{definition}

In [\cite{LamOpFolgen} 3.1.2] it is proved that, the norm defined above is the maximal cross norm making $X\otimes Y$ a operator sequence space. This tensor product is called \textit{maximal} and denoted by $X \otimes_{\mathrm{Max}} Y$. Maximal tensor product have universal property with respect to the class of  sequentially bounded bilinear operators.

\begin{proposition}[\cite{LamOpFolgen}, 3.1.3, 3.1.4]\label{PrSQUnivPropMaxTenProd}
Let $X$, $Y$ and $Z$ be operator sequence spaces, then there exist sequential isometric isomorphisms 
$$
\mathcal{SB}(X\otimes_{\mathrm{Max}}Y, Z)
=\mathcal{SB}(X\otimes^{\mathrm{Max}}Y, Z)
=\mathcal{SB}(X\times Y, Z)
=\mathcal{SB}(X,\mathcal{SB}(Y,Z))
=\mathcal{SB}(Y,\mathcal{SB}(X,Z))
$$
natural in $X$, $Y$ and $Z$.
\end{proposition}

\begin{corollary}\label{CorSQUnivPropMaxTenProd}
Let $X$, $Y$ be operator sequence spaces, then there exist sequential isometric isomorphisms
$$
\mathcal{SB}(X^\triangle, Y)=\mathcal{SB}(X,Y^\triangle)=(X\otimes_{\operatorname{Max}} Y)^\triangle
$$
natural in $X$ and $Y$. 
\end{corollary}

\section{Rigged categories}

\subsection{Projectivity and injectivity. Freedom and cofreedom}

Now we will quote some definitions and results from \cite{HelMetrFrQmod}. Let $\mathcal{K}$ be an arbitrary category.
\begin{definition}[\cite{HelMetrFrQmod}, 2.1]\label{DefRigCat}
A pair ($\mathcal{K}, \square:\mathcal{K}\to\mathcal{L}$), where $\square$ is a faithful covariant functor, is called a rigged category. A dual rigged category of $(\mathcal{K}, \square)$ 
is a rigged category $(\mathcal{K}^{o},\square^{o}:\mathcal{K}^{o}\to\mathcal{L}^{o})$. 
\end{definition}
\begin{definition}[\cite{HelMetrFrQmod}, 2.1]\label{DefAdmMorph}
A morphism $\tau$ in $\mathcal{K}$ is called $\square$-admissible epimorphism (monomorphism) if $\square (\tau)$ is a retraction (coretraction) in $\mathcal{L}$.
\end{definition}
\begin{definition}[\cite{HelMetrFrQmod}, 2.2]\label{DefProjInj}
An object $P\in \mathcal{K}$ ($I \in \mathcal{K}$) is called $\square$-projective ($\square$-injective) with respect to a rigged category $(\mathcal{K}, \square)$, if for every $X,Y\in\mathcal{K}$ and every $\square$-admissible epimorphism $\tau : Y \to X$ (monomorphism $\tau : X \to Y$)  the map $\operatorname{Hom}_{\mathcal{K}}(P,\tau)$ ($\operatorname{Hom}_{\mathcal{K}}(\tau, I)$) is surjective.
\end{definition}
\begin{definition}[\cite{HelMetrFrQmod}, 2.10]\label{DefFrAndCoFr}
An object $F \in \mathcal{K}$ is called $\square$-free ($\square$-cofree) with base $M \in \mathcal{L}$, if there exist a morphism $j : M \to \square(F)$ ($j : \square(F) \to M$), 
such that for each $X \in \mathcal{K}$ and each morphism $\varphi : M \to \square(X)$ ($\varphi : \square(X) \to M$) there exist the unique $\psi : F \to X$ ($\psi : X \to F$), 
making the following diagram
$$
\xymatrix{
{\square (F)} \ar@{-->}[dr]^{\square (\psi)} & \\
{M} \ar[u]^{j} \ar[r]^{\varphi} &{\square (X)}}
\qquad\qquad\quad
\xymatrix{
{\square (F)} \ar[d]_{j} & \\
{M} &{\square (X)} \ar[l]_\varphi \ar@{-->}[ul]_{\square(\psi)}
}
$$
commutative. The morphism $j$ is called universal arrow.
\end{definition}
\begin{definition}\label{DefFrAndCoFrLove}
A rigged category $(\mathcal{K},\square)$ is called freedom-loving (cofreedom-loving), if every object in $\mathcal{L}$ is a base of some $\square$-free ($\square$-cofree) object in $\mathcal{K}$.
\end{definition}

The following results will be extremely useful in near future.

\begin{proposition}\label{PrUniqFr}
Let $M\in\mathcal{L}$ be a base of $\square$-free ($\square$-cofree) objects $F_1$, $F_2$ in the rigged category $(\mathcal{K},\square)$, then $F_1$ and $F_2$ are isomorphic.  
\end{proposition}
\begin{proof}
Consider category $\mathcal{K}_M$, whose objects are pairs of the form $(X,\varphi:M\to\square X)$, and morphisms from $(X_1,\varphi_1)$ to $(X_2,\varphi_2)$ are morphisms $\psi$ in 
$\mathcal{K}$, such that $\varphi_2=\square(\psi)\varphi_1$. Composition of morphisms in $\mathcal{K}_M$ is the same as in $\mathcal{K}$. Clearly, every object $F$ is $\square$-free in $\mathcal{K}$ with base $M$ and universal arrow $j$ if and only if $(F,j)$ is the terminal object in $\mathcal{K}_M$. Recall that every terminal object in any category is unique up to isomorphism. It is remains to note that every isomorphism in $\mathcal{K}_M$ is an isomorphism in $\mathcal{K}$.
\end{proof}

\begin{proposition}\label{PrCompOfFrIsFr} 
Let $\square_{12}:\mathcal{K}_1\to\mathcal{K}_2$, $\square_{23}:\mathcal{K}_2\to\mathcal{K}_3$ be faithful functors. Denote $\square_{13}=\square_{23}\square_{12}$. Let 
$F_1$ be $\square_{12}$-free ($\square_{12}$-cofree) object with base $F_2$ and universal arrow $j_{12}$ in the rigged category $(\mathcal{K}_1,\square_{12})$. Let $F_2$ be $\square_{23}$-free 
($\square_{23}$-cofree) object with base $F_3$ and universal arrow $j_{23}$ in the rigged category $(\mathcal{K}_2,\square_{23})$. Then $F_1$ is a $\square_{13}$-free ($\square_{13}$-cofree) 
object with base $F_3$ and universal arrow $\square_{23}(j_{23})j_{12}$ in the rigged category $(\mathcal{K}_1,\square_{13})$. 
$$
\xymatrix{
& { \mathcal{K}_2} \ar[dr]^{\square_{23}} & \\
{\mathcal{K}_1} \ar[ur]^{\square_{12}} \ar[rr]_{\square_{13}} & & {\mathcal{K}_3}}
$$
As the consequence, if rigged categories $(\mathcal{K}_1,\square_{12})$, $(\mathcal{K}_1,\square_{12})$ are freedom-loving (cofreedom-loving), then so does $(\mathcal{K}_1,\square_{13})$. For cofree objects the proof is the same.
\end{proposition}
\begin{proof}
Consider arbitrary object $X\in\mathcal{K}_1$ and a morphism $\varphi:F_3\to \square_{13}(X)$. Since $F_2$ is a $\square_{23}$-free object, then there exist unique $\psi:F_2\to \square_{12}(X)$, 
such that $\varphi=\square_{23}(\psi)j_{23}$. Since $F_1$ is $\square_{12}$-free, then there exist the unique $\chi:F_1\to X$, such that $\psi=\square_{12}(\chi)j_{12}$.
$$
\xymatrix
{
{\square_{23}(\square_{12}(F_1)}) \ar[rr]^{\square_{23}(\square_{12}(\chi))}& & {\square_{23}(\square_{12}(X)}\\
{\square_{23}(F_2)} \ar[u]^{\square_{23}(j_{12})} \ar[urr]^{\square_{23}(\psi)} & &\\
{F_3} \ar[u]^{j_{23}} \ar[uurr]^{\varphi} & \\
} 
$$
Therefore $\varphi=\square_{23}(\psi)j_{23}=\square_{23}(\square_{12}(\chi))\square_{23}(j_{23})j_{12}=\square_{13}(\chi) j_{13}$, where $j_{13}=\square_{23}(j_{23})j_{12}$ is a universal arrow. 
Since $X$ and $\varphi$ are arbitrary, then $F_1$ is a $\square_{13}$-free object with base $F_3$. For cofree objects the proof is the same.
\end{proof}

\begin{proposition}[\cite{HelMetrFrQmod}, 2.3]\label{PrRetractsProjInj} Let $(\mathcal{K},\square)$ be a rigged category, and $P\in\mathcal{K}$ ($I\in\mathcal{K}$) be 
$\square$-projective ($\square$-injective) object, then
\newline
1) if $\sigma:P\to Q$ ($\sigma:I\to J$) is a retraction, then $Q$ is $\square$-projective ($J$ $\square$-injective)
\newline
2) if $\sigma:X\to P$ ($\sigma:X\to I$) is $\square$-admissible epimorphism (monomorphism), then $\sigma$ is a retraction (coretraction)
\end{proposition}

\begin{proposition}[\cite{HelMetrFrQmod}, 2.11]\label{PrFrCoFrProjInjObjProp} Let $(\mathcal{K},\square)$ be a rigged category, then
\newline
1) if $F\in\mathcal{K}$ is a $\square$-free ($\square$-cofree), then it is $\square$-projective ($\square$-injective)
\newline
2) if $X$ such object in $\mathcal{K}$, that $\square(X)$ is a base of $\square$-free ($\square$-cofree) object $F$, then there exist $\square$-admissible epimorphism (monomorphism) from $F$ to $X$ (from $X$ to $F$).
\newline
3) if $\mathcal{K}$ is freedom-loving (cofreedom-loving), then $P\in\mathcal{K}$ ($I\in\mathcal{K}$) is $\square$-projective ($\square$-injective) if and only if it is a retract of $\square$-free ($\square$-cofree) object.
\end{proposition}

\begin{proposition}[\cite{HelMetrFrQmod}, 2.13]\label{PrCoprodFrIsFr} Let $(\mathcal{K},\square)$ be a rigged category, and $\Lambda$ be any set. Assume for each 
 $\lambda \in \Lambda$ an object $F_{\lambda} \in \mathcal{K}$ is $\square$-free ($\square$-cofree) with base $M_{\lambda} \in \mathcal{L}$. Assume that $\{ F_\lambda:\lambda \in \Lambda\}$ admits coproduct (product) $F$, and the family $\{ M_\lambda:\lambda \in \Lambda\}$ admits coproduct (product) $M$. 
Then the object $F$ is $\square$-free ($\square$-cofree) with base $M$.
\end{proposition}

\begin{proposition}[\cite{HelMetrFrQmod}, 4.5]\label{PrFunctorMapFrToFr} Let $(\mathcal{K}_1, \square_1: \mathcal{K}_1 \to \mathcal{L}_1)$ and $(\mathcal{K}_2, \square_2 : \mathcal{K}_2 \to \mathcal{L}_2)$ are rigged categories,
and we are given covariant functors $\Phi : \mathcal{K}_1 \to \mathcal{K}_2$ and $\Psi : \mathcal{L}_1 \to \mathcal{L}_2$ such that the diagram
$$
\xymatrix{
{\mathcal{K}_1}\ar[d]_{\Phi}
\ar[rr]^{\square_1} & & {\mathcal{L}_1}\ar[d]^{\Psi}\\
{\mathcal{K}_2}\ar[rr]^{\square_2} & & {\mathcal{L}_2}}
$$
is commutative. Assume that $\Phi$ and $\Psi$ has left (right) adjoint functors $\Phi^*$ and $\Psi^*$ (${}^*\Phi$ and ${}^*\Psi$) respectively, and $F\in\mathcal{K}_2$ is a 
$\square_2$-free ($\square_2$-cofree) object with base $M\in\mathcal{L}_2$.
Then $\Phi^*(F)\in\mathcal{K}_1$ (${}^*\Phi(F)\in\mathcal{K}_1$) is a $square_1$-free ($\square_1$-cofree) object with base $\Psi^*(M)\in\mathcal{L}_1$ (${}^*\Psi(M)\in\mathcal{L}_1$). 
\end{proposition}

\subsection{Normed semilinear spaces}

In what follows we will need the following construction.

\begin{definition}\label{DefSemiLinSp} A semilinear space $V$ over field $K$ is an ordered triple $(V, K, \cdot)$, where $V$ is a nonempty set, whose elements are called vectors, $K$ is a field, whose elements are called scalars, $\cdot : K \times V \to V$ is a map satisfying the following axioms

1) for all $x\in V$, $\alpha,\beta\in K$ holds $\alpha \cdot (\beta \cdot x) = (\alpha \beta) \cdot x $

2) for all $x\in V$ holds $1_K \cdot x = x$

3) there exist a vector $0 \in V$, such that $0_K \cdot x = 0$ for all $x\in V$.

The vector $0\in V$ is called a zero vector.
\end{definition}

Clearly, zero vector is unique and $\alpha \cdot 0 = 0$ for all $\alpha\in K$.

\begin{example}\label{ExSemiLinModelSp}
Consider wedge sum $\bigvee\{K: \lambda \in\Lambda\}$ of copies of the field $K$, which intersects by zero vector, for some set $\Lambda$. Multiplication in wedge sum is inherited from the filed. Obviously, this is a semilinear space over field $K$, which we will denote $K^{\Lambda}$. By $K^{\varnothing}$ we will understand semilinear space, consisting of single zero vector. 
\end{example}

\begin{definition}\label{DefSemiLinOp} A map $\varphi : V \to W$ between semilinear spaces $V$ and $W$ is called semilinear operator, if $\varphi(\alpha \cdot x) = \alpha \cdot \varphi(x)$ for all $\alpha \in K$ and $x \in V$.
\end{definition}

Consider category $Lin_{0}^{K}$, whose objects are semilinear spaces over field $K$, and morphisms --- semilinear operator. We can easily get complete characterization of objects of this category.

\begin{proposition}\label{PrSemiLinSpDesc}
Every semilinear space is isomorphic in $Lin_{0}^{K}$ to $K^{\Lambda}$ for some set $\Lambda$.
\end{proposition}
\begin{proof} We say that two vectors $x,y\in V$ are equivalent if $x=\alpha y$ for some $\alpha\in K\setminus\{0\}$. This relation $\sim$ is an equivalence relation. Let $\{x_\lambda:\lambda\in \Lambda\}$  be a set of representatives of each equivalence class except equivalence class of zero vector. Then the semilinear operator $\varphi: K^\Lambda\to V: z_\lambda\mapsto z_\lambda x_\lambda$ is an isomorphism in  $Lin_0^K$
\end{proof}

\begin{definition}\label{DefSemiLinNorSp} A semilinear normed space over normed field $K$ is a pair $(E, \Vert \cdot \Vert)$, where $E$ is a semilinear space over field $K$ and $ \Vert \cdot \Vert : E \to \mathbb{R}_+$ is a map, which we will call a norm, satisfying the following relations:

1) if $x\in E$ and $\Vert x \Vert = 0$, then $x = 0$;

2) for all $x\in E$ and $\alpha\in K$ holds $\Vert \alpha \cdot x \Vert = | \alpha| \Vert x \Vert$.
\end{definition}

\begin{example}\label{ExSemiLinNorModelSp}
For a given normed field $K$ we define a norm on $K^{\Lambda}$, by equality $\Vert z_\lambda\Vert:=|z_\lambda|_K$ for each $z_\lambda\in K^\Lambda$. 
\end{example}

\begin{definition}\label{DefSemiLinBndOp} A semilinear operator $\varphi : E \to F$ between semilinear normed spaces $E$ and $F$ is called bounded, if $\Vert \varphi (x)\Vert \leq C \Vert x \Vert$ for  some constant $C\in\mathbb{R}_+$. Infima of all such constants we will call a norm of $\varphi$ and will denote it by $\Vert\varphi\Vert$.
\end{definition}

Now consider category $Nor_0^K$, whose objects are semilinear normed spaces, and morphisms --- bounded semilinear operators. It is not hard to classify objects of this category.

\begin{proposition}\label{PrSemiLinNorSpDesc} Every semilinear normed space in $Nor_0^K$ is isomorphic to $K^{\Lambda}$ for some set $\Lambda$.
\end{proposition}

\begin{proof} Using proposition \ref{PrSemiLinSpDesc} consider equivalence relation $\sim$ and a set $\{x_\lambda:\lambda\in\Lambda\}$ of representatives of equivalence classes, except equivalence class of zero vector. Fix some $\alpha\in K$ such that $0<|\alpha|<1$. For each $\lambda\in\Lambda$ there exist $m_\lambda\in\mathbb{Z}$ such that $|\alpha|^{-m_\lambda}\leq\Vert x_\lambda\Vert< |\alpha|^{-m_\lambda+1}$. Define $y_\lambda=\alpha^{m_\lambda} x_\lambda$, then $1\leq \Vert y_\lambda\Vert<|\alpha|$. Now it is easy to see that the semilinear operator $\varphi: K^\Lambda\to E: z_\lambda\mapsto z_\lambda y_\lambda$ is an isomorphism in $Nor_0^K$
\end{proof}

In what follows by $Nor_0$ we will denote the category $Nor_0^\mathbb{C}$.

\subsection{Examples of rigged categories}

Consider several examples. For simplicity we will deal only with normed spaces. One can easily extend these constructions to the case of normed modules.

\begin{example}[Metric freedom, \cite{HelMetrFrQmod}]\label{ExMetrFr}
Let $\mathcal{K} = Nor_1$, $\mathcal{L} = Set$. The functor $\square$ sends a normed space to its closed unit ball, and morphism is mapped to its birestriction to unit balls in domain and range space. In this case $\square$-admissible epimorphisms are strict coisometries., $\square$-free object with one point base is $\mathbb{C}$. Hence from proposition \ref{PrCoprodFrIsFr} 
immediately follows, that $\square$-free object with base $\Lambda$ is $l_1^0(\Lambda)$.
\end{example}

\begin{example}[Topological freedom]\label{ExTopFr}
Let $\mathcal{K} = Nor$, $\mathcal{L} = Nor_0$. The functor $\square$ sends a normed space to its underlying semilinear normed space with the same norm, and a morphism remains the same. In this case $\square$-admissible epimorphisms are are topologically surjective operators, $\square$-free object with base $\mathbb{C}^\Lambda$ is $l_1^0(\Lambda)$.
\end{example}

\begin{example}[Metric cofreedom, \cite{HelMetrFrQmod}]\label{ExMetrCoFr}
Let $\mathcal{K} = Nor_1$, $\mathcal{L} = Set^0$. The functor $\square$ send a normed space $X$ into the unit ball of $X^*$, a morphism is mapped to birestriction of dual morphism to unit balls of domain and range spaces. In this case $\square$-admissible epimorphisms are isometries, $\square$-cofree object with base $\Lambda$ easily constructed from example \ref{ExMetrFr} 
and proposition \ref{PrCoprodFrIsFr} --- this is the space $l_\infty(\Lambda)$.
\end{example}

\begin{example}[Topological cofreedom, \cite{ShtTopFr}]\label{ExTopCoFr}
Let $\mathcal{K} = Nor$, $\mathcal{L} = Nor_0^o$. The functor $\square$ send a normed space to underlying semilinear normed space of its dual, a morphism is mapped to the its adjoint. In this case $\square$-admissible epimorphisms are topologically injective operators, $\square$-cofree objects with base $\mathbb{C}^\Lambda$ easily constructed from example \ref{ExTopFr} and proposition \ref{PrCoprodFrIsFr} --- this is the space $l_\infty(\Lambda)$. 
\end{example}

All these examples have their obvious Banach analogues, given by completion of free and cofree objects mentioned above. Moreover these examples have their quantum versions: the role of free object with one point base instead of $\mathbb{C}$ plays the operator space $\mathcal{N}_{\infty} :=  \bigoplus{}_1^0\{\mathcal{N}(\mathbb{C}^n):n\in\mathbb{N}\}$([\cite{HelMetrFrQmod}, 5.9], 
see also \cite{ShtTopFr}). Our immediate goal is to show the same role for operator sequence spaces is played by $t_2^{\infty} :=  \bigoplus_1^0\{t_2^n:n\in\mathbb{N}\}$.

\section{Free operator sequence spaces}

\subsection{Metric freedom}

We begin with metric version of freedom for operator sequence spaces. Consider functor
$$
\begin{aligned}
\square_{sqMet} : SQNor_1 \to Set : X&\mapsto\prod\left\{ B_{X^{\wideparen{n}}}:n \in \mathbb{N}\right\}\\
\varphi&\mapsto \prod\left\{\varphi^{\wideparen{n}}|_{B_{X^{\wideparen{n}}}}^{B_{Y^{\wideparen{n}}}}:n\in\mathbb{N}\right\}
\end{aligned}
$$
sending a operator sequence space to the cartesian product of unit balls of its amplifications. 

\begin{proposition}\label{PrDecsMetrAdmEpiMorph}
$\square_{sqMet}$-admissible epimorphisms are exactly sequentially strictly coisometric operators.
\end{proposition}
\begin{proof}
A morphism $\varphi$ is $\square_{sqMet}$-admissible epimorphism if $\square_{sqMet}(\varphi)$ is invertible from the right as morphism in $Set$. This is equivalent to surjectivity of $\square_{sqMet}(\varphi)$, which is equivalent to surjectivity of $\varphi^{\wideparen{n}}|_{B_{X^{\wideparen{n}}}}^{B_{Y^{\wideparen{n}}}}$ for all $n\in\mathbb{N}$. The latter means that $\varphi^{\wideparen{n}}$ strictly coisometric for each $n\in\mathbb{N}$. So $\varphi^{\wideparen{n}}$ sequentially strictly coisometric.
\end{proof}

By $I_n$ we denote the element of $(t_2^n)^{\wideparen{n}} = \mathcal{B}(l_2^n, l_2^n)$, corresponding to the identity operator.

\begin{proposition}\label{PrMetrFrLem} Let $X$ be a operator sequence space and $x \in B_{X^{\wideparen{n}}}$. Then there exist unique sequentially contractive operator 
$\psi_n \in \mathcal{SB}(t_2^n, X)$, such that $\psi_n^{\wideparen{n}}(I_n) = x$.
\end{proposition}
\begin{proof}
Since, $I_n = (e_i)_{i\in\mathbb{N}_n}$, where $e_i$ is the $i$-th orth of underlying space $t_2^n$. Obviously, there exist unique linear operator $\psi_n$, satisfying $\psi_n(e_i) = x_i$, $i\in\mathbb{N}_n$. 
It is remains to check that $\psi_n$ is sequentially contractive. Let $k \in \mathbb{N}$ and $y \in B_{(t_2^n)^{\wideparen{k}}} $, then $y_i = \sum_{j = 1}^n \alpha_{ij}e_j$, $i\in\mathbb{N}_k$ 
for some matrix $\alpha\in M_{k,n}$. Then
$$
\Vert\psi_n^{\wideparen{k}}(y)\Vert_{\wideparen{k}}
=\left\Vert\left(\psi_n(y_i)\right)_{i\in\mathbb{N}_k}\right\Vert_{\wideparen{k}}
=\left\Vert\left(\sum\limits_{j=1}^n\alpha_{ij}\psi_n(e_j)\right)_{i\in\mathbb{N}_k}\right\Vert_{\wideparen{k}}
=\left\Vert\left(\sum\limits_{j=1}^n\alpha_{ij}x_j\right)_{i\in\mathbb{N}_k}\right\Vert_{\wideparen{k}}
$$
$$
=\Vert\alpha x\Vert_{\wideparen{k}}
\leq\Vert\alpha\Vert\Vert x\Vert_{\wideparen{n}}
=\Vert y\Vert_{(t_2^n)^{\wideparen{k}}}\Vert x\Vert_{\wideparen{n}}\leq 1
$$
Therefore $\psi_n$ is sequentially contractive.
\end{proof}

\begin{proposition}\label{PrOnePtMetrFr} Metrically free operator sequence space with one point base is a space $t_2^{\infty} := \bigoplus_1^0 \{t_2^n: n \in \mathbb{N}\}$.
\end{proposition}
\begin{proof}
Define universal arrow as such $j:\{\lambda\}\to t_2^\infty:\lambda\mapsto(I_1,I_2,\ldots,I_n,\ldots)$. Let $X$ be arbitrary operator sequence space and 
$\varphi:\{\lambda\}\to \prod_{n \in \mathbb{N}} B_{X^{\wideparen{n}}}$ be some map. Denote $x=\varphi(\lambda)$. From proposition \ref{PrMetrFrLem} and properties of corpducts it  follows that, there exist unique sequentially contractive operator $\psi=\bigoplus_1^0\{\psi_n:n\in\mathbb{N}\}\in  \mathcal{SB}\left(\bigoplus_1^0\{ t_2^n:n\in\mathbb{N}\}, X\right)$, such that $\psi^{\wideparen{n}}(i_n(I_n)) = x$, for all $n \in \mathbb{N}$. Here $i_n:t_2^n\to t_2^\infty$ stands for standard embedding.
$$
\xymatrix{
{\square_{sqMet} (t_2^\infty)} \ar@{-->}[dr]^{\square_{sqMet} (\psi)} & \\
{\{\lambda\}} \ar[u]^{j} \ar[r]^{\varphi} &{\square_{sqMet} (X)}}
$$
In this case $\varphi=\square_{sqMet}(\psi) j$. Since $X$ and $\varphi$ are arbitrary, then $t_2^\infty$ is metrically free with one point base. 
\end{proof}

Thus we are ready to state the final result.

\begin{theorem}\label{ThMetrFrDesc} Metrically free operator sequence space with base $\Lambda$ is up to sequential isometric isomorphism a $\bigoplus{}_1^0$-sum of copies of the space $t_2^{\infty}$, indexed by elements of the set $\Lambda$. 
\end{theorem}
\begin{proof}
Result follows from propositions \ref{PrCoprodFrIsFr} and \ref{PrOnePtMetrFr}
\end{proof}

\begin{corollary}\label{CorSQSpaceIsImgMetrAdmEpiMorph}
Every operator sequence space is an image of sequentially strictly coisometric operator from $\bigoplus_1^0\{t_2^\infty:\lambda\in\Lambda\}$ for some set $\Lambda$.
\end{corollary}
\begin{proof}
From theorem \ref{ThMetrFrDesc} we see that $(SQNor_1,\square_{sqMet})$ is freedom-loving. Now the desired result follows from propositions \ref{PrFrCoFrProjInjObjProp} and \ref{PrDecsMetrAdmEpiMorph}.
\end{proof}

Similar propositions are valid in Banach case (just replace $\bigoplus{}_1^0$-sums with $\bigoplus{}_1$-sums).

\subsection{Topological freedom}

Let's proceed to consideration of sequential version of topological  freedom. Consider functor
$$
\begin{aligned}
\square_{sqTop} : SQNor \to Nor_0: X &\mapsto \bigoplus{}_\infty \{X^{\wideparen{n}} : n \in \mathbb{N}\}\\
\varphi&\mapsto\bigoplus{}_\infty\{\varphi^{\wideparen{n}}:n\in\mathbb{N}\},\\
\end{aligned}
$$
sending operator sequence space to underlying semilinear normed space of $\bigoplus_\infty$-sum of its amplifications.

\begin{proposition}\label{PrCTopSurIsRetrInNor0} Let $\varphi:X\to Y$ be bounded linear operator between normed spaces $X$ and $Y$, then it is $c$-topologically surjective if and only if there exist bounded semilinear operator $\rho:Y\to X$ such that $\Vert\rho\Vert\leq c$ and $\varphi\rho=1_Y$.
\end{proposition}
\begin{proof} Assume $\varphi$ is $c$-topologically surjective. Consider relation $\sim$ on $S_Y$ defined as follows: $e_1\sim e_2$  if and only if  there exist $\alpha\in\mathbb{T}$ such that $e_1=\alpha e_2$. Clearly, $\sim$ is an eqivalence relation, so we can consider a set of non zero representatives of equivalence classes, say $\{r_\lambda:\lambda\in\Lambda\}$. By construction, for each $e\in S_Y$ we have unique $\alpha(e)\in\mathbb{T}$ and $\lambda(e)\in\Lambda$ such that $e=\alpha(e)r_{\lambda(e)}$. Clearly, for any $z\in\mathbb{T}$ and $e\in S_Y$ we have $\alpha(ze)=z\alpha(e)$ and $\lambda(ze)=\lambda(e)$. Since $\varphi$ is $c$-topologically surjective, then, in particular, for each $\lambda\in\Lambda$ we have $x(\lambda)\in X$ such that $\Vert x(\lambda)\Vert\leq c\Vert r_\lambda\Vert$ and $\varphi(x(\lambda))=r_\lambda$. Consider, map $\tilde{\rho}:S_Y\to X:e\mapsto \alpha(e)x(\lambda(e))$. It is easy to see that for all $z\in\mathbb{T}$ and $e\in S_Y$ holds $\tilde{\rho}(z e)=z\tilde{\rho}(e)$, $\Vert\tilde{\rho}(e)\Vert\leq c$ and $\varphi(\tilde{\rho}(e))=e$. Now consider map $\rho:Y\to X: y\mapsto \Vert y\Vert\tilde{\rho}(\Vert y\Vert^{-1} y)$ and $\rho(0)=0$. Using properties of $\tilde{\rho}$ it is trivial to check that $\rho$ is semilinear operator such that $\Vert\rho\Vert\leq c$ and $\varphi\rho=1_Y$.

Conversely, assume there exist bounded semilinear operator $\rho:Y\to X$ such that $\Vert\rho\Vert\leq c$ and $\varphi\rho=1_Y$. Take any $y\in Y$ and consider $x=\rho(y)$, then $\Vert x\Vert\leq C\Vert y\Vert$ and $\varphi(x)=y$. Hence $\varphi$ is $c$-topologically surjective.
\end{proof}

\begin{proposition}\label{PrDecsTopAdmEpiMorph} $\square_{sqTop}$-admissible epimorphisms are exactly sequentially topologically surjective operators.
\end{proposition}
\begin{proof}
For a given operator sequence space $Z$ by $i_n^Z:Z^{\wideparen{n}}\to\square_{sqTop}(Z)$ we denote natural embedding, and by $p_n^Z:\square_{sqTop}(Z)\to Z^{\wideparen{n}}$ we denote natural projection. Assume that $\varphi:X\to Y$ is $c$-sequentially topologically surjective. Fix $n\in\mathbb{N}$, then by proposition \ref{PrCTopSurIsRetrInNor0} there exist bounded semilinear operator $\rho^n$ such that $\varphi^{\wideparen{n}}\rho^n=1_{Y^{\wideparen{n}}}$ and $\Vert\rho^n\Vert\leq c$. Consider map $ \rho=\bigoplus{}_\infty\{\rho^n:n\in\mathbb{N}\}$. For each $y\in \square_{sqTop}(Y)$ we have 
$$
\Vert \rho(y)\Vert=\sup\{\Vert\rho^n(p_n^Y(y))\Vert_{\wideparen{n}}: n\in\mathbb{N}\}\leq
c\sup\{\Vert p_n^Y(y)\Vert_{\wideparen{n}}: n\in\mathbb{N}\}=c\Vert y\Vert
$$
so $\rho$ is semilinear bounded operator. Moreover, $\square_{sqTop}(\varphi)\rho=1_{\square_{sqTop}(Y)}$, hence $\varphi$ is $\square_{sqTop}$-admissible epimorphism. Conversely, if 
$\varphi$ is $\square_{sqTop}$-admissible epimorphism, then there exist bounded right inverse semilinear operator $\rho$ to $\square_{sqTop}(\varphi)$. Then for every $y\in Y^{\wideparen{n}}$ holds $\square_{sqTop}(\varphi)\\\rho(i_n^Y(y))=i_n^Y(y)$. In particular $\varphi^{\wideparen{n}}(p_n^X(\rho(i_n^Y(y))))=y$. Denote $x=p_n^X(\rho(i_n^Y(y)))$ and 
$c=\Vert\rho\Vert$, then $\varphi^{\wideparen{n}}(x)=y$ and $\Vert x\Vert_{\wideparen{n}}\leq\Vert\rho(i_n^Y(y))\Vert\leq c\Vert i_n^Y(y)\Vert=c\Vert y\Vert_{\wideparen{n}}$. Therefore, 
$\varphi$ is sequentially topologically surjective.
\end{proof}

We are ready to state the main result of this section.

\begin{proposition}\label{PrMetrFrIsTopFr} Let $F$ be metrically free operator sequence space with base $\Lambda$. Then $F$ is topologically free operator sequence space with base $\mathbb{C}^{\Lambda}$.
\end{proposition}
\begin{proof} Let $j':\Lambda\to \square_{sqMet}(F)$ be universal arrow in the diagram of metric freedom of $F$. Define semilinear bounded operator $j: \mathbb{C}^{\Lambda} \to \square_{sqTop}(F): z_\lambda\mapsto z_\lambda j(\lambda)$. Consider arbitrary bounded semilinear operator $\varphi : \mathbb{C}^{\Lambda} \to \square_{sqTop}(X)$, where $X$ is arbitrary operator sequence space. Then for $\varphi':=\Vert \varphi \Vert_{sb}^{-1}\varphi $ there exist unique $\psi^{'}$, such that $\varphi'=\square_{sqMet}(\psi')j$. Now, it is easy to see that for $\psi:=\Vert \varphi \Vert_{sb} \psi^{'}$ the diagram
$$
\xymatrix{
{\square_{sqTop}(F)}\ar@{-->}[dr]^{\square_{sqTop}(\psi)} & \\
{\mathbb{C}^{\Lambda}}\ar[u]_{j}\ar[r]_{\varphi}  &{\square_{sqTop}(X)} }
$$
is commutative. Assume there two morphisms $\psi_1$ and $\psi_2$ making the diagram above commutative. Denote $C=\max( \Vert \varphi\Vert_{sb}, \Vert \psi_1 \Vert_{sb}, \Vert \psi_2\Vert_{sb})$, then morphisms $C^{-1}\psi_1$ and $C^{-1}\psi_2$ make the following diagram
$$
\xymatrix{
{\square_{sqMet}(F)}\ar@{-->}[dr]^{?} & \\
{\mathbb{C}^\Lambda}\ar[u]_{j'}\ar[r]_{ C^{-1}\varphi'}  &{\square_{sqMet}(X)} }
$$
commutative. This contradicts uniqueness of morphism $\psi'$, so $\psi$ is unique.
\end{proof}

As the consequence we get complete description of topologically free operator sequence spaces

\begin{theorem}\label{ThTopFrDesc} 
A operator sequence space is topologically free if and only if it is sequentially topologically isomorphic to $\bigoplus{}_1^0$-sum of copies of the space $t_2^\infty$, indexed by elements of some set $\Lambda$.
\end{theorem}

\begin{corollary}\label{CorSQSpaceIsImgTopAdmEpiMorph}
Every operator sequence space is an image of sequentially topologically surjective operator from $\bigoplus_1^0\{t_2^\infty:\lambda\in\Lambda\}$  for some set $\Lambda$.
\end{corollary}
\begin{proof}
From theorem \ref{ThMetrFrDesc} it follows that the rigged category $(SQNor,\square_{sqTop})$ is freedom-loving. Now the desired result follows  from propositions \ref{PrFrCoFrProjInjObjProp} and \ref{PrDecsTopAdmEpiMorph}
\end{proof}

Similar propositions are valid in Banach case (just replace $\bigoplus{}_1^0$-sums with $\bigoplus{}_1$-sums).

\subsection{Pseudotopological freedom and projectivity}

One may ask, whether existence of uniform constant $C$ in the definition of sequential topological surjectivity is necessary? Indeed more natural definition would require just topological surjectivity of all amplifications of sequentially bounded operator. Later we will see that this class of admissible epimorphisms doesn't give rich homological theory. 

The type of projectivity given by this kind of sequentially bounded admissible epimorphisms we will call pseudotopological. Consider functor
$$
\begin{aligned}
\square_{sqpTop} : SQNor \to Nor_0: X &\mapsto X^{\wideparen{1}}\\
\varphi&\mapsto\varphi\\
\end{aligned}
$$
sending operator sequence space to the underlying semilinear normed space of the first amplification.

\begin{definition}\label{DefPsSQTopSurjOp} A sequentially bounded operator $\varphi:X\to Y$is called pseudotopologically surjective if for every $n\in\mathbb{N}$ there exist $c_n>0$ such that for all $y\in Y^{\wideparen{n}}$ we can find $x\in X^{\wideparen{n}}$ with $\varphi^{\wideparen{n}}(x)=y$ and $\Vert x\Vert_{\wideparen{n}}\leq c_n\Vert y\Vert_{\wideparen{n}}$ 
\end{definition}

\begin{proposition}\label{PrDecsPsTopAdmEpiMorph} Let $\varphi:X\to Y$ be sequentially bounded operator, then the following are equivalent
\newline
1) $\varphi$ $\square_{sqpTop}$-admissible epimorphism
\newline
2) $\varphi$ is pseudotopologically surjective
\newline
3) $\varphi^{\wideparen{1}}$ is topologically surjective
\end{proposition}
\begin{proof}
$1)\implies 2)$ Assume $\varphi$ is $\square_{sqpTop}$-admissible epimorphism, then for some $c_1>0$ and any $y\in Y$ there exist $x\in X$ with $\varphi(x)=y$ and $\Vert x\Vert\leq c_1\Vert y\Vert$. Let $n\in\mathbb{N}$ and $y\in Y^{\wideparen{n}}$, then consider $x\in X^{\wideparen{n}}$, such that $\varphi(x_i)=y_i$ and $\Vert x_i\Vert\leq c_1\Vert y_i\Vert$ for all $i\in\mathbb{N}_n$. Let $e_i\in M_{1,n}$ be row-matrix with $1$ in the $i$-th place and $0$ in others, then 
$$
\Vert x\Vert_{\wideparen{n}}
\leq  \left(\sum\limits_{i=1}^n\Vert x_i\Vert_{\wideparen{1}}^2\right)^{1/2}
\leq  \left(\sum\limits_{i=1}^n c_1^2\Vert y_i\Vert_{\wideparen{1}}^2\right)^{1/2}
\leq c_1\left(\sum\limits_{i=1}^n\Vert e_i y\Vert_{\wideparen{n}}^2\right)^{1/2}
$$
$$
\leq c_1\left(\sum\limits_{i=1}^n\Vert e_i\Vert^2 \Vert y\Vert_{\wideparen{n}}^2\right)^{1/2}
=c_1n^{1/2}\Vert y\Vert_{\wideparen{n}}
$$
Clearly, $\varphi(x)=y$ so $\varphi$ is pseudotopologically surjective.
\newline
$2)\implies 3)$ Obvious.
\newline
$3)\implies 1)$ By proposition \ref{PrCTopSurIsRetrInNor0} there exist bounded semilinear operator $\rho$ such that $\varphi\rho=1_Y$. This means, that $\square_{sqpTop}(\varphi)$ have right inverse, i.e. $\varphi$ is $\square_{sqpTop}$-admissible epimorphism. 
\end{proof}

Consider functors
$$
\begin{aligned}
\square_{sqRel} : SQNor \to Nor: X &\mapsto X^{\wideparen{1}}\\
\varphi&\mapsto\varphi\\
\end{aligned}
\qquad\qquad
\begin{aligned}
\square_{norTop} : Nor \to Nor_0: X &\mapsto X\\
\varphi&\mapsto\varphi\\
\end{aligned}
$$
Note the obvious identity $\square_{sqpTop}=\square_{norTop}\square_{sqRel}$.

\begin{proposition}\label{PrSQRelChar}
In the rigged category $(SQNor,\square_{sqRel})$
\newline
1) $\square_{sqRel}$-free objects are exactly operator sequence spaces sequentially topologically isomorphic to $\max(E)$ for some normed space $E$. This category is freedom-loving. 
\newline
2) Every retract of $\square_{sqRel}$-free object have maximal structure of operator sequence space
\newline
3) each $\square_{sqRel}$-projective object is $\square_{sqRel}$-free.
\end{proposition}
\begin{proof}
1) Let $E\in Nor$. We will show that $\max(E)$ is $\square_{sqRel}$-free object with base $E$. Universal  arrow will be as such $j:E\to\square_{sqRel}(\max(E)):x\mapsto x$. Let $X$ be arbitrary operator sequence space and $\varphi:E\to\square_{sqRel}(X)$ be arbitrary bounded linear operator. Consider linear operator $\psi: \max(E)\to X:x\mapsto\varphi(x)$. From proposition \ref{PrCharMaxSQ} 
it follows, that $\psi$ is sequentially bounded. Clearly, $\varphi=\square_{sqRel}(\psi)j$.
Since $X$ and $\varphi$ are arbitrary, then $\max(E)$ is $\square_{sqRel}$-free object. From proposition \ref{PrUniqFr} it follows that $\square_{sqRel}$-free objects are sequentially topologically isomorphic to 
$\max(E)$. Since $E$ is arbitrary normed space, then the rigged category $(SQNor,\square_{sqRel})$ is freedom-loving.
\newline
2) Let $\sigma:\max(E)\to X$ be a retraction in $SQNor$. Then $\sigma$ is topologically surjective and by proposition \ref{PrMaxPreserveQuotients} $X$ have maximal structure.
\newline
3) Let $P$ be $\square_{sqRel}$-projective object,then from proposition \ref{PrFrCoFrProjInjObjProp} it follows that it is a retract of $\square_{sqRel}$-free object, and from paragraph 2) that $P$ have maximal structure of operator sequence space, i.e. $P=\max(\square_{sqRel}(P))$. Now from paragraph 1) we see that $P$ is $\square_{sqRel}$-free.  
\end{proof}

\begin{proposition}\label{PrNorTopChar}
In the rigged category $(Nor, \square_{norTop})$
\newline
1) $\square_{norTop}$-admissible epimorphisms are exactly topologically surjective operators
\newline
2) $\square_{norTop}$-free objects are normed spaces topologically isomorphic to $l_1^0(\Lambda)$ with base $\mathbb{C}^\Lambda$.
\newline
3) $\square_{norTop}$-projective objects are normed spaces topologically isomorphic to $l_1^0(\Lambda)$ for some set $\Lambda$. 
\end{proposition}
\begin{proof}
1) Follows from proposition \ref{PrCTopSurIsRetrInNor0}/

2) Consider map $j:\mathbb{C}^\Lambda\to l_1^0(\Lambda):z_\lambda\to z_\lambda\delta_\lambda$. For a given semilinear bounded operator $\varphi:\mathbb{C}^\Lambda\to\square_{norTop}(X)$, where $X$ is an arbitrary normed space consider linear operator $\psi:l_1^0(\Lambda)\to X:f\mapsto\sum_{\lambda\in\Lambda}f(\lambda)\varphi(1_\lambda)$. Since $\Vert\psi(f)\Vert\leq\Vert\varphi\Vert\Vert f\Vert$, then $\psi$ is bounded. Moreover it is straightforward to check that $\square_{norTop}(\psi)j=\varphi$. Uniqueness of $\psi$ follows from the chain of equalities 
$$
\psi(f)=\sum_{\lambda\in\Lambda} f(\lambda)\psi(\delta_\lambda)=\sum_{\lambda\in\Lambda} f(\lambda)\square_{norTop}(\psi)(j(1_\lambda))=\sum_{\lambda\in\Lambda} f(\lambda)\varphi(1_\lambda)
$$

3) See \cite{GroTopNorPr} theorem 0.12
\end{proof}

\begin{theorem}\label{ThPsTopFrDesc}
A sequential  operator space is pseudotopologically projective if and only if it is sequentially topologically isomorphic to $\max(l_1^0(\Lambda))$ for some set $\Lambda$.
\end{theorem}
\begin{proof}
From proposition \ref{PrNorTopChar} it follows that $l_1^0(\Lambda)$ is $\square_{norTop}$-free with base  $\mathbb{C}^\Lambda$. From proposition \ref{PrSQRelChar} we get that $\max(l_1^0(\Lambda))$is  $\square_{sqRel}$-free with base $l_1^0(\Lambda)$. Then from proposition \ref{PrCompOfFrIsFr} we see that $\max(l_1^0(\Lambda))$ is $\square_{sqpTop}$-free with base $\mathbb{C}^\Lambda$. Now from proposition \ref{PrUniqFr} we know that all pseudotopologically free objects are of the form $\max(l_1^0(\Lambda))$ for some set $\Lambda$.
\end{proof}

\begin{corollary}\label{CorSQSpaceIsImgPsTopAdmEpiMorph}
Every operator sequence space is an image of topologically surjective operator from $\max(l_1^0(\Lambda))$ for some set $\Lambda$.
\end{corollary}
\begin{proof}	
From theorem \ref{ThPsTopFrDesc} it follows that the rigged category $(SQNor,\square_{sqpTop})$ is freedom-loving. Now the desired result follows from propositions \ref{PrFrCoFrProjInjObjProp} and \ref{PrDecsPsTopAdmEpiMorph}.
\end{proof}

\begin{theorem}\label{ThPsTopProjDesc}
Every pseudotopologically projective operator sequence space is sequentially topologically isomorphic to  $\max(l_1^0(\Lambda))$ for some set $\Lambda$.
\end{theorem}
\begin{proof}
Let $P$ be pseudotopologically projective operator sequence space. From proposition \ref{ThPsTopFrDesc}, \ref{CorSQSpaceIsImgPsTopAdmEpiMorph} we see that there exist $\square_{sqpTop}$-admissible epimorphism $\sigma:\max(l_1^0(\Lambda))\to P$ for some set $\Lambda$. Since $\max(l_1^0(\Lambda))$ is  $\square_{sqpTop}$-free object, then from proposition \ref{PrRetractsProjInj} we get that $\sigma$ is a  retraction in $SQNor$. From paragraph 2) of proposition \ref{PrSQRelChar} we get, the structure of operator sequence space $P$ is maximal, i.e. $P=\max(\square_{sqRel}(P))$. Since $\sigma$ is a retraction in $SQNor$, it is also a retraction in $Nor$ from the space $l_1^0(\Lambda)$. By proposition \ref{PrNorTopChar} the space $l_1^0(\Lambda)$ is $\square_{sqRel}$-free. As $\square_{sqRel}(P)$ is its retract in $Nor$, then from proposition \ref{PrFrCoFrProjInjObjProp} we see that $\square_{sqRel}(P)$ is  $\square_{norTop}$-projective. In this case again from proposition \ref{PrNorTopChar} we conclude, that $\square_{sqRel}(P)$ is topologically isomorphic to $l_1^0(\Lambda')$ for some set $\Lambda'$. Applying $\max$ functor to this isomorphism, we establish sequential topological isomorphism between  $P=\max(\square_{sqRel}(P))$ and $\max(l_1^0(\Lambda'))$.
\end{proof}

Similar propositions are valid in Banach case (just replace $l_1^0$ spaces with $l_1$ space).

\section{Cofree operator sequence spaces}

In what follows we will use the following simple observation 

From propositions \ref{PrDualOfCoprodIsProd} and \ref{PrSQSpaceIsSBFromT2n} it follows, that there exist sequential isometric isomorphisms
$$
(t_2^\infty)^\triangle
=\bigoplus{}_\infty\{(t_2^n)^\triangle:n\in\mathbb{N}\}
=\bigoplus{}_\infty\{l_2^n:n\in\mathbb{N}\}
=l_2^\infty
$$
Therefore applying again proposition \ref{PrDualOfCoprodIsProd} we get a sequential isometric isomorphism:
$$
\left(\bigoplus{}_1^0\{t_2^\infty:\lambda\in\Lambda\}\right)^\triangle
=\bigoplus{}_\infty\{l_2^\infty:\lambda\in\Lambda\}
$$
 
\subsection{Metric cofreedom}

Consider functor
$$
\begin{aligned}
\square_{sqMet}^d : SQNor_1 \to Set^o: X &\mapsto \prod \left\{B_{(X^\triangle )^{\wideparen{n}}}:n\in\mathbb{N}\right\}\\
\varphi&\mapsto\prod\left\{ (\varphi^\triangle )^{\wideparen{n}}|_{B_{(Y^\triangle )^{\wideparen{n}}}}^{B_{(X^\triangle )^{\wideparen{n}}}}:n\in\mathbb{N}\right\}\\
\end{aligned}
$$
\begin{proposition}\label{PrDecsMetrAdmMonoMorph}
$\square_{sqMet}^d$-admissible monomorphisms are exactly sequentially isometric operators.
\end{proposition}
\begin{proof}
A morphism $\varphi$ is $\square_{sqMet}^d$-admissible monomorphism if and only if $\square_{sqMet}^d(\varphi)$ is invertible from the left in $Set^o$. This is equivalent to surjectivity of 
$\square_{sqMet}^d(\varphi^\triangle)$. The latter is equivalent to surjectivity of $(\varphi^\triangle)^{\wideparen{n}}|_{B_{X^{\wideparen{n}}}}^{B_{Y^{\wideparen{n}}}}$ 
for all $n\in\mathbb{N}$. This means that $(\varphi^\triangle)^{\wideparen{n}}$ is strictly coisometric for each $n\in\mathbb{N}$, i.e. $\varphi^\triangle$ is sequentially strictly coisometric. By theorem 
\ref{ThDualSQOps} it is equivalent to $\varphi$ being sequentially isometric.
\end{proof}

\begin{theorem}\label{ThMetCoFrDesc}
Metrically cofree operator sequence space with base $\Lambda$ is up to sequential isometric isomorphism  a $\bigoplus{}_\infty$-sum of copies of the space $l_2^\infty:=\bigoplus{}_\infty\{l_2^n:n\in\mathbb{N}\}$, indexed by elements $\Lambda$.
\end{theorem}
\begin{proof}
Let $\Lambda$ be an arbitrary. Consider commutative diagram
$$
\xymatrix{
SQNor_1^o \ar[d]_{\nabla } \ar[rr]^{(\square_{sqMet}^{d})^o} & & {Set}\ar[d]^{1_{Set}}\\
SQNor_1\ar[rr]^{\square_{sqMet}}&  &{Set}}
$$
Here ${}^\nabla$ is a covariant version of ${}^\triangle$ functor. That diagram is indeed commutative since for any operator sequence spaces $X$, $Y$ and arbitrary $\varphi\in\mathcal{SB}(X,Y)$ holds
$$
1_{Set}((\square_{sqMet}^d)^o(\varphi))
=\prod\limits_{n\in\mathbb{N}} (\varphi^\triangle )^{\wideparen{n}}|_{B_{(Y^\triangle )^{\wideparen{n}}}}^{B_{(X^\triangle )^{\wideparen{n}}}}
=\square_{sqMet}({}^\nabla(\varphi))
$$
From remark \ref{CorSQUnivPropMaxTenProd} we see that ${}^\nabla$ have left adjoint functor, which is  ${}^\triangle$. Analogously $1_{Set}$ is adjoint to itself from the left and from the right. 
By theorem \ref{ThMetrFrDesc} the object $\bigoplus{}_1^0\{t_2^\infty:\lambda\in\Lambda\}$ is  $\square_{sqMet}$-free, so by proposition \ref{PrFunctorMapFrToFr} the object 
$(\bigoplus{}_1^0\{t_2^\infty:\lambda\in\Lambda\})^\triangle=\bigoplus{}_\infty\{l_2^\infty:\lambda\in\Lambda\}$ is $(\square_{sqMet}^d)^o$-free, which is the same as being $\square_{sqMet}^d$-cofree. Since the set $\Lambda$ is arbitrary, using proposition \ref{PrUniqFr} we get that all $\square_{sqMet}$-cofree objects with base $\Lambda$ are sequentially isometrically isomorphic to the space constructed above.
\end{proof}

\begin{corollary}\label{CorSQSpaceIsFromMetrAdmMonoMorph}
From every operator sequence space there exist a sequentially isometic operator into $\bigoplus_\infty\{l_2^\infty:\lambda\in\Lambda\}$  for some set $\Lambda$.
\end{corollary}
\begin{proof}
From theorem \ref{ThMetrFrDesc} it follows that the rigged category $(SQNor_1,\square_{sqMet}^d)$ is cofreedom-loving. Now the desired result from propositions \ref{PrFrCoFrProjInjObjProp} and \ref{PrDecsMetrAdmMonoMorph}.
\end{proof}

\begin{proposition}\label{PrCharacDualSQSp} An operator sequence space $X$ is a dual operator sequence space if and there is sequentially isometric weak${}^*$-weak${}^*$ homeomorphism onto weak${}^*$ closed subspace of $\bigoplus_\infty\{l_2^\infty:\lambda\in\Lambda\}$  for some set $\Lambda$.
\end{proposition}
\begin{proof}
Assume $X$ is a dual operator sequence space with sequential predual $X_\triangle$. By proposition \ref{CorSQSpaceIsImgMetrAdmEpiMorph} for some set $\Lambda$ we have sequentially coisometric operator $\pi:\bigoplus_1^0\{t_2^\infty:\lambda\in\Lambda\}\to X_\triangle$. By theorem \ref{ThDualSQOps} operator $\pi^\triangle$ is a sequential isometry from $X_\triangle^\triangle=X$ into $(\bigoplus_1^0\{t_2^\infty:\lambda\in\Lambda\})^\triangle=\bigoplus_\infty\{l_2^\infty:\lambda\in\Lambda\}$. By lemma A.2.5 \cite{BleOpAlgAndMods} operator $\pi^\triangle$ is weak${}^*$-weak${}^*$ homeomorphism onto its weak${}^*$ closed image.

Conversely, if $X$ is a weak${}^*$ closed subspace of $Y:=\bigoplus_\infty\{l_2^\infty:\lambda\in\Lambda\}$  for some set $\Lambda$, then by proposition \ref{PrDualForWStarClQuotsAndSubsp} we have $X=(Y/X_\perp)^\triangle$. Hence $X$ is dual operator sequence space with sequential predual $X_\triangle:=Y/X_\perp$.
\end{proof}

Similar propositions are valid in Banach case.

\subsection{Topological cofreedom}

Consider functor
$$
\begin{aligned}
\square_{sqTop}^d : SQNor \to Nor_0^o, X &\mapsto \bigoplus{}_\infty \{(X^{\triangle })^{\wideparen{n}} : n \in \mathbb{N}\}\\
\varphi&\mapsto\bigoplus{}_\infty \{(\varphi^\triangle )^{\wideparen{n}} : n \in \mathbb{N}\}
\end{aligned}
$$

\begin{proposition}\label{PrDecsTopAdmMonoMorph}
$\square_{sqTop}^d$-admissible monomorphisms are exactly sequentially topologically injective operators.
\end{proposition}
\begin{proof}
A morphism $\varphi$ is a $\square_{sqTop}^d$-admissible monomorphism if and only if $\square_{sqTop}^d(\varphi)$ is invertible as morphism in $Nor_0^o$. This is equivalent to say that $\square_{sqTop}^d(\varphi)=\square_{sqTop}(\varphi^\triangle)$ is invertible from the right as morphism in $Nor_0$. From proposition \ref{PrDecsTopAdmEpiMorph} this is equivalent to sequential topological surjectivity of $\varphi^\triangle$. 
By theorem \ref{ThDualSQOps} this is equivalent to $\varphi$ being sequentially topologically injective.
\end{proof}

\begin{theorem}\label{ThTopCoFrDesc}
A operator sequence space is topologically cofree if and on;y if it is sequentially topologically isomorphic to $\bigoplus{}_\infty$ sum of copies of the space $l_2^\infty$ indexed by elements of some set $\Lambda$.
\end{theorem}
\begin{proof}
Let $\Lambda$ be an arbitrary set. Consider commutative diagram
$$
\xymatrix{
SQNor^o \ar[d]_{\nabla } \ar[rr]^{(\square_{sqTop}^d)^o} & & {Nor_0} \ar[d]^{1_{Nor_0}}\\
SQNor\ar[rr]^{\square_{sqTop}} & & {Nor_0}}
$$
Here $\nabla$ is a covariant version of $\triangle$ functor.
This diagram is commutative since for any operator sequence spaces$X$, $Y$ and arbitrary $\varphi\in\mathcal{SB}(X,Y)$ holds
$$
1_{Nor_0}((\square_{sqTop}^d)^o(\varphi))
=\bigoplus{}_\infty \{(\varphi^\triangle )^{\wideparen{n}} : n \in \mathbb{N}\}
=\square_{sqTop}({}^\nabla(\varphi))
$$
From remark \ref{CorSQUnivPropMaxTenProd} we see that ${}^\nabla$ have left adjoint functor,which is ${}^\triangle$. Analogously $1_{Nor_0}$ is adjoint to itself from the left and from the right. By theorem \ref{ThTopFrDesc} the object $\bigoplus{}_1^0\{t_2^\infty:\lambda\in\Lambda\}$ is $\square_{sqTop}$-free, so by proposition \ref{PrFunctorMapFrToFr} the object 
$(\bigoplus{}_1^0\{t_2^\infty:\lambda\in\Lambda\})^\triangle=\bigoplus{}_\infty\{l_2^\infty:\lambda\in\Lambda\}$ is $(\square_{sqTop}^d)^o$-free which is the same as being $\square_{sqTop}^d$-cofree. Using proposition \ref{PrUniqFr} we get, that all $\square_{sqTop}$-cofree objects with base 
$\mathbb{C}^\Lambda$ are sequentially topologically isomorphic to the space constructed above.
\end{proof}

\begin{corollary}\label{PrMetrCoFrIsTopFr}
Every metrically cofree operator sequence space is topologically cofree.
\end{corollary}

\begin{corollary}\label{CorSQSpaceIsFromTopAdmMonoMorph}
From every operator sequence space there exist sequentially topologically injective operator into $\bigoplus_\infty\{l_2^\infty:\lambda\in\Lambda\}$  for some set $\Lambda$.
\end{corollary}
\begin{proof} From theorem \ref{ThTopCoFrDesc} it follows that the rigged category $(SQNor,\square_{sqTop}^d)$ is cofreedom-loving. Now the desired result follows from propositions \ref{PrFrCoFrProjInjObjProp} and \ref{PrDecsTopAdmMonoMorph}.
\end{proof} 

Similar propositions are valid in Banach case.

\subsection{Pseudotopological cofreedom and injectivity}

Consider functor
$$
\begin{aligned}
\square_{sqpTop}^d : SQNor \rightarrow Nor_0^o, X &\mapsto (X^\triangle)^{\wideparen{1}}\\
\varphi&\mapsto\varphi^\triangle
\end{aligned}
$$
sending operator sequence space to the underlying semilinear normed space of the first amplification of its sequential dual, and morphism is mapped to its adjoint  considered as bounded semilinear operator.

\begin{definition}\label{DefPsSQTopInjOp} A sequentially bounded operator $\varphi:X\to Y$ is called pseudotopologically injective,if for every $n\in\mathbb{N}$ 
there exist $c_n>0$ such that for all $x\in X^{\wideparen{n}}$ holds $c_n\Vert\varphi^{\wideparen{n}}(x)\Vert_{\wideparen{n}}\geq \Vert x\Vert_{\wideparen{n}}$
\end{definition}

\begin{proposition}\label{PrDecsPsTopAdmMonoMorph}
Let $\varphi:X\to Y$ be sequentially bounded operator between operator sequence spaces, then the following are equivalent
\newline
1) $\varphi$ is $\square_{sqpTop}$-admissible monomorphism
\newline
2) $\varphi$ is pseudotopologically injective
\newline
3) $\varphi^{\wideparen{1}}$ is topologically injective
\end{proposition}
\begin{proof}
$1)\implies 2)$ Let $\varphi$ be $\square_{sqpTop}^d$-admissible monomorphism. Then $\square_{sqpTop}^d(\varphi)$ is invertible from the left as morphism in $Nor_0^o$. This is equivalent to say that
$\square_{sqpTop}^d(\varphi)=\square_{sqpTop}(\varphi^\triangle)$ is invertible from the right as morphism in $Nor_0$. From proposition \ref{PrDecsPsTopAdmEpiMorph} it is equivalent to pseudotopological surjectivity of $\varphi^\triangle$. 
By proposition \ref{PrDualOps} this is equivalent to $\varphi$ being pseudotopologically injective.
\newline
$2)\implies 3)$ Obvious.
\newline
$3)\implies 1)$ Let $\varphi^{\wideparen{1}}$ be topologically injective then by proposition \ref{PrDualOps} $(\varphi^\triangle)^{\wideparen{1}}$ is topologically surjective. From proposition \ref{PrDecsPsTopAdmEpiMorph} $\varphi^\triangle$ is $\square_{sqpTop}$-admissible epimorphism, i.e. $\square_{sqpTop}(\varphi^\triangle)$ is invertible from the right as morphism in $Nor_0$. Hence $\square_{sqpTop}^d(\varphi)=\square_{sqpTop}(\varphi^\triangle)$ is invertible from the left as morphism in $Nor_0^o$. Hence $\varphi$  is $\square_{sqpTop}^d$-admissible monomorphism. 
\end{proof}

Consider functors
$$
\begin{aligned}
\square_{sqRel}^d : SQNor \to Nor^o: X &\mapsto (X^\triangle)^{\wideparen{1}}\\
\varphi&\mapsto\varphi\\
\end{aligned}
\qquad\qquad
\begin{aligned}
\square_{norTop}^d : Nor^o \to Nor_0^o: X &\mapsto X\\
\varphi&\mapsto\varphi\\
\end{aligned}
$$
Note the obvious identity $\square_{sqpTop}^d=\square_{norTop}^d\square_{sqRel}^d$.

\begin{proposition}\label{PrSQReldChar}
In the rigged category $(SQNor,\square_{sqRel}^d)$
\newline
1) $\square_{sqRel}^d$-cofree objects are exactly operator sequence spaces sequentially topologically isomorphic to $\min(E^*)$ for some normed space $E$. This category is cofreedom-loving. 
\newline
2) Every retract of $\square_{sqRel}^d$-cofree object have minimal structure of operator sequence space.
\end{proposition}
\begin{proof}
1) Let $E\in Nor$. Consider commutative diagram
$$
\xymatrix{
SQNor^o \ar[d]_{\nabla } \ar[rr]^{(\square_{sqRel}^{d})^o} & & {Nor}\ar[d]^{1_{Nor}}\\
SQNor\ar[rr]^{\square_{sqRel}}&  &{Nor}}
$$
Here ${}^\nabla$ is a covariant version of ${}^\triangle$ functor.
This diagram is commutative since for any operator sequence spaces $X$, $Y$ and arbitrary $\varphi\in\mathcal{SB}(X,Y)$ holds
$$
1_{Nor}((\square_{sqRel}^d)^o(\varphi))
=\varphi^\triangle
=\square_{sqRel}({}^\nabla(\varphi))
$$
From remark \ref{CorSQUnivPropMaxTenProd} we see that ${}^\nabla$ have left adjoint functor, which is ${}^\triangle$. Analogously $1_{Nor}$ is adjoint to itself from the left and from the right. By proposition \ref{PrSQRelChar} the object $\max(E)$ is $\square_{sqRel}$-free, so from propositions \ref{PrFunctorMapFrToFr}, \ref{PrDualityAndMinMax} the object 
$(\max(E))^\triangle=\min(E^*)$ is $(\square_{sqRel}^d)^o$-free, which is the same as being $\square_{sqRel}^d$-cofree. 
Since the space $E$ is arbitrary, using proposition \ref{PrUniqFr} we get that all $\square_{sqRel}^d$-cofree objects with base $E$ are sequentially topologically isomorphic to the space constructed above. As the consequence the rigged category $(SQNor,\square_{sqRel}^d)$ is cofreedom-loving.
\newline
2) Let $\sigma:\min(E^*)\to X$ be a retraction in $SQNor$. Right inverse of $\sigma$ we will denote by $\rho$. Since $\rho$ is topologically injective, then by proposition \ref{PrMinPreserveEmbedings} we see that $X$ have minimal struture.
\end{proof}

\begin{proposition}\label{PrNorTopdChar}
In the rigged category $(Nor^o, \square_{norTop}^d)$
\newline
1) $\square_{norTop}^d$-admissible monomorphisms are exactly topologically surjective operators
\newline
2) $\square_{norTop}^d$-cofree objects are exactly normed spaces topologically isomorphic to $l_1^0(\Lambda)$ with base $\mathbb{C}^\Lambda$.
\newline
3) $\square_{norTop}^d$-injective objects are normed spaces topologically isomorphic to  $l_1^0(\Lambda)$ for some set $\Lambda$. 
\end{proposition}
\begin{proof}
All results follow from proposition \ref{PrNorTopdChar} if one note that $\square_{norTop}^d=\square_{norTop}^o$.
\end{proof}

\begin{theorem}\label{ThPsTopCoFrDesc} 
A operator sequence space is pseudotopologically cofree if and only if it is sequentially topologically isomorphic to $\min(l_\infty(\Lambda))$ for some set $\Lambda$.
\end{theorem}
\begin{proof}
From proposition \ref{PrNorTopdChar} it follows that $l_1^0(\Lambda)$ is $\square_{norTop}^d$-cofree with base $\mathbb{C}^\Lambda$. From proposition \ref{PrSQReldChar} we get that $\min((l_1^0(\Lambda))^*)=\min(l_\infty(\Lambda))$ is  $\square_{sqRel}^d$-cofree with base 
$l_1^0(\Lambda)$. Then from proposition \ref{PrCompOfFrIsFr} we see that  $\min(l_\infty(\Lambda))$ is $\square_{sqpTop}^d$-cofree with base $\mathbb{C}^\Lambda$. Now from proposition \ref{PrUniqFr} we know that all pseudotopologically cofree objects are of the form $\min(l_\infty(\Lambda))$ for some set $\Lambda$.
\end{proof}

\begin{corollary}\label{CorSQSpaceIsFromPsTopAdmMonoMorph}
From every operator sequence space there exist topologically injective operator into $\min(l_\infty(\Lambda))$ for some set $\Lambda$.
\end{corollary}
\begin{proof}
From theorem \ref{ThPsTopCoFrDesc} it follows that the rigged category $(SQNor,\square_{sqTop}^d)$ is cofreedom-loving. Now the desired result  follows from propositions \ref{PrFrCoFrProjInjObjProp} and \ref{PrDecsPsTopAdmMonoMorph}.
\end{proof}

\begin{theorem}\label{ThPsTopInjDesc} 
Every pseudotopologically injective operator sequence space is sequentially topologically isomorphic to $\min(F)$, where $F$ is a retract in $Nor$ of the space $l_\infty(\Lambda)$ for some set $\Lambda$.
\end{theorem}
\begin{proof}
Let $I$ be pseudotopologically injective operator sequence space. 
From propositions \ref{ThPsTopCoFrDesc}, \ref{CorSQSpaceIsFromPsTopAdmMonoMorph} we see that there exist $\square_{sqpTop}^d$-admissible monomorphism $\sigma:I\to\min(l_\infty(\Lambda))$ for some set $\Lambda$. Since $\min(l_\infty(\Lambda))$ is $\square_{sqpTop}^d$-cofree object, then from proposition \ref{PrRetractsProjInj} it follows that $\sigma$ is a coretraction in $SQNor$. Let $\rho$ be right inverse morphism of $\sigma$ in $SQNor$. It is a retraction in $SQNor$, then from paragraph 2) of proposition \ref{PrSQReldChar} we get that the structure of operator sequence space $I$ is minimal, i.e. $I=\min(\square_{sqRel}(I))$. Since $\rho$ is a retraction in $SQNor$, it is retraction in $Nor$ from $l_\infty(\Lambda)$ to $F:=\square_{sqRel}(I)$.
\end{proof}

Similar propositions are valid in Banach case.

\begin{remark} Unfortunately, in theorem \ref{ThPsTopInjDesc} by analogy with theorem  \ref{ThPsTopProjDesc} we can't state that retracts of $l_\infty(\Lambda)$ are of the form $l_\infty(\Lambda')$ for some set $\Lambda'$. Indeed, in \cite{RosInjLmuSp} corollary 4.4 it was shown existence of topologically injective space $F$ which can't be topologically isomorphic to any dual space and in particular to $l_\infty(\Lambda')$ for any set $\Lambda'$. On the other hand for some set $\Lambda$ there exist isometric embedding $i:F\to l_\infty(\Lambda)$. Since $F$ is topologically injective, then by proposition \ref{PrRetractsProjInj} this is a coretraction. Therefore, $F$ is a retract of $l_\infty(\Lambda)$, which can not be topologically isomorphic to $l_\infty(\Lambda')$ for any set $\Lambda'$.
\end{remark}

\end{document}